\documentclass[11pt]{article}

\usepackage{times}
\usepackage{amsthm}
\usepackage{amsmath}
\usepackage{amssymb}
\usepackage{mathrsfs}
\usepackage{pb-diagram}
\usepackage{color}

\usepackage{hyperref}

\renewcommand\eqref[1]{(\ref{#1})} %Need with hyperref

\def\A{{\mathcal A}}
\def\o{\omega}
\def\O{\Omega}
\def\th{\theta}

\def\N{\mathbb{N}}
\def\T{\mathbb{T}}

\def\V{\mathcal V}

\def\B{{\mathbb B}}

\def\H{\mathcal H}
\def\K{\mathcal K}
\def\R{\mathbb{R}}

\def\CC{{\mathfrak C}}

\def\W{{\mathcal W}}
\def\di{\diamond}

\def\e{{\sf e}}
\def\g{{\mathfrak g}}
\def\m{{\sf m}}
\def\wm{\widehat{\sf m}}

\def\({\left(}
\def\[{\left[}
\def\){\right)}
\def\]{\right]}

\def\si{\sigma}

\def\G{{\sf G}}
\def\wG{\widehat{\sf{G}}}
\def\p{\parallel}
\def\<{\langle}
\def\>{\rangle}

\providecommand{\CC}{\mathfrak{C}}

\def\fscr{\mathscr}

\newtheorem{Theorem}{Theorem}[section]
\newtheorem{Remark}[Theorem]{Remark}
\newtheorem{Lemma}[Theorem]{Lemma}

\newtheorem{Corollary}[Theorem]{Corollary}
\newtheorem{Proposition}[Theorem]{Proposition}
\newtheorem{Definition}[Theorem]{Definition}
\newtheorem{Example}[Theorem]{Example}

\numberwithin{equation}{section}

\begin{document}

%-------------------------------------------------------------------------------------------------------
% Title
%-------------------------------------------------------------------------------------------------------

%\title{Pseudo-differential Operators on\\ Type I Locally Compact Groups}
\title{Pseudo-differential operators, Wigner transform\\ and Weyl systems on type I locally compact groups}

\date{\today}

\author{M. M\u antoiu and M. Ruzhansky 
}
\date{\small}
\maketitle \vspace{-1cm}

%-------------------------------------------------------------------------------------------------------
% Abstract
%-------------------------------------------------------------------------------------------------------

\begin{abstract}
Let $\G$ be a unimodular type I second countable locally compact group and $\wG$ its unitary dual. We introduce and study a global pseudo-differential calculus for operator-valued symbols defined on $\G\times\wG$\,, and its relations to suitably defined Wigner transforms and Weyl systems. We also unveil its connections with crossed products $C^*$-algebras associated to certain $C^*$-dynamical systems, and apply it to the spectral analysis of covariant families of operators. Applications are given to nilpotent Lie groups, in which case we relate quantizations with operator-valued and scalar-valued symbols. 
\end{abstract}

\tableofcontents

%-------------------------------------------------------------------------------------------------------
\section{Introduction}\label{duci}
%-------------------------------------------------------------------------------------------------------

Let $\G$ be a locally compact group with unitary dual $\wG$\,, composed of classes of unitary equivalence of strongly continuous irreducible representations. To have a manageable Fourier transformation, it will be assumed second countable, unimodular and postliminal (type I). The formula 
\begin{equation}\label{fred}
\[{\sf Op}(a)u\]\!(x)=\int_\G\Big(\int_{\wG}{\rm Tr}_\xi\!\[\xi(y^{-1}x)a(x,\xi)\]\!d\wm(\xi)\Big)u(y)d\m(y)
\end{equation}
is our starting point for a global pseudo-differential calculus on $\G$\,; it involves operator-valued symbols defined on $\G\times\wG$\,.  In \eqref{fred} $d\m$ is the Haar measure of the group $\G$\,, $d\wm$ is the Plancherel measure on the space $\wG$ and for the pair  $(x,\xi)$ formed of an element $x$ of the group and a unitary irreducible representation $\xi:\G\rightarrow\mathbb B(\H_\xi)$\,, the symbol $a(x,\xi)$ is essentially assumed to be a trace-class operator in the representation Hilbert space $\H_\xi$\,.
In further extensions of the theory it is important to also include densely defined symbols to cover, for example, 
differential operators on Lie groups (in which case one can make sense of  \eqref{fred} for such $a(x,\xi)$ by letting it act on the dense in $\H_\xi$ subspace of smooth vectors of the representation $\xi$\,, see \cite{FR1}).

\medskip
Particular cases of \eqref{fred} have been previously initiated in \cite{RT,RT1} and then intensively developed further in \cite{DaR, DaR1,DR1,Fi,RTW,RW} for compact Lie groups, and in \cite{FR,FR1,FR2} for large classes of nilpotent Lie groups (graded Lie groups), as far-reaching versions of the usual Kohn-Nirenberg quantization on $\G=\mathbb R^n$\,, cf. \cite{Fo}\,. The idea to use the irreducible representation theory of a type I group in defining pseudo-differential operators seems to originate in \cite[Sect. I.2]{Ta}, but it has not been developed before in such a generality. All the articles cited above already contain historical discussions and references to the literature treating pseudo-differential operators (quantization) in group-like situations, so we are not going to try to put this subject in a larger perspective. 

\medskip
Let us just say that an approach involving pseudo-differential operators with representation-theoretical operator-valued symbols has the important privilege of being global. On most of the smooth manifolds there is no notion of full scalar-valued symbol for a pseudo-differential operator defined using local coordinates. This is unfortunately true even in the rather simple case of a compact Lie group, for which the local theory, only leading to a principal symbol, has been shown to be equivalent to the global operator-valued one (cf. \cite{RT,RTW}).  On the other hand, in the present article we are not going to rely on compactness, on the nice properties implied by nilpotency, not even on the smooth structure of a Lie group. In the category of type I second countable locally compact groups one has a good integration theory on $\G$ and a manageable integration theory on $\wG$\,, allowing a general form of the Plancherel theorem, and this is enough to develop the basic features of a quantization. Unimodularity has been assumed, for simplicity, but by using tools from \cite{DM} it might be possible to develop the theory without it.

\medskip
More structured cases (still more general than those studied before) will hopefully be analysed in the close future, having the present paper as a framework and a starting point. In particular, classes of symbols of H\"ormander type would need more than a smooth structure on $\G$\,. The smooth theory, still to be developed, seems technically difficult if the class of Lie groups is kept very general. Of course, only in this setting one could hope to cover differential operators and certain types of connected applications. On the other hand, the setting of our article allows studying multiplication and invariant operators as very particular cases, cf. Subsection \ref{fourtinitin}.

\medskip
The formula (\ref{fred}) is a generalisation of the Kohn-Nirenberg quantization rule for the particular case $\G=\mathbb R^n$. But for $\mathbb R^n$ there are also the so-called $\tau$-quantizations
\begin{equation*}\label{tao}
\[{\sf Op}^\tau\!(a)u\]\!(x)=\int_{\R^n}\!\Big(\int_{\R^n}a\big((1-\tau)x+\tau y,\eta\big)e^{i(x-y)\eta}d\eta\Big)u(y)dy\,,
\end{equation*} 
related to ordering issues, with $\tau\in[0,1]$\,, and the Kohn-Nirenberg quantization is its special case for $\tau=0$.
It is possible to provide extensions of the pseudo-differential calculus on type I groups corresponding to any measurable function $\tau:\G\rightarrow\G$\,. The general formula turns out to be
\begin{equation}\label{bilfred}
\[{\sf Op}^\tau\!(a)u\]\!(x)=\int_\G\!\Big(\int_{\wG}\,{\rm Tr}_\xi\Big[\xi(y^{-1}x)a\big(x\tau(y^{-1}x)^{-1}\!,\xi\big)\Big]d\wm(\xi)\Big)u(y)d\m(y)\,,
\end{equation}
from which (\ref{fred}) can be recovered putting $\tau(x)=\e$ (the identity) for every $x\in\G$\,. 
This formula and its integral version will be summarised in \eqref{bilfred3}.
The case $\tau(x)=x$ is also related to a standard choice 
\begin{equation}\label{id}
\[{\sf Op}^{{\sf id}_\G}(a)u\]\!(x)=\int_\G\!\Big(\int_{\wG}{\rm Tr}_\xi\big[\xi(y^{-1}x)a(y,\xi)\big]d\wm(\xi)\Big)u(y)\,d\m(y)\,,
\end{equation}
familiar at least in the case $\G=\R^n$ (derivatives to the left, positions to the right).
In the presence of $\tau$ some formulae are rather involved, but the reader can take the basic case $\tau(\cdot)=\e$ as the main example. Anyhow, for the function spaces we consider in this paper, the formalisms corresponding to different mappings $\tau$ are actually isomorphic. Having in mind the Weyl quantization for $\G=\R^n$\, we deal in Section \ref{fourtinin} with the problem of a symmetric quantization, for which one has ${\sf Op}^\tau\!(a)^*={\sf Op}^\tau\!(a^\star)$\,, where $a^\star$ is an operator version of complex conjugation. 
We also note that if the symbol $a(x,\xi)=a(\xi)$ is independent of $x$, the operator
${\sf Op}^\tau\!(a)$ is left-invariant and independent of $\tau$, and can be rewritten in the form of the Fourier multiplier
\begin{equation}\label{EQ:fm1}
{\fscr F}\[{\sf Op}^\tau\!(a)u\](\xi)=a(\xi){\fscr F}\[u\](\xi),\quad \xi\in \wG\,,
\end{equation}
at least for sufficiently well-behaved functions $u$, i.e. as an operator of ``multiplication'' of the operator-valued 
Fourier coefficients from the left.

\medskip
One of our purposes is to sketch two justifications of formula \eqref{bilfred}, which both hold without a Lie structure on $\G$ (we refer to \cite{BBM,MPR1} to similar strategies in quite different situations). They also enrich the formalism and have certain applications, some of them included here, others subject of subsequent developments.  Let us say some words about the two approaches.

\medskip
1. A locally compact group $\G$ being given, we have a canonical action by (left) translations on various $C^*$-algebras of functions on $\G$\,. There are crossed product constructions associated to such situations, presented in Section \ref{fourtin}: One gets $^*$-algebras of scalar-valued functions on $\G\times\G$ involving a product which is a convolution in one variable and a pointwise multiplication in the other variable, suitably twisted by the action by translations. A $C^*$-norm with an operator flavour is also available, with respect to which one takes a completion. Since we have to accommodate the parameter $\tau$, we were forced to outline an extended version of crossed products. 

Among the representations of these $C^*$-algebras there is a distinguished one presented and used in Subsection \ref{firea}, the Schr\"odinger representation, in the Hilbert space $L^2(\G)$\,. If $\G$ is type I, second countable and unimodular, there is a nicely-behaved Fourier transform sending functions on $\G$ into operator-valued sections defined over $\wG$\,. This can be augmented to a partial Fourier transform sending functions on $\G\times\G$ into sections over $\G\times\wG$\,. Starting from the crossed products, this partial Fourier transform serves to define, by transport of structure, $^*$-algebras of symbols with a multiplication generalising the Weyl-Moyal calculus as well as Hilbert space representations  of the form (\ref{bilfred}). They are shown to be generated by products of suitable multiplication and convolution operators.

The $C^*$-background can be used, in a slightly more general context, to generate covariant families of pseudo-differential operators, cf. Subsection \ref{midol}.  It also leads to results about the spectrum of certain bounded or unbounded pseudo-differential operators, as it is presented in Subsection \ref{midol} and will be continued in a subsequent paper.

\medskip
2. A second approach relies on Weyl systems. If $\G=\mathbb R^n$ one can write 
$$
{\sf Op}(a)=\int_{\mathbb R^{2n}}\!\widehat a(\xi,x)W(\xi,x)\,dxd\xi\,,
$$ 
where the Weyl system (phase-space shifts) 
$$
\big\{W(\xi,x):=V(\xi)U(x)\mid(x,\xi)\in\mathbb R^{2n}\big\}
$$ 
is a family of unitary operators in $L^2(\mathbb R^n)$ obtained by putting together translations and modulations. This is inspired by the Fourier inversion formula, but notice that $W$ is only a projective representation; this is a precise way to codify the canonical commutation relations between positions $Q$ (generating $V$) and momenta $P$ (generating $U$) and ${\sf Op}$ can be seen as a non-commutative functional calculus $a\mapsto a(Q,P)\equiv{\sf Op}(a)$\,. Besides its phase-space quantum mechanical interest, this point of view also opens the way to some new topics or tools such as the Bargmann transform, coherent states, the anti-Wick quantization, coorbit spaces, etc. 

In Section \ref{fourtein} we show that such a ``Weyl system approach" and its consequences are also available in the context of second countable, unimodular type I groups; in particular it leads to (\ref{bilfred}). 
The Weyl system in this general case, adapted in Definitions \ref{siegmund} and \ref{andrei} to the existence of the quantization parameter $\tau$\,, has nice technical properties (including a fibered form of square integrability) that are proven in Subsection \ref{tintin}. This has useful consequences at the level of the quantization process, as shown in Subsection \ref{teletin}. In particular, it is shown that ${\sf Op}^\tau$ is a unitary map from a suitable class of square integrable sections over $\G\times\wG$ to the Hilbert space of all Hilbert-Schmidt operators on $L^2(\G)$\,. The intrinsic $^*$-algebraic structure on the level of symbols is briefly treated in Subsection \ref{fraterin}. In Subsection \ref{tontolin} we rely on complex interpolation and non-commutative $L^p$-spaces to put into evidence certain families of Schatten-class operators.

\medskip
Without assuming that $\G$ is a Lie group we do not have the usual space of smooth compactly supported functions
readily available as the standard space of test functions. So, in Section \ref{fergulin}, we will be using its generalisation to the setting of
locally compact groups by Bruhat \cite{Br}, and these Bruhat spaces $\mathcal D(\G)$ and $\mathcal D'(\G)$ will
replace the usual spaces of test functions and distributions in our setting. An important fact is that they are nuclear. Taking suitable tensor products one also gets a space $\mathscr D(\G\times\wG)$ of regularising symbols and (by duality) a space $\mathscr D'(\G\times\wG)$ of ``distributions", allowing to define unbounded pseudo-differential operators.

In Subsection \ref{fifitan} we show that pseudo-differential operators with regularising operator-valued symbols can be used to describe compactness of families of vectors or operators in $L^2(\G)$\,.

\medskip
Besides the usual ordering issue (derivatives to the left or to the right), already appearing for $\mathbb R^n$ and connected to the Heisenberg commutation relations and the symplectic structure of phase space, for general groups there is a second ordering problem coming from the intrinsic non-commutativity of $\G$. The Weyl system used in Section \ref{fourtein} relies on translations to the right, aiming at a good correspondence with the previously studied compact and nilpotent cases. Another Weyl system, involving left translations, is introduced in Section \ref{fourtamin} and used in defining a left quantization. It turns out that this one is directly linked to crossed product $C^*$-algebras.

\medskip
We dedicated the last section to a brief overview of quantization on (connected, simply connected) nilpotent Lie groups. Certain subclasses have been thoroughly examined in references cited above, so we are going to concentrate on some new features. Besides the extra generality of the present setting (non-graded nilpotent groups, $\tau$-quantizations, $C^*$-algebras), we are also interested in the presence of a second formalism, involving scalar-valued symbols. We show that it is equivalent to the one involving operator-valued symbols, emerging as a particular case of the previous sections. This is a rather direct consequence of the excellent behaviour of the exponential function in the nilpotent case. On one hand, the analysis in this paper here outlines a $\tau$-extension of the scalar-valued calculus on nilpotent Lie groups initiated by Melin \cite{Melin}, see also \cite{Glo1, Glo2} for further developments on
homogeneous and general nilpotent Lie groups. On the other hand, it relates this to the operator-valued calculus 
developed in \cite{FR, FR1}.

After some basic constructions involving various types of Fourier transformations are outlined, the detailed development of the pseudo-differential operators with scalar-valued symbol follows along the lines already indicated. So, to save space and avoid repetitions, we will be rather formal and sketchy and leave many details to the reader. Actually the Lie structure of a nilpotent group permits a deeper investigation that was treated in \cite{FR1} and should be still subject of future research. 

\medskip
Thus, to summarise, the main results of this paper are as follows:

\begin{itemize}
\item We develop a rigorous framework for the analysis of pseudo-differential operators on locally compact groups of type I, which we assume also unimodular for technical simplicity.
\item We introduce notions of Wigner and Fourier-Wigner transforms, and of Weyl systems, adap\-ted to this general setting. These notions are used to define and analyse $\tau$-quantizations (or quantization by Weyl systems) of operators modelled on families of quantizations on $\mathbb R^n$ that include the Kohn-Nirenberg and Weyl quantizations.
\item We develop the $C^*$-algebraic formalism to put $\tau$-quantizations in a more general perspective, also allowing analysis of operators with `coefficients' taking values in different $C^*$-algebras. The link with a left form of $\tau$-quantization is given via a special covariant representation, the Schr\"odinger representation. This is further applied to investigate spectral properties of covariant families of operators.
\item Although the initial analysis is set for operators bounded on $L^2(\G)$, this can be extended further to include densely defined operators and, more generally, operators from $\mathcal D(\G)$ to $\mathcal D'(\G)$. Since $\G$ does not have to be a Lie group (i.e. there may be no compatible smooth differential structure on $\G$) we show how this can be done using the so-called Bruhat space $\mathcal D(\G)$, an analogue of the space of smooth compactly supported functions in the setting of general locally compact groups.
\item The results are applied to a deeper analysis of $\tau$-quantizations on nilpotent Lie groups. On one hand, this extends the setting of graded Lie groups developed in depth in \cite{FR, FR1} to more general nilpotent Lie groups, also introducing a possibility for Weyl-type quantizations there. On the other hand, it extends the invariant Melin calculus \cite{Melin, Glo2} on homogeneous groups to general non-invariant operators with the corresponding  $\tau$-versions of scalar-symbols on the dual of the Lie algebra;
\item We give a criterion for the existence of Weyl-type quantizations in our framework, namely, to quantizations in which real-valued symbols correspond to self-adjoint operators. We show the existence of such quantizations in several settings, most interestingly in the setting of general groups of exponential type.
\end{itemize}

In this paper we are mostly interested in symbolic understanding of pseudo-differential operators. Approaches through
kernels exist as well, see e.g. Meladze and Shubin \cite{MS} and further works by these authors on
operators on unimodular Lie groups, or Christ, Geller, G\l owacki and Polin \cite{CGGP} on homogeneous groups -- but see
also an alternative (and earlier) symbolic approach to that on the Heisenberg group by Taylor \cite{Ta}.

\bigskip
{\bf Acknowledgements:} MM was supported by N\'ucleo Milenio de F\'isica Matem\'atica RC120002 and the Fondecyt Project 1120300.
MR was supported by the EPSRC Grant EP/K039407/1 and by the Leverhulme Research Grant RPG-2014-02.
The authors were also partly supported by EPSRC Mathematics Platform grant EP/I019111/1.

%---------------------------------------------------------------------------------------------------------
\section{Framework}\label{bogart}
%-----------------------------------------------------------------------------------------------------------

In this section we set up a general framework of this paper, also recalling very briefly necessary elements of the theory of
type I groups and their Fourier analysis.

%---------------------------------------------------------------------------------------------------------
\subsection{General}\label{brocart}
%-----------------------------------------------------------------------------------------------------------

For a given (complex, separable) Hilbert space $\H$\,, the scalar product $\<\cdot,\cdot\>_\H$ will be linear in the first variable and anti-linear in the second.  One denotes by $\mathbb B(\H)$ the $C^*$-algebra of all linear bounded operators in $\H$ and by $\mathbb K(\H)$ the closed bi-sided $^*$-ideal of all the compact operators. The Hilbert-Schmidt operators form a two-sided $^*$-ideal $\mathbb B^2(\H)$ (dense in $\mathbb K(\H)$) which is also a Hilbert space with the scalar product $\<A,B\>_{\mathbb B^2(\H)}:={\rm Tr}\!\(AB^*\)$\,. This Hilbert space is unitarily equivalent to the Hilbert tensor product $\H\otimes\overline{\H}$\,, where $\overline{\H}$ is the Hilbert space opposite to $\H$\,. The unitary operators form a group $\mathbb U(\H)$\,. The commutant of a subset $\mathcal N$ of $\mathbb B(\H)$ is denoted by $\mathcal N'$.

\medskip
Let $\G$ be a locally compact group with unit $\e$ and fixed left Haar measure $\m$\,. Our group will soon be supposed unimodular, so $\m$ will also be a right Haar measure. By $\mathcal C_{\rm c}(\G)$ we denote the space of all complex continuous compactly supported functions on $\G$\,. For $p\in[1,\infty]$\,, the Lebesgue spaces $L^p(\G)\equiv L^p(\G;\m)$ will always refer to the Haar measure. We denote by $C^*(\G)$ the full (universal) $C^*$-algebra of $\G$ and by $C^*_{\rm red}(\G)\subset\mathbb B\big[L^2(\G)\big]$ the reduced $C^*$-algebra of $\G$\,. Recall that any representation $\pi$ of $\G$ generates canonically a non-degenerate represention $\Pi$ of the $C^*$-algebra $C^*(\G)$\,. The notation $A(\G)$ is reserved for Eymard's Fourier algebra of the group $\G$\,.

\medskip
The canonical objects in representation theory \cite{Di,Fo1} will be denoted by ${\rm Rep}(\G),{\rm Irrep(\G)}$ and $\widehat{\G}$\,. An element of ${\rm Rep}(\G)$ is a Hilbert space representation $\pi:\G\rightarrow\mathbb U(\H_\pi)\subset\mathbb B(\H_\pi)$\,, always supposed to be strongly continuous. If it is irreducible, it belongs to ${\rm Irrep(\G)}$ by definition. Unitary equivalence of representations will be denoted by $\cong$\,. We set $\,\wG:={\rm Irrep(\G)}/_{\cong}$ and call it {\it the unitary dual of $\G$}\,. If $\G$ is Abelian, the unitary dual $\wG$ is the Pontryagin dual group; if not, $\wG$ has no group structure.
{\it A primary (factor) representation} $\pi$ satisfies, by definition, $\pi(\G)'\cap\pi(\G)''=\mathbb C\,{\sf id}_{\H_\pi}$\,.

\begin{Definition}\label{uja}
The locally compact group $\G$ will be called {\rm admissible} if it is second countable, type I and unimodular.
\end{Definition}

Admissibility will be a standing assumption and it is needed for most of the main constructions and results. There are hopes to extend at least parts of this paper to non-unimodular groups, by using techniques of \cite{DM}. 

\begin{Remark}\label{nschi}
{\rm We assume that the reader is familiar with  the concept of {\it type I group}. Let us only say that for such a group every primary representation is a direct sum of copies of some irreducible representation; for the full theory we refer to \cite{Di,Fo1,Fu}. In \cite[Th. 7.6]{Fo1} (see also \cite{Di}), many equivalent characterisations are given for a {\it second  countable} locally compact group to be type I. In particular, in such a case, the notion is equivalent to postliminarity (GCR). Thus $\G$ is type I if and only if for every irreducible representation $\pi$ one has $\mathbb K(\H_\pi)\subset\Pi\big[C^*(\G)\big]$\,. 

The single way we are going to use the fact that $\G$ is type I is through one main consequence of this property, to be outlined below: the existence of a measure on the unitary dual $\wG$ for which a Plancherel Theorem holds.
}
\end{Remark}

\begin{Example}\label{joriski}
{\rm Compact and Abelian groups are type I. So are the Euclidean and the Poincar\'e groups. Among the connected groups, real algebraic, exponential (in particular nilpotent) and semi-simple Lie groups are type I. Not all the solvable groups are type I; see \cite[Th. 7.10]{Fo1} for a criterion. 
A discrete group is type I \cite{Th} if and only if it is the finite extension of an Abelian normal subgroup. So the non-trivial free groups or the discrete Heisenberg group are not type I.}
\end{Example}

\begin{Remark}\label{rank}
{\rm We recall that, being second countable, $\G$ will be separable, $\si$-compact and completely metrizable; in particular, as a Borel space it will be standard. The Haar measure $\m$ is $\sigma$-finite and $L^p(\G)$ is a separable Banach space if $p\in[1,\infty)\,$. In addition, all the cyclic representations have separable Hilbert spaces; this applies, in particular, to irreducible representations.

A second countable discrete group is at most countable.
}
\end{Remark}

We mention briefly some harmonic analysis concepts; full treatement is given in \cite{Di,Fo1}. The precise definitions and properties will either be outlined further on, when needed, or they will not be explicitly necessary.

Both ${\rm Irrep(\G)}$ and {\it the unitary dual} $\,\wG:={\rm Irrep(\G)}/_{\cong}\,$ are endowed with (standard) Borel structures \cite[18.5]{Di}. The structure on $\wG$ is the quotient of that on ${\rm Irrep(\G)}$ and is called {\it the Mackey Borel structure}. There is a measure on $\wG$\,, called {\it the Plancherel measure associated to $\m$} and denoted by $\wm$ \cite[18.8]{Di}. Its basic properties, connected to the Fourier transform, will be briefly discussed below. 

The unitary dual $\wG$ is also a separable locally quasi-compact Baire topological space having a dense open locally compact subset \cite[18.1]{Di}. Very often this topological space is not Hausdorff (this is the difference between "locally quasi-compact" and "locally compact").

\begin{Remark}\label{senzatie}
{\rm We are going to use a systematic abuse of notation that we now explain. There is a $\wm$-measurable field $\big\{\,\H_\xi\mid\xi\in\wG\,\big\}$ of Hilbert spaces  and a measurable section $\wG\ni\xi\mapsto\pi_\xi\in{\rm Irrep(\G)}$ such that each $\pi_\xi:\G\rightarrow\mathbb B(\H_\xi)$ is a irreducible representation belonging to the class $\xi$\,. In various formulae, instead of $\pi_\xi$ we will write $\xi$\,, making a convenient identification between irreducible representations and classes of irreducible representations.  
The measurable field of irreducible representations $\big\{\,(\pi_\xi,\H_\xi)\mid\xi\in\wG\,\big\}$ is fixed and other choices would lead to equivalent constructions and statements.
}
\end{Remark}

One introduces the direct integral Hilbert space 
\begin{equation}\label{andy}
\mathscr B^2(\wG):=\int_{\wG}^\oplus\!\B^2(\H_\xi)\,d\wm(\xi)\,\cong\int_{\wG}^\oplus\!\H_\xi\otimes\overline\H_\xi\,d\wm(\xi)\,,
\end{equation}
with the obvious scalar product
\begin{equation}\label{justin}
\<\phi_1,\phi_2\>_{\mathscr B^2(\wG)}:=\int_{\wG}\,\<\phi_1(\xi),\phi_2(\xi)\>_{\mathbb B^2(\H_\xi)}d\wm(\xi)=\int_{\wG}{\rm Tr}_\xi\!\[\phi_1(\xi)\phi_2(\xi)^*\]d\wm(\xi)\,,
\end{equation}
where ${\rm Tr}_\xi$ refers to the trace in $\mathbb B(\H_\xi)$\,.
More generally, for $p\in[1,\infty)$\, one defines $\mathscr B^p(\wG)$ as the family of measurable fields $\,\phi\equiv\big(\phi(\xi)\big)_{\xi\in\wG}\,$ for which $\phi(\xi)$ belongs to the Schatten-von Neumann class $\mathbb B^p(\H_\xi)$ for almost every $\xi$ and
\begin{equation}\label{bp}
\p\!\phi\!\p_{\mathscr B^p(\wG)}:=\Big(\int_{\wG}\p\!\phi(\xi)\!\p_{\mathbb B^p(\H_\xi)}^p\!d\wm(\xi)\Big)^{1/p}<\infty\,.
\end{equation}
They are Banach spaces. We also recall that the von Neumann algebra of decomposable operators $\mathscr B(\wG):=\int_{\wG}^\oplus\B(\H_\xi)\,d\wm(\xi)$ acts to the left and to the right in the Hilbert space $\mathscr B^2(\wG)$ in an obvious way.

\medskip
On $\,\Gamma:=\G\times\wG$\,, which might {\it not} be a locally compact space or a group, we consider the product measure $\m\otimes\wm$\,. It is independent of our choice for $\m$ (if $\m$ is replaced by $\lambda\m$\, for some strictly positive number $\lambda$\,, the corresponding Plancherel measure will be $\lambda^{-1}\wm$)\,. 
Very often we are going to need $\,\widehat{\Gamma}:=\wG\times\G$ (this notation should {\it not} suggest a duality) with the measure $\wm\otimes\m$\,. We could identify it with $\Gamma$
(by means of the map $(\xi,x)\mapsto(x,\xi)$) but in most cases it is better not to do this identification. 

\medskip
Associated to these two measure spaces, we also need the Hilbert spaces
\begin{equation}\label{horace}
\mathscr B^2(\Gamma)\equiv\mathscr B^2\big(\G\times\wG\big):= L^2(\G)\otimes \mathscr B^2(\wG)
\end{equation}
and
\begin{equation}\label{Horace}
\mathscr B^2(\widehat{\Gamma})\equiv\mathscr B^2\big(\wG\times\G\big):= \mathscr B^2(\wG)\otimes L^2(\G)\,,
\end{equation}
also having direct integral decompositions.

%---------------------------------------------------------------------------------------------------------
\subsection{The Fourier transform}\label{brokart}
%-----------------------------------------------------------------------------------------------------------

The Fourier transform \cite[18.2]{Di} of $u\in L^1(\G)$ is given in weak sense by
\begin{equation}\label{ion}
({\fscr F}u)(\xi)\equiv \widehat{u}(\xi):=\int_\G u(x)\xi(x)^*d\m(x) \in\mathbb B(\H_\xi)\,.
\end{equation} 
Here and subsequently the interpretation of $\xi\in\wG$ as a true irreducible representation is along the lines of Remark \ref{senzatie}. Actually, by the compressed form \eqref{ion} we mean that for $\varphi_\xi,\psi_\xi\in\H_\xi$ one has
\begin{equation*}\label{ionn}
\big<({\fscr F}u)(\xi)\varphi_\xi,\psi_\xi\big>_{\H_\xi}:=\int_\G\!u(x)\big<\varphi_\xi,\pi_\xi(x)\psi_\xi\big>_{\H_\xi}d\m(x) \,.
\end{equation*}

Some useful facts \cite[18.2 and 3.3]{Di}: 
\begin{itemize}
\item
The Fourier transform $\,{\fscr F}:L^1(\G)\rightarrow\mathscr B(\wG)$ is linear, contractive and injective\,.
\item
For every $\epsilon>0$ there exists a quasi-compact subset $K_\epsilon\subset\wG$ such that $\,\p\!({\fscr F}u)(\xi)\!\p_{\mathbb B(\H_\xi)}\,\le\epsilon\,$ if $\xi\notin K_\epsilon$\,. 
\item
The map $\,\wG\ni\xi\mapsto\p\!({\fscr F}u)(\xi)\!\p_{\mathbb B(\H_\xi)}\in\R\,$ is lower semi-continuous. It is even continuous, whenever $\wG$ is Hausdorff.
\end{itemize}

Recall \cite{Di,Fo,Fu} that {\it the Fourier transform ${\fscr F}$ extends (starting from $L^1(\G)\cap L^2(\G)$) to a unitary isomorphism 
${\fscr F}:L^2(\G)\rightarrow \mathscr B^2(\wG)$}\,. This is the generalisation of Plancherel's Theorem to (maybe non-commutative) admissible groups and it will play a central role in our work. 

\begin{Remark}\label{bodrum}
{\rm It is also known \cite{Li,Fu} that ${\fscr F}$ restricts to a bijection
\begin{equation}\label{jection}
{\fscr F}_{(0)}:L^2(\G)\cap A(\G)\rightarrow \mathscr B^2(\wG)\cap\mathscr B^1(\wG)
\end{equation}
 with inverse given by (the traces refer to $\H_\xi$)
\begin{equation}\label{marian}
\({\fscr F}_{(0)}^{-1}\phi\)\!(x)=\int_{\widehat{\G}}{\rm Tr}_\xi[\xi(x)\phi(\xi)]d\wm(\xi)\,.
\end{equation}
Rephrasing this, the restriction of the inverse ${\fscr F}^{-1}$ to the subspace $\mathscr B^2(\wG)\cap\mathscr B^1(\wG)$ has the explicit form \eqref{marian}, and this will be a useful fact.
Note the consequence, valid for $u\in L^2(\G)\cap A(\G)$ and for $\m$-almost every $x\in\G$\,:
\begin{equation}\label{megazork}
u(x)=\int_{\widehat{\G}}{\rm Tr}_\xi[({\fscr F}u)(\xi)\xi(x)]d\wm(\xi)=
\int_{\widehat{\G}}{\rm Tr}_\xi[\xi(x)\widehat{u}(\xi)]d\wm(\xi)\,.
\end{equation}
In particular, this holds for $u\in\mathcal C_{\rm c}(\G)$\,. The extension ${\fscr F}_{(1)}$ of ${\fscr F}_{(0)}$ to $A(\G)$ makes sense as an isometry 
${\fscr F}_{(1)}:A(\G)\rightarrow\mathscr B^1(\wG)$\,.

\medskip
Combining the quantization formula \eqref{fred} with the Fourier transform \eqref{ion}, we can write  \eqref{fred} also as
\begin{equation}\label{fred0}
\[{\sf Op}(a)u\]\!(x)= \int_{\wG}{\rm Tr}_\xi\!\[\xi(x)a(x,\xi) \widehat{u}(\xi)\]\!d\wm(\xi)\,,
\end{equation}
which can be viewed as an extension of the Fourier inversion formula \eqref{megazork}.
}
\end{Remark}

\begin{Remark}\label{vudrun}
{\rm By a formula analoguous to \eqref{ion}, the Fourier transform is even defined (and injective) on bounded complex Radon measures $\mu$ on $\G$\,. One gets easily 
\begin{equation*}\label{mormaici}
\sup_{\xi\in\wG}\p\!{\fscr F}\mu\!\p_{\mathbb B(\H_\xi)}\,\le\,\p\!\mu\!\p_{M^1(\G)}\,:=|\mu|(\G)\,.
\end{equation*}
}
\end{Remark}

\begin{Remark}\label{ltima}
{\rm There are many (full or partial) Fourier transformations that can play important roles, as
\begin{equation}\label{forier}
{\fscr F}\otimes{\sf id}:L^2(\G\times\G)\rightarrow\mathscr B^2(\widehat{\Gamma})\,,\quad {\sf id}\otimes{\fscr F}:L^2(\G\times\G)\rightarrow\mathscr B^2(\Gamma)\,.
\end{equation}
\begin{equation}\label{furier}
{\fscr F}\otimes{\fscr F}^{-1}:\mathscr B^2(\Gamma)\rightarrow\mathscr B^2(\widehat{\Gamma})\,,\quad{\fscr F}^{-1}\otimes{\fscr F}:\mathscr B^2(\widehat{\Gamma})\rightarrow\mathscr B^2(\Gamma)\,.
\end{equation}
They might admit various extensions or restrictions.}
\end{Remark}

%-------------------------------------------------------------------------------------------------------
\section{Quantization by a Weyl system}\label{fourtein}
%-------------------------------------------------------------------------------------------------------

In this section we introduce a notion of a Weyl system in our setting and outline its relation to Wigner and Fourier-Wigner transforms. This is then used to define pseudo-differential operators through $\tau$-quantization for an arbitrary measurable function
$\,\tau:\G\rightarrow\G$\,. The introduced formalism is then applied to study (involutive) algebra properties of symbols and operators as well as Schatten class properties in the setting of non-commutative $L^p$-spaces. One of the goals here is to give rigorous understanding to the $\tau$-quantization formula \eqref{bilfred}. 

%-------------------------------------------------------------------------------------------------------
\subsection{Weyl systems and their associated transformations}\label{tintin}
%-------------------------------------------------------------------------------------------------------

Let us fix a measurable function $\,\tau:\G\rightarrow\G$\,. We will often use the notation $\tau x\equiv \tau(x)$ to avoid writing too many brackets.

\begin{Definition}\label{siegmund}
For $x\in\G$ and $\pi\in{\rm Rep(\G)}$ one defines a unitary operator $W^\tau\!(\pi,x)$ in the Hilbert space $L^2(\G;\H_\pi)\equiv L^2(\G)\otimes\H_\pi$ by
\begin{equation}\label{frieda}
\[W^\tau\!(\pi,x)\Theta\]\!(y):=\pi\!\[y(\tau x)^{-1}\]^*[\Theta(yx^{-1})]=\pi[\tau(x)]\pi(y)^*[\Theta(yx^{-1})]\,.
\end{equation}
\end{Definition}

If $\pi\cong\rho$\,, i.e. if $\rho(x)U=U\pi(x)$ for some unitary operator $U:\H_\pi\rightarrow\H_\rho$ and for every $x\in\G$\,, then it follows easily that 
\begin{equation*}\label{fritz}
W^\tau(\rho,x)=({\sf id}\otimes U)W^\tau(\pi,x)({\sf id}\otimes U)^{-1}.
\end{equation*}

We record for further use the formula
\begin{equation}\label{lividiu}
\begin{aligned}
W^{\tau'}\!(\pi,x)&=\big[{\sf id}\otimes\pi(\tau'x)\big]\big[{\sf id}\otimes\pi(\tau x)^*\big]W^\tau\!(\pi,x)\\
&=\big[{\sf id}\otimes\pi\!\((\tau'x)(\tau x)^{-1}\)\big]W^\tau\!(\pi,x)
\end{aligned}
\end{equation}
making the connection between operators defined by different parametres $\tau,\tau'$ as well as the explicit form of the adjoint
\begin{equation*}\label{liviu}
\[W^\tau\!(\pi,x)^*\Theta\]\!(y)=\pi\!\[yx(\tau x)^{-1}\]\!\[\Theta(yx)\].
\end{equation*}
One also notes that $W^\tau\!({\mathfrak 1},x)=R\big(x^{-1}\big)$\,, where $R$ is the right regular representation of $\G$ and ${\mathfrak 1}$ is the $1$-dimensional trivial representation. In this case $\H_{\mathfrak 1}=\mathbb C$\,, so $L^2(\G;\H_{\mathfrak 1})$ reduces to $L^2(\G)$\,.

\begin{Remark}\label{stupoare}
{\rm 
One can not compose the operators $W^\tau\!(\pi,x)$ and $W^\tau\!(\rho,y)$ in general, since they act in different Hilbert spaces. Note, however, that the family $\,{\rm Rep}(\G)/_{\cong}\,$ of all the unitary equivalence classes of representations form an Abelian monoid with the tensor composition
\begin{equation*}\label{tensori?}
(\pi\otimes\rho)(x):=\pi(x)\otimes\rho(x)\,,\quad x\in\G\,,
\end{equation*}
and the unit $\mathfrak 1$ (after a suitable reinterpretation in terms of equivalence classes). The subset $\,\wG={\rm Irrep}(\G)/_{\cong}\,$ is not a submonoid in general, but the generated submonoid, involving finite tensor products of irreducible representations, could be interesting. It is instructive to compute the operator in $L^2(\G;\H_\pi\otimes\H_\rho)$ 
\begin{equation}\label{racoare}
\big[W(\pi,x)\otimes{\sf id}_\rho\big]\big[W(\rho,y)\otimes{\sf id}_\pi\big]=\big[{\sf id}_{L^2(\G)}\otimes\rho(x)\otimes{\sf id}_\pi\big]W(\pi\otimes\rho,yx)\,;
\end{equation}
to get this result one has to identify $\H_\pi\otimes\H_\rho$ with $\H_\rho\otimes\H_\pi$\,. If $\G$ is Abelian, the unitary dual $\wG$ is the Pontryagin dual group, the irreducible representations are $1$-dimensional and for $\xi\equiv\pi\in{\wG}$ and $\eta\equiv\rho\in{\wG}$ the identity \eqref{racoare} reads
\begin{equation*}\label{racorica}
W(\xi,x)W(\eta,y)=\eta(x)W(\xi\eta,xy)\,.
\end{equation*}
Thus $\,W:\wG\times\G\rightarrow\mathbb B[L^2(\G)]$ is a unitary projective representation with $2$-cocycle (multiplier)
\begin{equation*}\label{ite}
\varpi:\big(\wG\times\G\big)\times\big(\wG\times\G\big)\rightarrow\mathbb T\,,\quad\varpi\big((\xi,x),(\eta,y)\big):=\eta(x)\,.
\end{equation*}
Similar computations can be done for $W^\tau$ with general $\tau$\,.
}
\end{Remark}

From now one we mostly concentrate on the family of operators $W^\tau\!(\xi,x)$ where $x\in\G$ and $\xi$ is an irreducible representation. Extrapolating from the case $\G=\R^n$, we call this family {\it a Weyl system}.

\medskip
Below, for an operator $T$ in $L^2(\G;\H_\xi)\cong L^2(\G)\otimes\H_\xi$ and a pair of vectors $u,v\in L^2(\G)$\,, the action of $\,\left<Tu,v\right>_{L^2(\G)}\in\mathbb B(\H_\xi)$ on $\,\varphi_\xi\in\H_\xi$ is given by 
\begin{equation}\label{tauriel}
\left<Tu,v\right>_{L^2(\G)}\varphi_\xi:=\int_{\G}[T(u\otimes\varphi_\xi)](y)\overline{v(y)}\,d\m(y)\in\H_\xi\,.
\end{equation}

\begin{Definition}\label{andrei}
For $(x,\xi)\in\G\times\wG$ and $u,v\in L^2(\G)$ one sets
\begin{equation}\label{leo}
\mathcal W^\tau_{u,v}(\xi,x):=\left<W^\tau\!(\xi,x)\overline u,\overline v\right>_{L^2(\G)}\in\mathbb B(\H_\xi)\,.
\end{equation}
\end{Definition}

This definition is suggested by the standard notion of {\it representation coefficient function} from the theory of unitary group representations. However, in general, $\wG\times\G$ is not a group, $\mathcal W^\tau_{u,v}$ is not scalar-valued, and
$W^\tau\!(\xi,x)W^\tau\!(\eta,y)$ makes no sense. 

\begin{Remark}\label{toate}
{\rm Note the identity
\begin{equation}\label{suspinare}
\left<\mathcal W_{u,v}^\tau(\xi,x)\varphi_\xi,\psi_\xi\right>_{\H_\xi}=\big<W^\tau\!(\xi,x)(\overline u\otimes\varphi_\xi),\overline v\otimes\psi_\xi\big>_{L^2(\G;\H_\xi)}\,,
\end{equation}
valid for $u,v\in L^2(\G)\,,\,\varphi_\xi,\psi_\xi\in\H_\xi\,,\,(\xi,x)\in\widehat\Gamma$\,. It follows immediately from \eqref{leo} and \eqref{tauriel}. In fact \eqref{suspinare} can serve as a definition of $\mathcal W_{u,v}^\tau(\xi,x)$\,.}
\end{Remark}

\begin{Proposition}\label{mataus}
The mapping $(u,v)\mapsto\mathcal W^\tau_{u,v}$ defines a unitary map (denoted by the same symbol) $\mathcal W^\tau:\overline{L^2(\G)}\otimes L^2(\G)\rightarrow\mathscr B^2(\widehat\Gamma)$\,, called {\rm the Fourier-Wigner $\tau$-transformation}.
\end{Proposition}

\begin{proof}
Let us define the change of variables 
\begin{equation}\label{religie}
{\rm cv}^\tau:\G\times\G\rightarrow\G\times\G\,,\quad {\rm cv}^\tau(x,y):=\big(x\tau(y^{-1}x)^{-1}\!,y^{-1}x\big)
\end{equation}
with inverse
\begin{equation}
\big({\rm cv}^\tau\big)^{-1}(x,y)=\big(x\tau(y),x\tau(y)y^{-1}\big)\,.
\end{equation}
Using the definition and the interpretation \eqref{tauriel}, one has for $\varphi_\xi\in\H_\xi$
$$
\begin{aligned}
\mathcal W^\tau_{u,v}(\xi,x)\varphi_\xi
=&\int_\G \[W^\tau\!(\xi,x)(\overline u\otimes\varphi_\xi)\]\!(z)v(z)\,d\m(z)\\
=&\int_\G v(z)\,\overline{u(zx^{-1})}\,\xi\!\(z\tau(x)^{-1}\)^*\!\varphi_\xi\,d\m(z)\\
=&\int_\G v\!\(y\tau(x)\)\overline{u(y\tau(x)x^{-1})}\,\xi\!\(y\)^*\!\varphi_\xi\,d\m(y)\\
=&\int_\G (v\otimes\overline{u})\big[\big({\rm cv}^\tau\big)^{-1}(y,x)\big]\,\xi\!\(y\)^*\!\varphi_\xi\,d\m(y)\,.
\end{aligned}
$$
By using the properties of the Haar measure and the unimodularity of $\G$\,, it is easy to see that the composition with the map ${\rm cv}^\tau$, denoted by ${\rm CV}^\tau$, is a unitary operator in $L^2(\G\times\G)\cong L^2(\G)\otimes L^2(\G)$\,. On the other hand, the conjugation $\overline{L^2(\G)}\ni w\mapsto\overline w\in L^2(\G)$ is also unitary.
Making use of the unitary partial Fourier transformation 
$$
({\fscr F}\otimes{\sf id}):L^2(\G)\otimes L^2(\G)\rightarrow \mathscr B^2(\wG)\otimes L^2(\G)\,,
$$ 
one gets 
\begin{equation}\label{ostrogot}
\mathcal W^\tau_{u,v}=({\fscr F}\otimes{\sf id})\big({\rm CV}^\tau\big)^{-1}(v\otimes\overline u)
\end{equation} 
and the statement follows.
\end{proof}

The unitarity of the Fourier-Wigner transformation implies the next irreducibility result:

\begin{Corollary}\label{starnac}
Let $\mathcal K$ be a closed subspace of $L^2(\G)$ such that $\,W^\tau\!(\xi,x)(\mathcal K\otimes\H_\xi)\subset\mathcal K\otimes\H_\xi\,$ for every $(\xi,x)\in\widehat\Gamma$\,.
Then $\mathcal K=\{0\}$ or $\mathcal K=L^2(\G)$\,.
\end{Corollary}

\begin{proof}
Suppose that $\mathcal K\ne L^2(\G)$ and let $\overline{v}\in\mathcal K^\perp\setminus\{0\}$\,. 

Let us examine the identity \eqref{suspinare}, where $\overline u\in\mathcal K$\,, $(\xi,x)\in\widehat\Gamma$ and $\varphi_\xi,\psi_\xi\in \H_\xi$\,. Since $W^\tau\!(\xi,x)(\overline{u}\otimes\varphi_\xi)\in\mathcal K\otimes\H_\xi$\,, the right hand side is zero. So the left hand side is also zero for $\varphi_\xi,\psi_\xi$ arbitrary, so $\mathcal W^\tau_{u,v}(\xi,x)=0$\,. Then, by unitarity
$$
\p\!u\!\p^2_{L^2(\G)}\p\!v\!\p^2_{L^2(\G)}\,=\,\p\!\mathcal W^\tau_{u,v}\!\p^2_{\mathscr B^2(\widehat\Gamma)}\,=\int_\G\int_{\wG}\p\!\mathcal W^\tau_{u,v}(\xi,x)\!\p^2_{\mathbb B^2(\H_\xi)}\!d\m(x)d\wm(\xi)=0\,,
$$
and since $v\ne 0$ one must have $u=0$\,.
\end{proof}

Depending on the point of view, one uses one of the notations $\mathcal W^\tau_{u,v}$ or $\mathcal W^\tau\!(u\otimes v)$\,. We also introduce 
\begin{equation}\label{verner}
\mathcal V^\tau_{u,v}\equiv\mathcal V^\tau(u\otimes v):=({\fscr F}^{-1}\otimes{\fscr F})\mathcal W^\tau_{v,u}
=({\sf id}\otimes{\fscr F})\big({\rm CV}^\tau\big)^{-1}(v\otimes\overline u)\in L^2(\G)\otimes\mathscr B^2(\wG)\,,
\end{equation}
which reads explicitly
\begin{equation*}\label{ostrogoth}
\mathcal V^\tau_{u,v}(x,\xi)=\int_\G \overline{u\big(x\tau(y)y^{-1}\big)}v\big(x\tau(y)\big)\xi(y)^*d\m(y)\,.
\end{equation*}
One can name the unitary mapping $\mathcal V^\tau:\overline{L^2(\G)}\otimes L^2(\G)\rightarrow\mathscr B^2(\Gamma)$ {\it the Wigner $\tau$-transformation}.
We record for further use {\it the orthogonality relations}, valid for $u,u',v,v'\in L^2(\G)$\,:

\begin{equation}\label{warner}
\big\<\mathcal W^\tau_{u,v},\mathcal W^\tau_{u',v'}\big\>_{\!\mathscr B^2(\widehat\Gamma)}=\big\<u',u\big\>_{\!L^2(\G)}\big\<v,v'\big\>_{\!L^2(\G)}=\big\<\mathcal V^\tau_{u,v},\mathcal V^\tau_{u',v'}\big\>_{\!\mathscr B^2(\Gamma)}\,.
\end{equation}

%-------------------------------------------------------------------------------------------------------
\subsection{Pseudo-differential operators}\label{teletin}
%-------------------------------------------------------------------------------------------------------

Let, as before, $\tau:\G\rightarrow\G$ be a measurable map. The next definition should be seen as a rigorous way to give sense to
the $\tau$-quantization ${\sf Op}^\tau\!(a)$ introduced in \eqref{bilfred}. 

We note that in general, due to various non-commutativities (of the group, of the symbols), there are essentially
two ways of introducing the quantization of this type - these will be given and discussed in the sequel in Section \ref{fourtamin},
see especially formulae \eqref{fred0d} and \eqref{fred0s}. In the context of compact Lie groups these issues have been 
extensively discussed in \cite{RT}, see e.g. Remark 10.4.13 there, and most of that discussion extends to our present setting. 
One advantage of the order of operators in the definition
\eqref{bilfred} is that the invariant operators can be viewed as Fourier multipliers with multiplication by the symbol from the left
\eqref{EQ:fm1},
which is perhaps a more familiar way of viewing such operators in non-commutative harmonic analysis. However, it will turn out that the other ordering has certain advantages from the point of view of $C^*$-algebra theories. We postpone these topics to subsequent sections.

\begin{Definition}\label{dietrich}
For $a\in\mathscr B^2(\Gamma)$ (with Fourier transform $\,\widehat a:=\big({\fscr F}\otimes{\fscr F}^{-1}\big)a\in\mathscr B^2(\widehat\Gamma)$) we define ${\sf Op}^\tau\!(a)$ to be the unique bounded linear operator in $L^2(\G)$ associated by the relation 
\begin{equation}\label{rubin}
{\rm op}^\tau_a(u,v)=\big\<{\sf Op}^\tau\!(a) u,v\big\>_{L^2(\G)}
\end{equation} 
to the bounded sesquilinear form ${\rm op}^\tau_a:L^2(\G)\times L^2(\G)\rightarrow\mathbb C$
\begin{equation}\label{fernandaa}
{\sf op}^\tau_a(u,v):=\big\<\widehat a,\mathcal W^\tau_{u,v}\big\>_{\!\mathscr B^2(\widehat\Gamma)}\,=\int_{\G}\!\int_{\wG}\,{\rm Tr}_\xi\!\left[\widehat a(\xi,x)\mathcal W^\tau_{u,v}(\xi,x)^*\right]d\m(x)d\wm(\xi)
\end{equation}
or, equivalently,
\begin{equation}\label{fernandell}
{\sf op}^\tau_a(u,v):=\big\<a,\mathcal V^\tau_{u,v}\big\>_{\!\mathscr B^2(\Gamma)}\,=\int_{\G}\!\int_{\wG}\,{\rm Tr}_\xi\!\left[a(x,\xi)\mathcal V^\tau_{u,v}(x,\xi)^*\right]d\m(x)d\wm(\xi)\,.
\end{equation}
One says that ${\sf Op}^\tau\!(a)$ is {\rm the $\tau$-pseudo-differential operator corresponding to the operator-valued symbol} $a$ while the map $a\to{\sf Op}^\tau\!(a)$ will be called {\rm the $\tau$-pseudo-differential calculus} or {\rm $\tau$-quantization}.
\end{Definition}

To justify Definition \ref{dietrich}, one must show that ${\rm op}^\tau_a$ is indeed a well-defined bounded sesquilinear form. 
Clearly ${\rm op}^\tau_a(u,v)$ is linear in $u$ and antilinear in $v$\,. Using the Cauchy-Schwartz inequality in the Hilbert space $\mathscr B^2(\widehat\Gamma)$\,, the Plancherel formula and Proposition \ref{mataus}, one gets
$$
|{\sf op}^\tau_a(u,v)|\le\,\p\!\widehat a\!\p_{\mathscr B^2(\widehat\Gamma)}\p\!\mathcal W^\tau_{u,v} \!\p_{\mathscr B^2(\widehat\Gamma)}\,=\,\p\!a\!\p_{\mathscr B^2(\Gamma)}\p\! u\!\p_{L^2(\G)}\p\! v\!\p_{L^2(\G)}. 
$$
This implies in particular the estimation $\p\!{\sf Op}^\tau\!(a)\!\p_{\mathbb B[L^2(\G)]}\,\le\,\p\!a\!\p_{\mathscr B^2(\Gamma)}$\,. This will be improved in the next result, in which we identify the rank-one, the trace-class and the Hilbert-Schmidt operators in $L^2(\G)$ as $\tau$-pseudo-differential operators.

\begin{Theorem}\label{carnat}
\begin{enumerate}
\item
Let us define by 
\begin{equation*}\label{god}
\Lambda_{u,v}(w):=\<w,u\>_{L^2(\G)}\,v\,,\quad\forall\,w\in L^2(\G)
\end{equation*}
the rank-one operator associated to the pair of vectors $(u,v)$\,. Then one has
\begin{equation}\label{safir}
\Lambda_{u,v}={\sf Op}^\tau\!\(\mathcal V^\tau_{u,v}\)\,,\quad\forall\,u,v\in L^2(\G)\,.
\end{equation}
\item
Let $\,T$ be a trace-class operator in $L^2(\G)$\,. Then there exist orthonormal seqences $(u_n)_{n\in\N}$\,,
$(v_n)_{n\in\N}$ and a sequence $(\lambda_n)_{n\in\N}\subset\mathbb C$ with $\,\sum_{n\in\N}|\lambda_n|<\infty$ such that
\begin{equation}\label{fardan}
T=\sum_{n\in\N}\lambda_n{\sf Op}^\tau\!\(\mathcal V^\tau_{u_n,v_n}\).
\end{equation}
\item
The mapping ${\sf Op}^\tau$ sends unitarily $\mathscr B^2(\Gamma)$ in the Hilbert space composed of all Hilbert-Schmidt operators in $L^2(\G)$\,.
\end{enumerate}
\end{Theorem}

\begin{proof}
1. By the definition \eqref{fernandell} and the orthogonality relations \eqref{warner}, one has for $u',v'\in L^2(\G)$
$$
\begin{aligned}
\big\<{\sf Op}^\tau\!(\mathcal V^\tau_{u,v}) u',v'\big\>_{L^2(\G)}&=\big\<\mathcal V^\tau_{u,v},\mathcal V^\tau_{u',v'}\big\>_{\mathscr B^2(\Gamma)}\\
&=\big\<u',u\big\>_{\!L^2(\G)}\big\<v,v'\big\>_{\!L^2(\G)}\\
&=\big\<\Lambda_{u,v}u',v'\big\>_{L^2(\G)}.
\end{aligned}
$$

2. Follows from 1 and from the fact \cite[pag. 494]{Tr} that every trace-class operator $T$ can be written as $\,T=\sum_{n\in\N}\lambda_n\Lambda_{u_n,v_n}$ with $u_n,v_n,\lambda_n$ as in the statement.

\medskip
3. One recalls that $\Lambda$ defines (by extension) a unitary map $\,\overline{L^2(\G)}\otimes L^2(\G)\to\mathbb B^2\big[L^2(\G)\big]$ and that $\V^\tau$ is also unitary and note that 
\begin{equation}\label{gott}
{\sf Op}^\tau=\Lambda\circ\big(\mathcal V^\tau\big)^{-1}=\Lambda\circ\big(\mathcal W^\tau\big)^{-1}\circ\big({\fscr F}\otimes{\fscr F}^{-1}\big)\,.
\end{equation}
Another proof consists in examining the integral kernel of ${\sf Op}^\tau\!(a)$ given in Proposition \ref{frizt}.
\end{proof}

The unitarity of the map ${\sf Op}^\tau$ can be written in the form
\begin{equation*}\label{ostro}
{\rm Tr}\big[{\sf Op}^\tau\!(a){\sf Op}^\tau\!(b)^*\big]=\int_{\G}\int_{\wG}{\rm Tr}_\xi\big[a(x,\xi)b(x,\xi)^*\big]d\m(x)d\wm(\xi)\,,
\end{equation*}
where ${\rm Tr}$ refers to the trace in $\mathbb B\big[L^2(\G)\big]$\,.

\begin{Proposition}\label{frizt}
If $\,a\in\mathscr B^2(\Gamma)$\,, then ${\sf Op}^\tau\!(a)$ is an integral operator with kernel $\,{\sf Ker}^\tau_a\in L^2(\G\times\G)$ given by 
\begin{equation}\label{dinou}
{\sf Ker}^\tau_a(x,y):={\rm CV}^\tau({\sf id}\otimes{\fscr F}^{-1})a\,.
\end{equation}
\end{Proposition}

\begin{proof}
Using the definitions, Plancherel's Theorem and the unitarity of ${\sf CV}^\tau$, one gets
$$
\begin{aligned}
\<{\sf Op}^\tau\!(a) u,v\>_{L^2(\G)}:=&\,\,\big\<a,\mathcal V^\tau_{u,v}\big\>_{\!\mathscr B^2(\Gamma)}\\
=&\,\,\Big\<a,({\sf id}\otimes{\fscr F})\big({\rm CV}^\tau\big)^{-1}(v\otimes\overline u)\Big\>_{\!L^2(\G)\otimes\mathscr B^2(\wG)}\\
=&\,\,\Big\<({\sf id}\otimes{\fscr F}^{-1})a,\big({\rm CV}^\tau\big)^{-1}(v\otimes\overline u)\Big\>_{\!L^2(\G)\otimes L^2(\G)}\\
=&\,\,\Big\<{\rm CV}^\tau({\sf id}\otimes{\fscr F}^{-1})a,(v\otimes\overline u)\Big\>_{\!L^2(\G)\otimes L^2(\G)}\\
=&\int_\G\int_\G \big[{\rm CV}^\tau({\sf id}\otimes{\fscr F}^{-1})a\big](x,y)(\overline v\otimes u)(x,y)d\m(y)d\m(x)\\
=&\int_\G\Big(\int_\G \big[{\rm CV}^\tau({\sf id}\otimes{\fscr F}^{-1})a\big](x,y)u(y)d\m(y)\Big) \overline{v(x)}d\m(x)\,,
\end{aligned}
$$
completing the proof.
\end{proof}

\begin{Remark}\label{emarc}
{\rm 
We rephrase Proposition \ref{frizt} as 
\begin{equation}\label{zaruz}
{\sf Op}^\tau={\sf Int}\circ{\sf Ker}^\tau={\sf Int}\circ{\rm CV}^\tau\!\circ({\sf id}\otimes{\fscr F}^{-1})\,,
\end{equation}
where ${\sf Int}:L^2(\G\times\G)\rightarrow\mathbb B^2\big[L^2(\G)\big]$ is given by 
\begin{equation*}\label{ruj}
[{\sf Int}(M)u](x):=\int_\G M(x,y)u(y)d\m(y)\,.
\end{equation*}
Now we see that ${\sf Op}^\tau$ actually coincides with the one defined in \eqref{bilfred}, at least in a certain sense. Formally, using \eqref{dinou}, one gets
\begin{equation}\label{siegfrid}
{\sf Ker}^\tau_a(x,y)=\int_{\wG}{\rm Tr}_\xi\Big[a\big(x\tau(y^{-1}x)^{-1}\!,\xi\big)\xi(y^{-1}x)\Big]\,d\wm(\xi)
\end{equation}
and this should be compared to \eqref{bilfred}. The formula \eqref{siegfrid} is rigorously correct if, for instance, the symbol $a$ belongs to $({\sf id}\otimes{\fscr F})\mathcal C_{\rm c}(\G\times\G)$\,, since the explicit form \eqref{marian} of the inverse Fourier transform holds on ${\fscr F}\mathcal C_{\rm c}(\G)\subset {\fscr F}\big[A(\G)\cap L^2(\G)\big]=\mathscr B^1(\wG)\cap\mathscr B^2(\wG)$\,. Thus we reobtain the formula \eqref{bilfred} as
\begin{equation}\label{bilfred3}
\begin{aligned}
\[{\sf Op}^\tau\!(a)u\]\!(x) &= \int_\G  {\sf Ker}^\tau_a(x,y) u(y) d\m(y)\\
&=\int_\G\!\Big(\int_{\wG}\,{\rm Tr}_\xi\Big[\xi(y^{-1}x)a\big(x\tau(y^{-1}x)^{-1}\!,\xi\big)\Big]d\wm(\xi)\Big)u(y)d\m(y)\,.
\end{aligned}
\end{equation}
}
\end{Remark}

\begin{Remark}\label{paul}
{\rm If $\tau,\tau'\!:\G\rightarrow\G$ are measurable maps, the associated pseudo-differential calculi are related by ${\sf Op}^{\tau'}\!\!(a)={\sf Op}^\tau\!(a_{\tau\tau'})$ where, based on \eqref{zaruz}, one gets
\begin{equation}\label{david}
({\sf id}\otimes{\fscr F}^{-1})a_{\tau\tau'}=\big[({\sf id}\otimes{\fscr F}^{-1})a\big]\circ{\sf cv}^{\tau'}\circ\big({\sf cv}^\tau\big)^{-1}.
\end{equation}
One computes easily
\begin{equation}\label{estrecho}
{\sf cv}^{\tau'\tau}(x,y):=\big[{\sf cv}^{\tau'}\!\circ\big({\sf cv}^\tau\big)^{-1}\big](x,y)=\big(x\tau(y)\tau'(y)^{-1}\!,y\big)\,.
\end{equation}
However, it seems difficult to turn this into a nice explicit formula for $a_{\tau\tau'}$\,, but this is already the case in the Euclidean space too. The crossed product realisation is nicer from this point of view (when ``turned to the right").  Using \eqref{radar} one can write
\begin{equation}\label{nexiune}
{\sf Sch}^{\tau'}\!(\Phi)={\sf Sch}^{\tau}(\Phi_{\tau\tau'})\,,
\end{equation}
with $\Phi_{\tau\tau'}=\Phi\circ{\sf cv}^{\tau'\tau}$. See also Remark \ref{labush}.}
\end{Remark}

%-------------------------------------------------------------------------------------------------------
\subsection{Involutive algebras of symbols}\label{fraterin}
%-------------------------------------------------------------------------------------------------------

Since our pseudo-differential calculus is one-to-one, we can define an involutive algebra structure on operator-valued symbols, emulating the algebra of operators.
One defines a composition law $\,\#_\tau$ and an involution $\,^{\#_\tau}$ on $\mathscr B^2(\widehat\Gamma)$ by 
\begin{equation*}\label{topaz}
{\sf Op}^\tau\!(a\#_\tau b):={\sf Op}^\tau\!(a){\sf Op}^\tau\!(b)\,,
\end{equation*}
\begin{equation*}\label{toppaz}
{\sf Op}^\tau\!(a^{\#_\tau}):={\sf Op}^\tau\!(a)^*.
\end{equation*}
The composition can be written in terms of integral kernels as
\begin{equation*}\label{intker}
{\sf Ker}^\tau_{a\#_\tau b}={\sf Ker}^\tau_a\bullet{\sf Ker}^\tau_b\,,
\end{equation*}
where, by \eqref{zaruz}, 
\begin{equation*}\label{zuruz}
{\sf Ker}^\tau\!:={\rm CV}^\tau\!\circ({\sf id}\otimes{\fscr F}^{-1})
\end{equation*} 
and $\,\bullet\,$ is the usual composition of kernels
\begin{equation*}\label{intco}
(M\bullet N)(x,y):=\int_\G M(x,z)N(z,y) d\m(z)\,,
\end{equation*} 
corresponding to ${\sf Int}(M\bullet N)={\sf Int}(M){\sf Int}(N)$\,.
It follows that for $a,b\in\mathscr B^2(\widehat\Gamma)$
\begin{equation}\label{compoz}
\begin{aligned}
a\#_\tau b&=\big({\sf Ker}^\tau\big)^{-1}\big({\sf Ker}^\tau_a\bullet{\sf Ker}^\tau_b\big)\\
&=({\sf id}\otimes{\fscr F})\circ({\rm CV}^\tau)^{-1}\Big\{\big[{\rm CV}^\tau\circ({\sf id}\otimes{\fscr F}^{-1})\big]a\bullet\big[{\rm CV}^\tau\circ({\sf id}\otimes{\fscr F}^{-1})\big]b\Big\}\,.
\end{aligned}
\end{equation}

Similarly, in terms of the natural kernel involution $M^\bullet(x,y):=\overline{M(y,x)}$ (corresponding to ${\sf Int}(M)^*={\sf Int}(M^\bullet)$)\,, one gets
\begin{equation}\label{invo}
a^{\#_\tau}=\big({\sf Ker}^\tau\big)^{-1}\big[({\sf Ker}^\tau_a)^\bullet\big]=({\sf id}\otimes{\fscr F})\circ({\rm CV}^\tau)^{-1}\Big\{\Big(\big[{\rm CV}^\tau\circ({\sf id}\otimes{\fscr F}^{-1})\big]a\Big)^{\!\bullet}\Big\}\,.
\end{equation}

\begin{Remark}\label{hilbalg}
{\rm As a conclusion, $\big(\mathscr B^2(\Gamma),\#_\tau,^{\#_\tau}\big)$ is a $^*$-algebra. This is part of a more detailed result, stating that $\Big(\mathscr B^2(\Gamma),\<\cdot,\cdot\>_{\mathscr B^2(\Gamma)},\#_\tau,^{\#_\tau}\Big)$ is an $H^*$-algebra, i.e. a complete Hilbert algebra \cite[App. A]{Di}. Among others, this contains the following compatibility relations between the scalar product and the algebraic laws
\begin{equation*}\label{pantazi}
\big\<a\#_\tau b,c\big\>_{\mathscr B^2(\Gamma)}=\big\<a, b^{\#_\tau}\#_\tau c\big\>_{\mathscr B^2(\Gamma)}\,,
\end{equation*}
\begin{equation*}\label{pantatzi}
\<a,b\>_{\mathscr B^2(\Gamma)}=\big\<b^{\#_\tau},a^{\#_\tau}\big\>_{\mathscr B^2(\Gamma)}\,,
\end{equation*}
valid for every $a,b,c\in\mathscr B^2(\Gamma)$\,. The simplest way to prove all these is to recall that $\mathbb B^2\big[L^2(\G)\big]$ is an $ H^*$-algebra with the operator multiplication, with the adjoint and with the complete scalar product $\<S,T\>_{\mathbb B^2}:={\rm Tr}[ST^*]$ and to invoke the algebraic and unitary isomorphism $\mathscr B^2(\Gamma)\overset{{\sf Op}^\tau}{\cong}\mathbb B^2\big[L^2(\G)\big]$\,.
}
\end{Remark}

Formulae \eqref{compoz} and \eqref{invo} take a more explicit integral form on symbols particular enough to allow applying formula \eqref{marian} for the inverse Fourier transform. Since, anyhow, we will not need such formulas, we do not pursue this here. Let us give, however, the simple algebraic rules satisfied by the Wigner $\tau$-transforms defined in \eqref{verner}\,:

\begin{Corollary}\label{wizzi}
For every $u,v,u_1,u_2,v_1,v_2\in L^2(\G)$ one has
\begin{equation}\label{vizi}
\mathcal V^\tau_{u_1,v_1}\#_\tau\,\mathcal V^\tau_{u_2,v_2}=\<v_2,u_1\>\mathcal V^\tau_{u_2,v_1}
\end{equation}
and
\begin{equation}\label{udrea}
\big(\mathcal V^\tau_{u,v}\big)^{\#_\tau}=\mathcal V^\tau_{v,u}\,.
\end{equation}
\end{Corollary}

\begin{proof}
The first identity is a consequence of the first point of Theorem \ref{carnat}:
$$
\begin{aligned}
{\sf Op}^\tau\!\!\(\mathcal V^\tau_{u_1,v_1}\#_\tau\,\mathcal V^\tau_{u_2,v_2}\)&={\sf Op}^\tau\!\!\(\mathcal V^\tau_{u_1,v_1}\){\sf Op}^\tau\!\!\(\mathcal V^\tau_{u_2,v_2}\)\\
&=\Lambda_{u_1,v_1}\Lambda_{u_2,v_2}\\
&=\<v_2,u_1\>\Lambda_{u_2,v_1}\\
&=\<v_2,u_1\>{\sf Op}^\tau\!\!\(\mathcal V^\tau_{u_2,v_1}\),
\end{aligned}
$$
which implies \eqref{vizi} because ${\sf Op}^\tau$ is linear and injective.

\medskip
The relation \eqref{udrea} follows similarly, taking into account the identity $\Lambda_{u,v}^*=\Lambda_{v,u}$\,.
\end{proof}

\begin{Remark}\label{cicau}
{\rm It seems convenient to summarise the situation in the following commutative diagram of unitary mappings (which are even isomorphisms of $H^*$-algebras):
$$
\begin{diagram}
\node{L^2(\G)\otimes L^2(\G)} \arrow{e,t}{{\sf id}\otimes{\fscr F}} \arrow{s,l}{{\fscr F}\otimes{\sf id}}\arrow{se,r}{{\sf Sch}^\tau}\node{L^2(\G)\otimes\mathscr B^2(\wG)} \arrow{s,l}{{\sf Op}^\tau} \node{\mathscr B^2(\wG)\otimes L^2(\G)}\arrow{w,t}{{\fscr F}^{-1}\otimes{\fscr F}} \\ 
\node{\mathscr B^2(\wG)\otimes L^2(\G)} \arrow{e,b}{{\sf Po}^\tau} \node{\mathbb B^2\big[L^2(\G)\big]}\node{\overline{L^2(\G)}\otimes L^2(\G)}\arrow{w,b}{\Lambda}\arrow{n,r}{\mathcal W^\tau}\arrow{nw,t}{\mathcal V^\tau}
\end{diagram}
$$
For completeness and for further use we also included two new maps. The first one is given by the formula ${\sf Po}^\tau\!:={\sf Op}^\tau\!\circ\big({\fscr F}^{-1}\otimes{\fscr F}\big)$ and it is the integrated form of the family of operators $\big\{ W^\tau(x,\xi)\mid (x,\xi)\in\G\times\wG\,\big\}$\,, defined formally by 
\begin{equation}\label{wisic}
{\sf Po}^\tau\!(\mathfrak a):=\int_{\G}\int_{\wG}\,{\rm Tr}_\xi\big[\mathfrak a(\xi,x)W^\tau(\xi,x)^*\big]d\m(x)d\wm(\xi)\,.
\end{equation}
Here we can think that $\mathfrak a=\big({\fscr F}\otimes{\fscr F}^{-1}\big) a$\,.
It is treated rigorously in the same way as ${\sf Op}^\tau$; the correct weak definition is to set for $u,v\in L^2(\G)$
\begin{equation}\label{starpor}
\big\<{\sf Po}^\tau\!(\mathfrak a) u,v\big\>_{L^2(\G)}=\big\<{\sf Op}^\tau\big[\big({\fscr F}^{-1}\otimes{\fscr F}\big)(\mathfrak a)\big] u,v\big\>_{L^2(\G)}=\big\<\mathfrak a,\mathcal W^\tau_{u,v}\big\>_{\!\mathscr B^2(\widehat\Gamma)}\,.
\end{equation} 
The second one is the Schr\"odinger representation ${\sf Sch}^\tau\!:={\sf Int}\circ{\rm CV}^\tau$ defined for $\,\Phi\in L^2(\G\times\G)$ by
\begin{equation}\label{radar}
\[{\sf Sch}^\tau\!(\Phi)v\]\!(x):=\int_\G\!\Phi\!\left(x\tau(y^{-1}x)^{-1}\!,y^{-1}x\right)\!v(y)\,d\m(y)\,. 
\end{equation}
It satisfies $\,{\sf Op}^\tau\!={\sf Sch}^\tau\!\circ\big({\sf id}\otimes{\fscr F}^{-1}\big)$ and we put it into evidence because it is connected to the $C^*$-algebraic formalism described in Subsection \ref{firea}.
}
\end{Remark}

%-------------------------------------------------------------------------------------------------------
\subsection{Non-commutative $L^p$-spaces and Schatten classes}\label{tontolin}
%-------------------------------------------------------------------------------------------------------

\begin{Definition}\label{gogonatele}
For $p\in[1,\infty)$ we introduce the Banach space $\mathscr B^{p,p}(\widehat\Gamma):=L^p\!\big[\G;\mathscr B^p(\wG)\big]$ with the norm
\begin{equation*}\label{scarbos}
\begin{aligned}
\p\!\mathfrak a\!\p_{\mathscr B^{p,p}(\widehat\Gamma)}\,&:=\Big(\int_\G \p\!\mathfrak a(x)\!\p^p_{\mathscr B^{p}(\wG)}\!d\m(x)\Big)^{1/p}\!\\
&=\Big(\int_\G \Big[\int_{\wG}\p\!\mathfrak a(\xi,x)\!\p^p_{\mathbb B^{p}(\H_\xi)}\!d\wm(\xi)\Big]d\m(x)\Big)^{1/p}\!,
\end{aligned}
\end{equation*}
where the convenient notation $\,\mathfrak a(\xi,x):=[\mathfrak a(x)](\xi)$ has been used.
\end{Definition} 

Note that $\mathscr B^{1,1}(\widehat\Gamma)\cong \mathscr B^1(\wG)\,\overline\otimes\,L^1(\G)$ (projective completed tensor product), while $\mathscr B^{2,2}(\widehat\Gamma)\cong\mathscr B^{2}(\widehat\Gamma)=\mathscr B^2(\wG)\otimes L^2(\G)\,$ (Hilbert tensor product). The double index indicates that the spaces $\mathscr B^{p,q}(\widehat\Gamma):=L^p\!\big[\G;\mathscr B^q(\wG)\big]$ could also be taken into account for $p\ne q$\,.

\medskip
To put the definition in a general context, we recall some basic facts about non-commutative $L^p$-spaces \cite{PX,Xu}. {\it A non-commutative measure space} is a pair $(\mathscr M,\mathcal T\,)$ formed of a von Neumann algebra $\mathscr M$ with positive cone $\mathscr M_+$\,, acting in a Hilbert space $\K$\,, endowed with a normal semifinite faithful trace $\,\mathcal T:\mathscr M_+\to[0,\infty]$\,. One defines  
\begin{equation*}\label{span}
\mathscr S_+:=\{m\in\mathscr M_+\mid \mathcal T[s(m)]<\infty\}\,,
\end{equation*}
where $s(m)$ is {\it the support} of $m$\,, i.e. the smallest orthogonal projection $e\in\mathscr M$ such that $eme=m$\,. Then $\mathscr S$, defined to be the linear span of $\mathscr S_+$\,, is a w$^*$-dense $^*$-subalgebra of $\mathscr M$\,. For every $p\in[1,\infty)$\,, the map $\p\!\cdot\!\p_{(p)}:\mathscr S\rightarrow[0,\infty)$ given by
\begin{equation*}\label{insulting}
\p\!m\!\p_{(p)}\,:=\big[\mathcal T\big(|m|^p\big)\big]^{1/p}=\big[\mathcal T\big((m^*m)^{p/2}\big)\big]^{1/p}
\end{equation*}
is a well-defined norm. The completion of $\big(\mathscr S,\p\!\cdot\!\p_{(p)}\!\!\big)$ is denoted by $\mathscr L^p(\mathscr M,\mathcal T)$ and is called {\it the non-commutative $L^p$-space associated to the non-commutative measure space $(\mathscr M,\mathcal T)$\,.} The scale is completed by setting $\,\mathscr L^\infty(\mathscr M,\mathcal T):=\mathscr M$. It can be shown that $\mathscr L^1(\mathscr M,\mathcal T)$ can be viewed as the predual of $\mathscr M$ and the elements of $\mathscr L^p(\mathscr M,\mathcal T)$ can be interpreted as closed, maybe unbounded, operators in $\K$ \cite{PX}.

\medskip
We are going to need two important properties of these non-commutative $L^p$-spaces.
\begin{itemize}
\item
{\it Duality}: if $p\ne\infty$ and $1/p+1/p'=1$\,, then $\big[\mathscr L^p(\mathscr M,\mathcal T)\big]^*\cong \mathscr L^{p'}(\mathscr M,\mathcal T)$ isometrically; the duality is defined by $\,\<m,n\>_{(p),(p')}\!:=\mathcal T(mn^*)$ (consequence of a non-commutative H\"older inequality).
\item
{\it Interpolation}: the complex interpolation of these spaces follows the rule
\begin{equation*}\label{interpolation}
\big[\mathscr L^{p_0}(\mathscr M,\mathcal T),\mathscr L^{p_1}(\mathscr M,\mathcal T)\big]_\th=\mathscr L^{p}(\mathscr M;\mathcal T)\,,\quad\th\in(0,1)\,,\ \ \frac{1}{p}=\frac{1-\th}{p_0}+\frac{\th}{p_1}\,.
\end{equation*}
\end{itemize}

In our case the non-commutative measure space can be defined as follows: The von Neumann algebra is 
\begin{equation*}
\mathscr B^{\infty,\infty}(\widehat\Gamma)= \mathscr B(\wG)\,\widetilde\otimes\,L^\infty(\G)= \int^\oplus_{\wG}\!\mathbb B(\H_\xi)d\wm(\xi)\,\widetilde\otimes\,L^\infty(\G)
\end{equation*} 
(weak$^*$-completion of the algebraic tensor product). Denoting as before by ${\rm Tr}_\xi$ the standard trace in $\mathbb B(\H_\xi)$\,, then on $\mathscr B(\wG\,)$ one has \cite[Sect II.5.1]{Dix} the direct integral trace $\,{\rm Tr}:=\int^\oplus_{\wG}{\rm Tr}_\xi\,d\wm(\xi)$ and on $\mathscr B(\wG)\,\widetilde\otimes\,L^\infty(\G)$ the tensor product \cite[1.7.5]{Xu} $\,\mathcal T:={\rm Tr}\otimes\int_\G$ of $\,{\rm Tr}$ with the trace given by Haar integration in the commutative von Neumann algebra $L^\infty(\G)$\,. Thus one gets the non-commutative measure space $\big(\mathscr B^{\infty,\infty}(\widehat\Gamma),\mathcal T\big)$\,. It is not difficult to show that the associated non-commutative $L^p$-spaces are the Banach spaces $\mathscr B^{p,p}(\widehat\Gamma)$ introduced in Definition \ref{gogonatele} (\cite[1.7.5]{Xu} is useful again).
In particular, we have the following rule of complex interpolation:
\begin{equation*}\label{interpolation}
\big[\mathscr B^{p_0,p_0}(\widehat\Gamma),\mathscr B^{p_1,p_1}(\widehat\Gamma)\big]_\th=\mathscr B^{p,p}(\widehat\Gamma)\,,\quad\th\in(0,1)\,,\ \frac{1}{p}=\frac{1-\th}{p_0}+\frac{\th}
{p_1}\,.
\end{equation*}

On the other hand, the Schatten-von Neumann ideals $\,\mathbb B^p\!\big[L^2(\G)\big]$ are the non-commutative $L^p$-spaces associated to the non-commutative measure space $\big(\mathbb B\big[L^2(\G)\big],{\sf Tr}\big)$\,. So they interpolate in the same way.

\begin{Proposition}\label{bp}
For every $p\in[2,\infty]$ one has a linear contraction
\begin{equation}\label{lc}
\W^\tau:\overline{L^2(\G)}\otimes L^2(\G)\rightarrow\mathscr B^{p,p}(\widehat\Gamma)\,.
\end{equation}
\end{Proposition}

\begin{proof}
We have seen in Proposition \ref{mataus} that $\W^\tau$ is unitary if $p=2$\,. If we also check the case $p=\infty$\,, then \eqref{lc} follows by complex interpolation. But the uniform estimate
\begin{equation*}\label{est}
\p\!\W^\tau_{u,v}(\xi,x)\!\p_{\mathbb B(\H_\xi)}\,\le\,\p\!u\!\p_{L^2(\G)}\p\!v\!\p_{L^2(\G)}
\end{equation*}
is an immediate consequence of \eqref{suspinare} and of the unitarity of $W^\tau\!(\xi,x)$ in $L^2(\G;\H_\xi)$\,.
\end{proof}

\medskip
For the next two results we switch our interest from ${\sf Op}^\tau$ to ${\sf Po}^\tau$, given by \eqref{starpor}, since for such general groups $\G$ there is no inversion formula for the Fourier transform at the level of the non-commutative $L^p$-spaces (the Hausdorff-Young inequality cannot be used for our purposes).

\begin{Theorem}\label{tomas}
If $\,\mathfrak a\in \mathscr B^{1,1}(\widehat\Gamma)=L^1\big[\G;\mathscr B^1(\wG)\big]$ then ${\sf Po}^\tau(\mathfrak a)$ is bounded in $L^2(\G)$ and
\begin{equation*}\label{lungenstrass}
\big\Vert{\sf Po}^\tau\!(\mathfrak a)\big\Vert_{\mathbb B[L^2(\G)]}\le\,\p\!\mathfrak a\p_{\mathscr B^{1,1}(\widehat\Gamma)}\,.
\end{equation*}
\end{Theorem}

\begin{proof}
One modifies \eqref{starpor} to a similar definition by duality
\begin{equation*}\label{starport}
\big\<{\sf Po}^\tau\!(\mathfrak a) u,v\big\>_{L^2(\G)}=\big\<\mathfrak a,\mathcal W^\tau_{u,v}\big\>_{(1),(\infty)}:=\mathcal T\Big(\mathfrak a\,\big[\mathcal W^\tau_{u,v}\big]^*\Big)\,,
\end{equation*} 
based on the case $p=\infty$ of Proposition \ref{bp} and on the duality of the non-commutative Lebesgue spaces.
\end{proof}

By using complex interpolation between the end points $p_0=2$ and $p_1=\infty$\,, one gets

\begin{Corollary}\label{thomas}
If $\,p\in[1,2]$\,, $\frac1p+\frac{1}{p'}=1$ and $\,\mathfrak a\in L^p\big[\G;\mathscr B^p(\wG)\big]$\,, then $\,{\sf Po}^\tau\!(\mathfrak a)$ belongs to $\mathbb B^{p'}\big[L^2(\G)\big]$ and
\begin{equation*}\label{lungenstrass}
\big\Vert{\sf Po}^\tau\!(\mathfrak a)\big\Vert_{\mathbb B^{p'}[L^2(\G)]}\le\;\p\!\mathfrak a\!\p_{\mathscr B^{p,p}(\widehat\Gamma)}.
\end{equation*}
\end{Corollary}

More refined results follow from real interpolation; the interested reader could write them down easily.

%-------------------------------------------------------------------------------------------------------
\section{Symmetric quantizations}\label{fourtinin}
%-------------------------------------------------------------------------------------------------------

Having in mind the well-known \cite{Fo} Weyl quantization, we inquire about the existence of a parameter $\tau$ allowing a symmetric quantization; if it exists, for emphasis, we denote it by $\si$\,. By definition, this means that $a^{\#_\si}=a^\star$ for every $a\in\mathscr B^2(\Gamma)$\,, where of course $a^\star(x,\xi):=a(x,\xi)^*$ (adjoint in $\mathbb B(\H_\xi)$) for every $(x,\xi)\in\Gamma$\,. At the level of pseudo-differential operators the consequence would be the simple relation ${\sf Op}^\si\!(a)^*={\sf Op}^\si\!(a^\star)$\,, so ``real-valued symbols" are sent into self-adjoint operators. 

%---------------------------------------------------------------------------------------------------------
\subsection{An explicit form for the adjoint}\label{bruqart}
%-----------------------------------------------------------------------------------------------------------

In order to study symmetry it is convenient to give a more explicit form of the involution \eqref{invo}; we need to alow different values of the parameter $\tau$\,. For any measurable map $\tau:\G\rightarrow\G$\,, let us define
\begin{equation}\label{vibrion}
\widetilde\tau:\G\rightarrow\G\,,\quad\widetilde\tau(x):=\tau\big(x^{-1}\big)\,x\,.
\end{equation}

It is worth mentioning that if $\tau(\cdot)={\sf e}$ then $\widetilde\tau={\rm id}_\G$ and if $\,\tau={\rm id}_\G$ then $\,\widetilde\tau(\cdot)={\sf e}$\,.
In addition $\,\widetilde{\widetilde\tau}=\tau$ holds.

\medskip
If $\G=\R^n$ and $\tau:=t\,{\rm id}_{\R^n}$ with $t\in[0,1]$\,, one has $\widetilde\tau=(1-t){\rm id}_{\R^n}$ and the next proposition is well-known for pseudo-differential operators on $\R^n$.

\begin{Proposition}\label{simfunia}
For every $a\in\mathscr B(\Gamma)$ one has ${\sf Op}^{\tau}\!(a)^*={\sf Op}^{\widetilde\tau}\!(a^\star)$\,.
\end{Proposition}

\begin{proof}
Hoping that ${\sf Op}^{\tau}\!(a)^*={\sf Op}^{\tau'}\!(a^\star)$ for some $\tau':\G\to\G$\,, by \eqref{zaruz}, one has to examine the equality
\begin{equation*}\label{sdfa}
\Big(\big[{\rm CV}^\tau\!\circ({\sf id}\otimes{\fscr F}^{-1})\big]a\Big)^{\!\bullet}=\big[{\rm CV}^{\tau'}\!\!\circ({\sf id}\otimes{\fscr F}^{-1})\big]a^\star.
\end{equation*}
This and the next identities should hold almost everywhere with respect to the product measure $\m\otimes\m$\,.
Using the easy relation 
\begin{equation*}\label{intersvicia}
[({\sf id}\otimes{\fscr F}^{-1})a^\star](y,z)=\overline{[({\sf id}\otimes{\fscr F}^{-1})a](y,z^{-1})}\,,
\end{equation*}
one gets immediately
\begin{equation}\label{sdga}
\Big(\big[{\rm CV}^{\tau'}\!\!\circ({\sf id}\otimes{\fscr F}^{-1})\big]a^\star\Big)(y,z)=\overline{[({\sf id}\otimes{\fscr F}^{-1})a]\big(y\tau'(z^{-1}y)^{-1},y^{-1}z\big)}\,.
\end{equation}
On the other hand
\begin{equation}\label{sdha}
\begin{aligned}
\Big(\big[{\rm CV}^\tau\!\circ({\sf id}\otimes{\fscr F}^{-1})\big]a\Big)^{\!\bullet}(y,z)&=\overline{\Big(\big[{\rm CV}^\tau\circ({\sf id}\otimes{\fscr F}^{-1})\big]a\Big)(z,y)}\\
&=\overline{\big[({\sf id}\otimes{\fscr F}^{-1})a\big]\big(z\tau(y^{-1}z)^{-1},y^{-1}z\big)}
\end{aligned}
\end{equation}
and the two expressions \eqref{sdga} and \eqref{sdha} always coincide $\m\otimes\m$-almost everywhere if and only if
\begin{equation*}\label{sdja}
y\tau'(z^{-1}y)^{-1}=z\tau(y^{-1}z)^{-1}\,,\quad \m\otimes\m-{\rm a.e}.\;(y,z)\in\G\times\G\,.
\end{equation*}
This condition can be transformed into
\begin{equation*}\label{sdjara}
\tau'(z^{-1}y)=\tau\big(y^{-1}z\big)z^{-1}y\,,\quad \m\otimes\m-{\rm a.e}.\;(y,z)\in\G\times\G\,,
\end{equation*}
which must be shown to be equivalent to $\tau'=\widetilde\tau\,$ holding $\m$-almost everywhere. 

This follows if we prove that $A\subset\G$ is $\m$-negligible if and only if $M(A):=\{(y,z)\in\G\times\G\mid z^{-1}y\in A\}$ is $\m\otimes\m$-negligible. Since $\m$ is $\si$-finite, there is a Borel partition $\G=\sqcup_{n\in\N}B_n$ with $\m(B_n)<\infty$ for every $n\in\N$\,. Thus $M(A)=\sqcup_{n\in\N}M_n(A)$\,, with 
$$
M_n(A):=\{(y,z)\in M(A)\mid z\in B_n\}=\{(y,z)\in \G\times B_n\mid y\in zA\}\,.
$$
Using the invariance of the Haar measure, one checks that $(\m\otimes\m)\big[M_n(A)\big]=\m(A)\m(B_n)$ and the conclusion follows easily.
\end{proof}

%---------------------------------------------------------------------------------------------------------
\subsection{Symmetry functions}\label{briqart}
%-----------------------------------------------------------------------------------------------------------

The measurable function $\si:\G\to\G$ is called {\it a symmetry function} if one has ${\sf Op}^\si\!(a)^*={\sf Op}^\si\!(a^\star)$ for every $a\in\mathscr B^2(\Gamma)$\,. When $\G$ is admissible and a symmetry function exists we say that {\it the group $\G$ admits a symmetric quantization}. As a consequence of Proposition \ref{simfunia} one gets

\begin{Corollary}\label{simfuna}
The map $\si:\G\rightarrow\G$ is a symmetry function if and only if for almost every $x\in\G$
\begin{equation}\label{zigrot}
\si(x)=\si(x^{-1})\,x\,.
\end{equation}
In particular, if $\si$ is a symmetry function and $a(\cdot,\cdot)\in\mathscr B^2(\Gamma)$ is self-adjoint pointwise (or $\m\otimes\wm$-almost everywhere) then  ${\sf Op}^\si\!(a)$ is a self-adjoint operator in $L^2(\G)$\,.
\end{Corollary}

The problem of existence of $\si$ satisfying \eqref{zigrot} seems rather obscure in general, so we only treat some particular cases. 

\begin{Proposition}\label{zponentialaa}
\begin{enumerate}
\item
The product $\,\G:=\prod_{k=1}^m\!\G_k$ of a family of groups admitting a symmetric quantization also admits a symmetric quantization. 
\item
The admissible central extension of a group admitting a symmetric quantization by another group with this property is a group admitting a symmetric quantization.
\item
Any exponential Lie group (in particular any connected simply connected nilpotent group) admits a symmetric quantization.
\end{enumerate}
\end{Proposition}

\begin{proof}
1. Finite products of admissible groups are admissible. If $\si_k$ is a symmetry function for $\G_k$\,, then $\si\big[(x_k)_k\big]:=\big(\si_k(x_k)\big)_k$ defines a symmetry function for $\G$\,.

\medskip
2. The structure of central group extensions can be described in terms of $2$-cocycles up to canonical isomorphisms. Let ${\sf N}$ be an Abelian locally compact group, ${\sf H}$ a locally compact group and $\varpi:{\sf H}\times{\sf H}\rightarrow{\sf N}$ a $2$-cocycle. On $\G:={\sf H}\times{\sf N}$ one has the composition law and the inversion
\begin{equation}\label{monggola}
(h_1,n_1)(h_2,n_2):=(h_1 h_2,n_1 n_2\varpi(h_1,h_2))\,,\quad(h,n)^{-1}:=\big(h^{-1},\varpi(h^{-1},h)^{-1}n^{-1}\big)\,.
\end{equation}
The properties of $\varpi$ are normalisation $\varpi(h,\e_{\sf H})=\e_{\sf N}=\varpi(\e_{\sf H},h)\,,\ \forall\,h\in{\sf H}$\,,
and the $2$-cocycle identity
\begin{equation}\label{aleanna}
\varpi(h_1,h_2)\varpi(h_1h_2,h_3)=\varpi(h_2,h_3)\varpi(h_1,h_2h_3)\,,\quad\forall\,h_1,h_2,h_3\in{\sf H}\,.
\end{equation}
We are given symmetry functions $\si_{\sf H}:{\sf H}\rightarrow{\sf H}$ and $\si_{\sf N}:{\sf N}\rightarrow{\sf N}$ thus satisfying
\begin{equation*}\label{tatara}
\si_{\sf H}(h)=\si_{\sf H}(h^{-1})\,h\,,\quad \si_{\sf N}(n)=\si_{\sf N}(n^{-1})\,n=n\,\si_{\sf N}(n^{-1})\,.
\end{equation*}
We define the measurable map $\,\si:\G\rightarrow\G$ by
\begin{equation}\label{hunna}
\begin{aligned}
\si(h,n):&=\Big(\si_{\sf H}(h),\si_{\sf N}\big[\varpi\big(\si_{\sf H}(h^{-1}),h\big)n\big]\Big)\\
&=\Big(\si_{\sf H}(h^{-1})\,h,\varpi\big(\si_{\sf H}(h^{-1}),h\big)n\,\si_{\sf N}\big[\varpi\big(\si_{\sf H}(h^{-1}),h\big)^{-1}n^{-1}\big]\Big)\,;
\end{aligned}
\end{equation}
the second line has been deduced from the first by using the properties of $\si_{\sf H}$ and $\si_{\sf N}$ and the fact that ${\sf N}$ is commutative.
We compute using \eqref{monggola} and \eqref{hunna}
\begin{equation}\label{vichingga}
\begin{aligned}
&\si\big((h,n)^{-1}\big)\,(h,n)=\si\big(h^{-1},\varpi(h^{-1},h)^{-1}n^{-1}\big)\,(h,n)\\
&=\Big(\si_{\sf H}(h^{-1}),\si_{\sf N}\big[\varpi\big(\si_{\sf H}(h),h^{-1}\big)\varpi(h^{-1},h)^{-1}n^{-1}\big]\Big)\,(h,n)\\
&=\Big(\si_{\sf H}(h^{-1})\,h,n\,\si_{\sf N}\big[\varpi\big(\si_{\sf H}(h),h^{-1}\big)\varpi(h^{-1},h)^{-1}n^{-1}\big]\varpi\big(\si_{\sf H}(h^{-1}),h\big)\Big)\,.
\end{aligned}
\end{equation}
The first components in \eqref{hunna} and \eqref{vichingga} are equal. The second ones would coincide if one shows
\begin{equation*}\label{franca}
\varpi\big(\si_{\sf H}(h^{-1}),h\big)^{-1}=\varpi\big(\si_{\sf H}(h),h^{-1}\big)\varpi(h^{-1},h)^{-1}\,,
\end{equation*}
which is equivalent to
\begin{equation}\label{franka}
\varpi\big(\si_{\sf H}(h),h^{-1}\big)\varpi\big(\si_{\sf H}(h^{-1}),h\big)=\varpi\big(h^{-1},h\big)\,.
\end{equation}
Taking in \eqref{aleanna} $\,h_3=h_2^{-1}=h$ and using the normalization of $\varpi$ one gets
\begin{equation*}\label{tatarara}
\varpi(h_1,h^{-1})\varpi(h_1 h^{-1}\!,h)=\varpi(h^{-1}\!,h)\,,\quad\forall\,h_3\in{\sf H}\,.
\end{equation*}
Choosing $h_3:=\si_{\sf H}(h)$ we get \eqref{franka}, because $\si_{\sf H}$ is a symmetry function.

\medskip
3. Assume that $\G$ is a Lie group with Lie algebra $\mathfrak g$\,. It is known that the exponential map $\exp:\mathfrak g\rightarrow\G$ restricts to a diffeomorphism $\exp:\mathfrak u\rightarrow{\sf U}$\,, where $\mathfrak u$ is a neighborhood of $0\in\mathfrak g$ (a ball centered in the origin, for example) and ${\sf U}$ is a neighborhood of $\e\in\G$\,. The inverse diffeomorphism is denoted by $\log:{\sf U}\rightarrow\mathfrak u$\,. 
One defines ``the midpoint mapping"
\begin{equation}\label{clantaa}
\si:{\sf U}\rightarrow\G\,,\quad\si(x):=\int_0^1\!\exp[s\log x]ds\,.
\end{equation}
We claim that
\begin{equation}\label{dutaa}
x\,\si(x^{-1})=\si(x)=\si(x^{-1})\,x\,,\quad\forall\,x\in\G\,.
\end{equation}
We prove the second equality; the first one is similar:
$$
\begin{aligned}
\si(x^{-1})\,x&=\int_0^1\!\!\exp\!\big[s\log(x^{-1})\big]ds\,x\\
&=\int_0^1\!\!\exp\!\big[-s\log(x)\big]ds\,\exp\!\big[\log(x)\big]\\
&=\int_0^1\!\!\exp\!\big[(1-s)\log(x)\big]ds\\
&=\int_0^1\!\!\exp\!\big[s\log(x)\big]ds=\si(x)\,.
\end{aligned}
$$
An exponential group $\G$ is, by definition, a Lie group for which one can take $\mathfrak u=\mathfrak g$ and ${\sf U}=\G$ (i.e. the exponential map is a global diffeomorphism), case in which the symmetry function $\si$ is globally defined. In addition $\G$ is admissible. Connected simply connected nilpotent group are exponential, by \cite[Th. 1.2.1]{CG}. 
\end{proof}

\begin{Example}\label{iarsaa}
{\rm For $\G=\R^n$ one just sets $\si(x):=x/2$ (getting finally the Weyl quantization). 

If $\si$ is required to satisfy $\si(\cdot^{-1})=\si(\cdot)^{-1}$ (e.g. being an endomorphism), 
\eqref{zigrot} reduces to $x=\si(x)^2$ for almost every $x$\,, which is equivalent to ``any" element in $\G$ having a square root; such a group does admit a symmetric quantization. This fails for many groups, as $\G:=\mathbb Z^n$ for instance. 

It is easy to see that the torus $\mathbb T$ has a measurable symmetry function given by the group endomorphism defined essentially as $\si(e^{2\pi it}):=e^{\pi it}$. It is not even continuous and this is an obvious drawback on such a basic Lie group, on which one has great expectations from a pseudo-differential calculus (cf. \cite{RT,RT2} and references therein for the rich theory in the case $\tau(x)=1$).

On the other hand, the symmetry function of an exponential Lie group is smooth.
}
\end{Example}

%-------------------------------------------------------------------------------------------------------
\section{Extension to distributions}\label{fergulin}
%-------------------------------------------------------------------------------------------------------

Having started with the formalism involving symbols in $\mathscr B^2(\Gamma)$ and operators bounded on $L^2(\G)$, 
it is useful to be able to extend it to e.g. unbounded symbols and to operators which are only densely defined on $L^2(\G)$.
If the group $\G$ is a Lie group (has a smooth manifold structure), we can, for example, think of operators
acting from the space of test functions to the space of distributions, or of operators having e.g. polynomial growth of symbols. 
Without assuming that $\G$ is a Lie group we do not have the usual space of smooth compactly supported functions
readily available as the standard space of test functions. So we will be using its generalisation to
locally compact groups by Bruhat \cite{Br}, and these Bruhat spaces $\mathcal D(\G)$ and $\mathcal D'(\G)$ will
replace the usual spaces of test functions and distributions, respectively, in our setting.

%---------------------------------------------------------------------------------------------------------
\subsection{Smooth functions and distributions -- Bruhat spaces }\label{broqart}
%-----------------------------------------------------------------------------------------------------------

The Bruhat spaces $\mathcal D(\G)$ and $\mathcal D'(\G)$ have been introduced in \cite{Br}, to which we refer for further details. Most of their properties hold for every locally compact group, but in some cases second countability is also used. Postliminarity, unimodularity or amenability are not needed.

\medskip
{\it A good subgroup of $\,\G$} is a compact normal subgroup ${\sf H}$ of $\G$ such that $\G/{\sf H}$ is (isomorphic to) a Lie group. The family of all good subgroups of $\G$ will be denoted by ${\sf good}(\G)$\,; it is stable under intersections. We are going to assume first that 
\begin{equation}\label{good}
\bigcap_{{\sf H}\in{\sf good}(\G)}\!\!\!{\sf H}=\{\e\}\,.
\end{equation} 
Denoting the connected component of the identity by $\G_0$\,, this happens, for instance, if $\G/\G_0$ is compact, in particular if $\G$ is connected. If \eqref{good} holds, every neighborhood of $\e$ contains an element of ${\sf good}(\G)$\,. Then $\G$ can be seen as the projective limit of the projective family of Lie groups 
\begin{equation*}\label{project}
\big\{\G/{\sf H}\to\G/{\sf K}\mid {\sf H},{\sf K}\in{\sf good}(\G)\,, {\sf K}\subset{\sf H}\big\}\,.
\end{equation*}
For every good group ${\sf H}$ one sets $\mathcal D(\G/{\sf H}):=\mathcal C^\infty_{\rm c}(\G/{\sf H})$ with the usual inductive limit topology. Functions on quotients are identified with invariant functions on the group by the map $j_{\sf H}(v):=v\circ q_{\sf H}$\,, where $q_{\sf H}:\G\rightarrow\G/{\sf H}$ is the  canonical surjection. Thus on 
\begin{equation*}\label{bruhat}
\mathcal D_{\sf H}(\G):=j_{\sf H}[\mathcal D(\G/{\sf H})]\subset \mathcal C_c(\G)\subset\mathcal C_0(\G)
\end{equation*}
one can transport the topology of $\mathcal D(\G/{\sf H})$\,. 

\begin{Definition}\label{bruholt}
The Bruhat space $\mathcal D(\G)$ of the locally compact group $\G$ is the topological inductive limit of the family of subspaces $\{\mathcal D_{\sf H}(\G)\mid {\sf H}\in{\sf good}(\G)\}$ of $\,\mathcal C_c(\G)$\,.

The strong dual of $\mathcal D(\G)$ is denoted by $\mathcal D'(\G)$\,; it contains $\mathcal D(\G)$ densely. Its elements are called {\it distributions}.
\end{Definition} 

The space $\mathcal D(\G)$ is barrelled and bornological. It is continuously and densely contained in $\mathcal C_c(\G)$ and complete. If $\G$ is already a Lie group, then $\{\e\}\in{\sf good}(\G)$ and clearly $\mathcal D(\G)=\mathcal C_{\rm c}^\infty(\G)$\,.

The spaces $\mathcal D(\G)$ and $\mathcal D'(\G)$ are complete and Montel (thus reflexive) as well as nuclear.

\medskip
We presented above the case in which our locally compact group satisfies $\,\bigcap_{{\sf H}\in{\sf good}(\G)}{\sf H}=\{\e\}$\,. Following \cite[Sect. 2]{Br}, we briefly indicate what to do without this assumption.

There exists an open subgroup $\G_1$ of $\G$ such that $\,\bigcap_{{\sf H_1}\in{\sf good}(\G_1)}{\sf H_1}=\{\e\}$\,. Thus the space $\mathcal D(\G_1)\subset \mathcal C_{\rm c}(\G)$ is available. The group $\G$ is partitioned in classes modulo $\G_1$ to the left: $\G=\bigcup x\G_1$ or to the right: $\G=\bigcup\G_1 x$\,.
By left translations one generates the subspaces $\mathcal D(x\G_1)$ of $\mathcal C_{\rm c}(\G)$ (with the transported topology); the elements are particular types of continuous functions on $\G$ compactly supported in $x\G_1$\,. Then define $\mathcal D^L(\G)\subset \mathcal C_{\rm c}(\G)$ to be the topological direct sum $\mathcal D^L(\G):=\bigoplus\mathcal D(x\G_1)$\,. It comes out that the similarly constructed $\mathcal D^R(\G):=\bigoplus\mathcal D(\G_1 x)$ is the same subspace of $\mathcal C_{\rm c}(\G)$ with the same topology. In addition, it does not depend on the choice of the open subgroup $\G_1$ so it deserves the notation $\mathcal D(\G)$\,. 

Then the construction of the space $\mathcal D'(\G)$ follows similarly and all the nice properties mentioned above still hold (cf. \cite{Br}). The main reason is the fact that topological direct sums are rather easy to control. Subsequently we will have recourse to these Bruhat spaces in the general case.

\medskip
We are going to use the symbol $\overline{\otimes}$ for the projective tensor product of locally compact spaces; note, however, that this will only be applied to spaces known to be nuclear. By the Kernel Theorem for Bruhat spaces \cite[Sect. 5]{Br} one has 
\begin{equation*}\label{inseninare}
\mathcal D(\G\times\G)\cong\mathcal D(\G)\,\overline\otimes\,\mathcal D(\G)\subset L^2(\G\times\G)
\end{equation*}
continuously and densely. Soon we are going to need the next result:

\begin{Lemma}\label{cadrat}
The mapping 
\begin{equation}\label{cadrilat}
{\rm CV}:\mathcal D(\G\times\G)\rightarrow\mathcal D(\G\times\G)\,,\quad [{\rm CV}(\Psi)](x,y):=\Psi[{\rm cv}(x,y)]=\Psi\big(x,y^{-1}x\big)
\end{equation}
is a well-defined topological isomorphism. Its inverse ${\rm CV}^{-1}$ is the operation of composing with
\begin{equation*}\label{inbersa}
{\rm cv}^{-1}:\G\times\G\rightarrow\G\times\G\,,\quad {\rm cv}^{-1}(x,y):=\big(x,xy^{-1}\big)\,.
\end{equation*}
By transposing the inverse one gets a topological isomorphism 
\begin{equation*}\label{transinv}
{\rm CV}:=\big[{\rm CV}^{-1}\big]^{\sf tr}:\mathcal D'(\G\times\G)\rightarrow\mathcal D'(\G\times\G)\,,
\end{equation*}
which is an extension of the one given in \eqref{cadrilat} (this explains the notational abuse).
\end{Lemma}

\begin{proof}
The proof is quite straightforward, but rather long if all the details are included, so it is essentially left to the reader. Besides using the definitions and the standard tools of duality, one should also note the following:
\begin{itemize}
\item
If ${\sf H}$ is a good subgroup of $\G$\,, then ${\sf H}\times{\sf H}$ is a good subgroup of $\G\times\G$ and $(\G\times\G)/({\sf H}\times{\sf H})$ is canonically isomorphic to $(\G/{\sf H})\times(\G/{\sf H})$\,.
\item
For ${\sf H}\in{\sf good}(\G)$ there is {\it a Lie group isomorphism}
\begin{equation*}\label{aku}
{\rm cv}_{\sf H}:(\G/{\sf H})\times(\G/{\sf H})\rightarrow(\G/{\sf H})\times(\G/{\sf H})\,,
\end{equation*}
\begin{equation*}\label{aka}
{\rm cv}_{\sf H}(x{\sf H},y{\sf H}):=\big((x{\sf H}),(y{\sf H})^{-1}x{\sf H}\big)=\big(x{\sf H},y^{-1}x{\sf H}\big)
\end{equation*}
and related to the initial change of variables ${\rm cv}$ through
\begin{equation*}
{\rm cv}_{\sf H}\circ\big(q_{\sf H}\times q_{\sf H}\big)=\big(q_{\sf H}\times q_{\sf H}\big)\circ{\rm cv}\,.
\end{equation*}
This and the fact that ${\rm cv}_{\sf H}$ is a proper map easily allow us to conclude that $\,{\rm CV}:\mathcal D_{{\sf H}\times{\sf H}}(\G\times\G)\rightarrow\mathcal D_{{\sf H}\times{\sf H}}(\G\times\G)$ is a well-defined isomorphism for every good subgroup ${\sf H}$\,.
\item
Let $\G_1$ be a subgroup of $\G$\,; then ${\rm cv}$ carries $\G_1\times\G_1$ into itself isomorphically.
\item
Let $\G_1$ be an open subgroup of $\G$ such that $\,\bigcap_{{\sf H_1}\in{\sf good}(\G_1)}{\sf H_1}=\{\e\}$\,. Then the family $\big\{{\sf H}_1\times{\sf H}_1\mid {\sf H}_1\in{\sf good}(\G_1)\big\}$ is directed under inclusion and
\begin{equation*}\label{aqui}
\bigcap_{{\sf H_1}\in{\sf good}(\G_1)}\!\!{\sf H_1}\times{\sf H}_1=\{(\e,\e)\}\,.
\end{equation*}
\end{itemize}
\end{proof}

\begin{Remark}\label{ager}
{\rm Of course, the case $\tau(x)=x$ can be treated the same way. If one tries to do the same for the change of variables ${\rm cv}^\tau$, in general one encounters rather complicated conditions relating the map $\tau$ to the family ${\sf good}(\G)$\,. However, if $\G$ is a Lie group, ${\sf good}(\G)$ has a smallest element $\{{\sf e}\}$ and thus $\mathcal D(\G)$ coincides with $\mathcal C_{\rm c}^\infty(\G)$\,. Then it is easy to see that the statements of the lemma hold if ${\rm cv}^\tau$ is proper and $\,\tau:\G\rightarrow\G$ is a $\mathcal C^\infty$-function.
}
\end{Remark}

%-------------------------------------------------------------------------------------------------------
\subsection{Restrictions and extensions of the pseudo-differential calculus}\label{ferfelin}
%-------------------------------------------------------------------------------------------------------

Let us define $\mathscr D(\wG):={\fscr F}[\mathcal D(\G)]$ with the locally convex topological structure transported from the Bruhat space $\mathcal D(\G)$\,. One has $\mathcal D(\G)\subset\mathcal C_{\rm c}(\G)\subset L^2(\G)\cap A(\G)$ (continuously and densely), so $\mathscr D(\wG)$ is a dense subspace of $\mathscr B^2(\wG)\cap\mathscr B^1(\wG)$ (with the intersection topology) and of $\mathscr B^2(\wG)$\,. Thus the explicit form of the inverse \eqref{marian} holds on $\mathscr D(\wG)$\,. One also has 
\begin{equation*}\label{megazork-e}
u(\e)=\int_{\widehat{\G}}{\rm Tr}_\xi[({\fscr F}u)(\xi)]d\wm(\xi)\,,\quad\forall\,u\in\mathcal D(\G)\,.
\end{equation*}

We are going to use the dense subspace
\begin{equation*}\label{inserare}
\mathscr D\big(\Gamma\big)\equiv\mathscr D\big(\G\times\wG\big):=\mathcal D(\G)\,\overline\otimes\,\mathscr D(\wG)\subset\mathscr B^2(\Gamma)\,,
\end{equation*}
possessing its own locally convex topology, obtained by transport of structure and the completed projective tensor product construction.
Taking also into account the strong dual, one gets a Gelfand triple $\mathscr D\big(\Gamma\big)\hookrightarrow\mathscr B^2(\Gamma)\hookrightarrow\mathscr D'\big(\Gamma\big)$\,.

\begin{Proposition}\label{quadrat}
The pseudo-differential calculus $\,{\sf Op}:L^2(\G)\otimes\mathscr B^2(\wG)\rightarrow\mathbb B^2\big[L^2(\G)\big]$
\begin{itemize}
\item
restricts to a topological isomorphism $\,{\sf Op}:\mathcal D(\G)\,\overline\otimes\,\mathscr D(\wG)\rightarrow\mathbb B\big[\mathcal D'(\G);\mathcal D(\G)\big]$\,,
\item
extends to a topological isomorphism $\,{\sf Op}:\mathcal D'(\G)\,\overline\otimes\,\mathscr D'(\wG)\rightarrow\mathbb B\big[\mathcal D(\G);\mathcal D'(\G)\big]$\,.
\end{itemize}
\end{Proposition}

\begin{proof}
The proof can essentially be read in the diagrams
\begin{equation*}\label{patrat}
\begin{diagram}
\node{\mathcal D(\G)\,\overline\otimes\, \mathscr D(\wG)} \arrow{e,t}{\!{{\sf id}\overline\otimes\fscr F}^{-1}} \arrow{s,l}{{\sf Op}}\node{\mathcal D(\G)\,\overline\otimes\,\mathcal D(\G)\cong\mathcal D(\G\times\G)\quad\quad\quad\quad\quad\quad\quad}\arrow{s,r}{{\rm CV}}\\ 
\node{\mathbb B\big[\mathcal D'(\G);\mathcal D(\G)\big]} \node{\mathcal D(\G\times\G)} \arrow{w,b}{{\sf Int}}
\end{diagram}
\end{equation*}
and
\begin{equation*}\label{potrat}
\begin{diagram}
\node{\mathcal D'(\G)\,\overline\otimes\, \mathscr D'(\wG)} \arrow{e,t}{\!{{\sf id}\overline\otimes\fscr F}^{-1}} \arrow{s,l}{{\sf Op}}\node{\mathcal D'(\G)\,\overline\otimes\,\mathcal D'(\G)\cong\mathcal D'(\G\times\G)\quad\quad\quad\quad\quad\quad\quad}\arrow{s,r}{{\rm CV}}\\ 
\node{\mathbb B\big[\mathcal D(\G);\mathcal D'(\G)\big]} \node{\mathcal D'(\G\times\G)} \arrow{w,b}{{\sf Int}}
\end{diagram}
\end{equation*}
The vertical arrows to the right are justified by Lemma \ref{cadrat}. We leave the details to the reader.
\end{proof}

Techniques from \cite{MP} could be applied to define and study large Moyal algebras of vector-valued symbols corresponding to the spaces $\mathbb B\big[\mathcal D(\G)\big]$ and $\mathbb B\big[\mathcal D'(\G)\big]$ of operators.

%-------------------------------------------------------------------------------------------------------
\subsection{Compactness criteria}\label{fifitan}
%-------------------------------------------------------------------------------------------------------

The next result shows that compactness of sets, operators and families of operators in the Hilbert space $L^2(\G)$ can be characterised by localisation with pseudo-differential operators with symbols in $\mathscr D(\Gamma)$\,. 
We adapt the methods of proof from \cite{MPa}, whose framework cannot be applied directly.

\begin{Theorem}\label{alaman}
\begin{enumerate}
\item
A bounded subset $\Delta$ of $L^2(\G)$ is relatively compact if and only if for every $\epsilon>0$ there exists $a\in\mathscr D(\Gamma)$ such that
\begin{equation}\label{gascon}
\sup_{u\in\Delta}\p\!{\sf Op}^\tau(a)u-u\!\p_{L^2(\G)}\,\,\le\,\epsilon\,.
\end{equation}
\item
Let $\mathcal X$ be a Banach space. An element $T\in\mathbb B\big[\mathcal X,L^2(\G)\big]$ is a compact operator if and only if for every $\epsilon>0$ there exists $a\in\mathscr D(\Gamma)$ such that
\begin{equation*}\label{gascan}
\p\!{\sf Op}^\tau(a)T-T\!\p_{\mathbb B[\mathcal X,L^2(\G)]}\,\,\le\,\epsilon\,.
\end{equation*}
\item
Let $\mathscr L\subset \mathbb B[L^2(\G)]$ be a family of bounded operators. Then $\mathscr L$ is a relatively compact family of compact operators in $\,\mathbb K[L^2(\G)]$ if and only if for every $\epsilon>0$ there exists $a\in\mathscr D(\Gamma)$ such that
\begin{equation*}\label{gascun}
\sup_{T\in\mathscr L}\big(\p\!{\sf Op}^\tau(a)T-T\!\p_{\mathbb B[L^2(\G)]}+\p\!{\sf Op}^\tau(a)T^*-T^*\!\p_{\mathbb B[L^2(\G)]}\big)\,\le\,\epsilon\,.
\end{equation*}
\end{enumerate}
\end{Theorem}

\begin{proof}
1. If $\Delta$ is relatively compact, it is totally bounded. Thus, for every $\epsilon>0$\,, there is a finite set $F$
such that for each $u\in\Delta$ there exists $u'\in F$ with $\p\!u-u'\!\p_{L^2(\G)}\,\le\epsilon/4$\,. This finite subset generates a finite-dimensional subspace $\H_F\subset L^2(\G)$ with finite-rank corresponding projection $P_F$\,. Then for every $u\in\Delta$\,, recalling our choice  for $u'$ and the fact that $P_Fu'=u'$, one gets
\begin{equation*}\label{bourvi}
\begin{aligned}
\p\!P_Fu-u\!\p_{L^2(\G)}\,&\le\,\p\!P_Fu-P_Fu'\!\p_{L^2(\G)}+\p\!P_Fu'-u'\!\p_{L^2(\G)}+\p\!u'-u\!\p_{L^2(\G)}\\
&\le 2\p\!u-u'\!\p_{L^2(\G)}\,\le\epsilon/2\,.
\end{aligned}
\end{equation*}
Let now $M:=\sup_{u\in\Delta}\!\p\!u\!\p_{L^2(\G)}$\,; if we find $a\in\mathscr D(\Gamma)$ such that 
\begin{equation}\label{destul}
\p\!{\sf Op}^\tau(a)-P_F\!\p_{\mathbb B[ L^2(\G)]}\,\le\epsilon/2M\,
\end{equation} 
one writes for every $u\in\Delta$
$$
\p\!{\sf Op}^\tau(a)u-u\!\p_{L^2(\G)}\,\le\,\p\!{\sf Op}^\tau(a)u-P_Fu\!\p_{L^2(\G)}+\p\!P_F u-u\!\p_{L^2(\G)}\,\le\epsilon/2+\epsilon/2=\epsilon
$$
and the formula \eqref{gascon} is proved. 

Since $\mathcal D(\G\times\G)$ is dense in $L^2(\G\times\G)$\,, the subspace $\mathscr D(\Gamma)$ is dense in $\mathscr B^2(\Gamma)$\,.  Consequently, ${\sf Op}^\tau:\mathscr B^2(\Gamma)\rightarrow\mathbb B^2[L^2(\G)]$ being unitary, one even gets an improved version of \eqref{destul} with the operator norm replaced by the Hilbert-Schmidt norm. This finishes the ``only if" implication.

\medskip
We now prove the converse assertion. Fix $\epsilon>0$ and choose $a\in\mathscr D(\Gamma)$ such that
$$
\sup_{u\in\Delta}\!\p\!{\sf Op}^\tau(a)u-u\!\p_{L^2(\G)}\,\le\epsilon/2\,.
$$
Since $\mathscr D(\Gamma)\subset\mathscr B^2(\Gamma)$\,, the operator ${\sf Op}^\tau(a)$ is Hilbert-Schmidt, in particular compact. The set $\Delta$ is assumed bounded, thus the range ${\sf Op}^\tau(a)\Delta$ is totally bounded. Consequently, there is a finite set $E$
such that for each $u\in\Delta$ there exists $u''\in E$ satisfying $\p\!{\sf Op}^\tau(a)u-u''\!\p_{L^2(\G)}\,\le\epsilon/2$\,. Therefore
$$
\p\!
u-u''\!\p_{L^2(\G)}\,\le\,\p\!u-{\sf Op}^\tau(a)u\!\p_{L^2(\G)}+\p\!{\sf Op}^\tau(a)u-u''\!\p_{L^2(\G)}\,\le\,\epsilon/2+\epsilon/2=\epsilon\,,
$$
so the set $\Delta$ is totally bounded, thus relatively compact. 

\medskip
2. The operator $T$ is compact if and only if $\,\Delta:=T\big(\{w\in\mathcal X\mid \,\p\!w\!\p_{\mathcal X}\,\le 1\}\big)$ is relatively compact in the Hilbert space $L^2(\G)$\,. By 1, this happens exactly when for every $\epsilon>0$ there is a symbol $a$ in $\mathscr D(\Gamma)$ such that 
\begin{equation*}\label{papuc}
\p\!{\sf Op}^\tau(a)T-T\!\p_{\mathbb B[\mathcal X,L^2(\G)]}\,\,=\sup_{\p w\p_\mathcal X\le 1}\!\p\![{\sf Op}^\tau(a)-1](Tw)\!\p_{L^2(\G)} \,\le\epsilon\,.
\end{equation*}

\medskip
3. The set $\mathscr L$ is called {\it collectively compact} if $\,\,\bigcup_{T\in\mathscr L}T\big(\{u\in L^2(\G)\mid \,\p\!u\!\p_{L^2(\G)}\,\le 1\}\big)$ is relatively compact in $L^2(\G)$\,.
It is a rather deep fact \cite{An,Pa} that $\mathscr L$ is a relatively compact subset of $\,\mathbb K[L^2(\G)]$ with respect to the operator norm if and only if both $\mathscr L$ and $\,\mathscr L^*:=\{T^*\mid T\in\mathscr L\}$ are collectively compact. This and the point 2 lead to the desired conclusion.
\end{proof}

\begin{Remark}\label{militar}
{\rm In Subsection \ref{fourtinitin} we are going to introduce multiplication operators ${\sf Mult}(f)$ and left convolution operators ${\sf Conv}_L(g)$\,. Completing Theorem \ref{alaman}, one can easily prove that a bounded subset $\Delta$ of $L^2(\G)$ is relatively compact if and only if for every $\epsilon>0$ there exist $f,g\in\mathcal C_{\rm c}(\G)$ such that
\begin{equation*}\label{spital}
\sup_{u\in\Delta}\big(\p\!{\sf Mult}(f)u-u\!\p_{L^2(\G)}+\p\!{\sf Conv}_L(g)u-u\!\p_{L^2(\G)}\!\big)\le\epsilon\,.
\end{equation*}
Such type of results, in a much more general setting, have been proved in \cite{Fei}. They are not depending on the existence of a pseudo-differential calculus.
On the other hand, essentially by the same proof, we could assign $f,g$ to the Bruhat space $\mathcal D(\G)$\,, which is not covered by \cite{Fei}.
}
\end{Remark}

%-------------------------------------------------------------------------------------------------------
\section{Right and left quantizations}\label{fourtamin}
%-------------------------------------------------------------------------------------------------------

Our construction of the pseudo-differential calculus ${\sf Op}^\tau$ started from a concrete expression \eqref{frieda} for the Weyl system $W^\tau\equiv W^\tau_R$\,; we set for $x,y\in\G\,,\,\xi\in\wG$ and $\Theta\in L^2(\G,\H_\xi)$
\begin{equation}\label{frrieda}
\[W_R^\tau(\xi,x)\Theta\]\!(y):=\xi\!\[y(\tau x)^{-1}\]^*[\Theta(yx^{-1})]\,.
\end{equation}
The extra index $R$ hints to the fact that right translations are used in \eqref{frrieda}. Building on \eqref{frrieda} we constructed a ``right" pseudo-differential calculus ${\sf Op}^\tau\equiv{\sf Op}^\tau_R$ given on suitable symbols $a$ by
\begin{equation}\label{bilfrred}
\[{\sf Op}^\tau_R(a)u\]\!(x)=\int_\G\!\Big(\int_{\wG}\,{\rm Tr}_\xi\Big[\xi(y^{-1}x)a\big(x\tau(y^{-1}x)^{-1}\!,\xi\big)\Big]d\wm(\xi)\Big)u(y)d\m(y)\,.
\end{equation}
We recall that one gets integral operators, i.e. one can write
\begin{equation}\label{zzaruz}
{\sf Op}^\tau_R={\sf Int}\circ{\sf Ker}_R^\tau={\sf Int}\circ{\rm CV}_R^\tau\circ({\sf id}\otimes{\fscr F}^{-1})\,,
\end{equation}
in terms of a partial Fourier transformation and the change of variables 
\begin{equation}\label{religie}
{\rm cv}^\tau\equiv{\rm cv}^\tau_R:\G\times\G\rightarrow\G\times\G\,,\quad {\rm cv}_R^\tau(x,y):=\big(x\tau(y^{-1}x)^{-1}\!,y^{-1}x\big)\,.
\end{equation}
Besides \eqref{frrieda} there is (at least) another version of a Weyl system, involving translations to the left, given by
\begin{equation}\label{ffrieda}
\[W^\tau_L(\xi,x)\Theta\]\!(y):=\xi\!\[(\tau x)^{-1}y\]^*[\Theta(x^{-1}y)]\,.
\end{equation}
Using it, by arguments similar to those of Subsections \ref{tintin} and \ref{teletin}, one gets {\it a left pseudo-differential calculus}
\begin{equation}\label{bilfret}
\[{\sf Op}^\tau_L(a)u\]\!(x)=\int_\G\!\Big(\int_{\wG}\,{\rm Tr}_\xi\Big[\xi(xy^{-1})a\big(\tau(xy^{-1})^{-1}x,\xi\big)\Big]d\wm(\xi)\Big)u(y)d\m(y)\,,
\end{equation}
which can also be written as
\begin{equation}\label{zzaruzz}
{\sf Op}^\tau_L={\sf Int}\circ{\sf Ker}_L^\tau={\sf Int}\circ{\rm CV}_L^\tau\circ({\sf id}\otimes{\fscr F}^{-1})\,,
\end{equation}
in terms of a different change of variables 
\begin{equation*}\label{religgie}
{\rm cv}^\tau_L:\G\times\G\rightarrow\G\times\G\,,\quad {\rm cv}_L^\tau(x,y):=\big(\tau(xy^{-1})^{-1}x,xy^{-1}\big)\,.
\end{equation*}
Once again we get a unitary map $\,{\sf Op}^\tau_L:\mathscr B^2(\Gamma)\rightarrow\mathbb B^2\big[L^2(\G)\big]$ and all the results obtained above have, mutatis mutandis, analogous versions in the left calculus. In particular, ${\sf Op}^\tau_L$ also have extensions to distribution spaces connected to the Bruhat space, as in Section \ref{fergulin}.

\begin{Remark}\label{ordering}
{\rm The parameter $\tau$ is connected to ordering issues even in the standard case $\G=\R^n$. In general, another ordering problem comes from the non-commutativity of the group $\G$ and the non-commutativity of $\mathbb B(\H_\xi)$ for each irreducible representation $\xi:\G\to\mathbb B(\H_\xi)$\,. It is to this problem that we refer now. It is worth writing again the two quantizations for the simple case $\tau(x)={\sf e}$:
\begin{equation}\label{fred0d}
\[{\sf Op}_R(a)u\]\!(x)= \int_{\wG}{\rm Tr}_\xi\!\[\xi(x)a(x,\xi) \widehat{u}(\xi)\]\!d\wm(\xi)\,,
\end{equation}
following easily from \eqref{bilfrred}, and
\begin{equation}\label{fred0s}
\[{\sf Op}_L(a)u\]\!(x)= \int_{\wG}{\rm Tr}_\xi\!\[\xi(x)\widehat{u}(\xi)a(x,\xi)\]\!d\wm(\xi)\,,
\end{equation}
following from \eqref{bilfret}. The two expressions coincide if $\G$ is Abelian, since then each $\H_\xi$ will be $1$-dimensional. We will say more on this in Subsections \ref{firea} and \ref{fourtinitin}.
}
\end{Remark}

\begin{Remark}\label{equation}
{\rm We note that the choices of left or right quantizations as in \eqref{fred0d} and \eqref{fred0s} may lead to parallel equivalent (as in the case of compact Lie groups) or non-equivalent (as in the case of graded Lie groups) theories.
In the main body of the article we adopted the conventions leading to ${\sf Op}_R$\,, mainly to recover the compact \cite{RT} and the nilpotent \cite{FR} case (both already exposed in book form) as particular cases. But the left quantization is connected to the formalism of crossed product $C^*$-algebras, as will be seen subsequently, and this can be very useful for certain applications.

We also note that there may be no canonical way of calling the quantization ``left'' or ``right''. Thus, the terminology was opposite in \cite[Remark 10.4.13]{RT}, although it was natural in that context. In the present paper, the adopted terminology is related to the group actions in \eqref{frrieda} and \eqref{ffrieda}, respectively. It also seems natural from the formula
${\sf Op}_L^\tau\!={\sf Sch}_L^\tau\circ({\sf id}\otimes{\fscr F})^{-1}$ obtained later in \eqref{surprize}, where
${\sf Sch}_L^\tau=r\rtimes^\tau\! L\,$ appears in \eqref{SCHL} as the integrated form of the Schr\"odinger representation,
with the left-regular group action $L$ given by  $\[L(y)v\]\!(x)=v\!\(y^{-1}x\),$ and multiplication $r(f)v=f v\,.$
In any case, we refer to Section \ref{firea} for further interpretation of the quantization formulae in terms of the 
appearing Schr\"odinger representations.
}
\end{Remark}

\begin{Remark}\label{schnell}
{\rm The measurable map $\,\si_L:\G\rightarrow\G$ is called {\it a symmetry function with respect to the left quantization} if one has ${\sf Op}^{\si_L}_{\,L}(a)^*={\sf Op}^{\si_L}_{\,L}(a^\star)$ for every $a\in\mathscr B^2(\Gamma)$\,. As in Proposition \ref{simfunia}, one shows that ${\sf Op}_L^{\tau}(a)^*={\sf Op}_L^{\widehat\tau}(a^\star)$\,, for $\,\widehat\tau(x):=x\,\tau(x^{-1})$\,, so $\si_L$ must satisfy this time $\si_L(x)=x\,\si_L(x^{-1})$ (almost everywhere). An analog of Proposition \ref{zponentialaa} also holds. For central extensions, instead of \eqref{hunna}, one must set
\begin{equation*}\label{hunnak}
\si_L(h,n):=\Big(\si_{\sf H}(h),\si_{\sf N}\big[\varpi\big(h,\si_{\sf H}(h^{-1})\big)n\big]\Big)\,.
\end{equation*}
Note that, in the Lie exponential case, the function \eqref{clantaa} is a symmetry function simultaneously to the left and to the right, cf. \eqref{dutaa}.
}
\end{Remark}

Let $a$ be an element of $\mathscr B^2(\Gamma)$ and $\tau,\tau':\G\rightarrow\G$ two measurable functions. One has
\begin{equation*}\label{dandana}
{\sf Op}^{\tau'}_L(a)={\sf Int}\big[{\sf Ker}_{L,a}^{\tau'}\big]\quad{\rm and}\quad{\sf Op}^{\tau}_R(a)={\sf Int}\big[{\sf Ker}_{R,a}^{\tau}\big]\,.
\end{equation*}
It is easy to deduce from \eqref{zzaruz} and \eqref{zzaruzz} the connection between the left and the right kernel:
\begin{equation*}\label{leftright}
{\sf Ker}_{L,a}^{\tau'}={\rm CV}_L^{\tau'}\big({\rm CV}_R^\tau\big)^{-1}\big[{\sf Ker}_{R,a}^{\tau}\big]\,,
\end{equation*}
meaning that one has $\,{\sf Ker}_{L,a}^{\tau'}={\sf Ker}_{R,a}^\tau\circ{\rm cv}^{\tau,\tau'}_{R,L}$\,, where $\,{\rm cv}^{\tau,\tau'}_{R,L}:=\big({\rm cv}^\tau_{R}\big)^{-1}\circ{\rm cv}^{\tau'}_{L}\,$ is explicitly
\begin{equation*}\label{schimb}
{\rm cv}^{\tau,\tau'}_{R,L}(x,y)=\big(\tau'(xy^{-1})^{-1}x\tau(xy^{-1}),\tau'(xy^{-1})^{-1}x\tau(xy^{-1})yx^{-1}\big)\,.
\end{equation*}
This relation looks frightening, but particular cases are nicer. Setting $\tau(x)={\sf e}=\tau'(x)$ for instance, one gets $\,{\rm cv}^{{\sf e},{\sf e}}_{R,L}(x,y)=\big(x,xyx^{-1}\big)$\,, while $\tau={\sf id}=\tau'$ leads to $\,{\rm cv}^{{\sf id},{\rm id}}_{R,L}(x,y)=\big(yxy^{-1},y\big)$\,.

Investigating when ${\rm cv}^{\tau,\tau'}_{R,L}={\sf id}_{\G\times\G}$ holds (leading to ${\sf Op}^{\tau'}_L(a)={\sf Op}^{\tau}_R(a)$ for every $a$), one could be disappointed. It comes out quickly that $xyx^{-1}=y$ for all $x,y$ is a necessary condition, so the group $\G$ must be Abelian! Then $\tau=\tau'$ is the remaining condition.

%-------------------------------------------------------------------------------------------------------
\section{The $C^*$-algebraic formalism}\label{fourtin}
%-------------------------------------------------------------------------------------------------------

In this section we describe a general formalism in terms of $C^*$-algebras that becomes useful as a background
setting for pseudo-differential operators, in particular allowing working with operators with coefficients taking values in various
Abelian $C^*$-algebras. Especially, an interpretation in terms of crossed product $C^*$-algebras become handy making use of
$C^*$-dynamical systems and their covariant representations. 
We introduce an analogue of the Schr\"odinger representation and its appearance in $\tau$-quantizations.
Consequently, we investigate the role of multiplication and convolution operators in describing general
families of pseudo-differential operators. The formalism is then used to investigate
covariant families of pseudo-differential operators and establish several results concerning their spectra.

%-------------------------------------------------------------------------------------------------------
\subsection{Crossed product $C^*$-algebras}\label{fourtin}
%-------------------------------------------------------------------------------------------------------

We change now the point of view and place the pseudo-differential calculus in the setting of $C^*$-algebras generated by actions of our group $\G$ on suitable function algebras. For a full general treatment of $C^*$-dynamical systems and their crossed products we refer to \cite{Pe,Wi}.

\begin{Definition}\label{wilfred}
{\rm A $C^*$-dynamical system} is a triple $(\A,\theta,\G)$ where
\begin{itemize}
\item
$\G$ is a locally compact group with Haar measure $\m$\,, 
\item
$\,\A$ is a $C^*$-algebra,
\item
$\theta:\G\rightarrow{\rm Aut}(\A)$ is a strongly continuous action by automorphisms.
\end{itemize}
\end{Definition}

The third condition means that each $\theta_x:\A\rightarrow\A$ is a $C^*$-algebra isomorphism, the map $\G\ni x\mapsto\th_x(f)\in\A$ is continuous for every $f\in\A$ and one has $\th_x\circ\th_y=\th_{xy}$ for all $x,y\in\G$\,.

\begin{Definition}
\begin{enumerate}
\item
To a $C^*$-dynamical system $(\A,\theta,\G)$ we associate the Banach $^*$-algebra structure on $L^1(\G;\A)$ (the space of all Bochner integrable functions $\G\to\A$) given by
\begin{equation*}\label{ulrich}
\p\!\Phi\!\p_{(1)}\,:=\int_{\G}\!\p\!\Phi(x)\!\p_{\A}d\m(x) \,,
\end{equation*}
\begin{equation*}\label{diamond}
(\Phi\diamond \Psi)(x):=\int_{\G}\!\Phi(y)\,\theta_y\!\left[\Psi(y^{-1}x)\right]d\m(y)\,,
\end{equation*}
\begin{equation*}\label{willibrant}
\Phi^\diamond(x):=\theta_x\!\left[\Phi(x^{-1})^*\right]\,.
\end{equation*}
\item
Then {\rm the crossed product $C^*$-algebra} $\A\!\rtimes_\theta\!\G:={\rm Env}\!\left[L^1(\G;\A)\right]$ is the enveloping $C^*$-algebra of this Banach $^*$-algebra, i.e
its completion in the universal norm
\begin{equation*}\label{uninorm}
\p\!\Phi\!\p_{\rm univ}\,:=\,\sup_{\Pi}\p\!\Pi(\Phi)\!\p_{\mathbb B(\H)}\,,
\end{equation*}
where the supremum is taken over all the $^*$-representations $\,\Pi:L^1(\G,\A)\rightarrow\mathbb B(\H)$\,. 
\end{enumerate}
\end{Definition}

The Banach space $L^1(\G;\A)$ can be identified with the projective tensor product $\A\,\overline\otimes\,L^1(\G)$\,, and $\mathcal C_{\rm c}(\G;\A)$\,, the space of all $\A$-valued continuous compactly supported function on $\G$\,, is a dense $^*$-subalgebra of $L^1(\G,\A)$ and of $\A\!\rtimes_\theta\!\G$ (cf. \cite{Wi}).

\begin{Definition}\label{peter}
Let $(\A,\theta,\G)$ be a $C^*$-dynamical system.  {\rm A covariant representation} is a triple $(r,T,\H)$ where
\begin{itemize}
\item
$\H$ is a Hilbert space,
\item
$T:\G\rightarrow\mathbb U(\H)$ is a (strongly continuous) unitary representation,
\item
$r:\A\rightarrow\mathbb B(\H)$ is a $^*$-representation,
\item
$T(x)r(f)T(x)^*=r\left[\theta_x(f)\right]$\,, \,for every $f\in\A$ and $x\in\G$\,.
\end{itemize}
\end{Definition}

\noindent
It is known that there is a one-to-one correspondence between 
\begin{itemize}
\item
covariant representations of the $C^*$-dynamical system $(\A,\theta,\G)$\,,
\item
non-degenerate $^*$-representations of the crossed product $\A\!\rtimes_\theta\!\G$\,.
\end{itemize}

We only need the direct correspondence: {\it The integrated form} of the covariant representation $(r,T,\H)$ is uniquely defined by $r\!\rtimes\!T:L^1(\G;\A)\rightarrow\mathbb B(\H)$\,, with
\begin{equation*}\label{maximilian}
(r\!\rtimes\!T)(\Phi):=\!\int_\G\!r[\Phi(x)]T(x)d\m(x)\,;
\end{equation*}
then the unique continuous extension $\,r\!\rtimes\!T:\A\!\rtimes_\theta\!\G\rightarrow\mathbb B(\H)$ is justified by the universal property of the envelopping $C^*$-algebra.

\medskip
This is the formalism one usually encounters in the references treating crossed products \cite{Pe,Wi}; in terms of pseudo-differential operators this would only cover the case $\tau(\cdot)=\e$\,, i.e. the Kohn-Nirenberg type quantization. To treat the general case of a measurable map $\tau:\G\rightarrow\G$\,, one needs the modifications
\begin{equation}\label{diamont}
(\Phi\diamond^\tau\!\Psi)(x):=\int_{\G}\theta_{\tau(x)^{-1}\tau(y)}[\Phi(y)]\,\theta_{\tau (x)^{-1}y\tau(y^{-1}x)}\!\left[\Psi(y^{-1}x)\right]d\m(y)\,,
\end{equation}
\begin{equation}\label{friedrichs}
\Phi^{\diamond^\tau}\!(x):=\theta_{\tau(x)^{-1}x\tau(x^{-1})}\!\[\Phi(x^{-1})\]^*,
\end{equation}
in the $^*$-algebra structure of $L^1(\G;\A)$ and the next modification of the integrated form of a covariant representation $(r,T,\H)$ as
\begin{equation}\label{max}
(r\!\rtimes^\tau\!T)(\Phi):=\!\int_\G r\!\[\theta_{\tau (x)}(\Phi(x))\]T(x)\,d\m(x)\,.
\end{equation}
By taking enveloping $C^*$-algebras, one gets a family $\{\A\!\rtimes_\theta^{\tau}\!\G\}_\tau$ of $C^*$-algebras indexed by all the measurable mappings $\tau:\G\rightarrow\G$\,. 

\begin{Remark}\label{labush}
{\rm In fact all these $C^*$-algebras are isomorphic: $\A\!\rtimes_\theta^{\tau'}\!\!\G\overset{\nu_{\tau\tau'}}{\longrightarrow}\A\!\rtimes_\theta^{\tau}\!\G$ is an isomorphism, uniquely determined by its action on $L^1(\G;\A)$ defined as
\begin{equation*}\label{inge}
\left(\nu_{\tau\tau'}\Phi\right)(x):=\theta_{\tau(x)^{-1}\tau'(x)}[\Phi(x)]\,.
\end{equation*}
The family of isomorphisms satisfy
\begin{equation*}\label{daniel}
\nu_{\tau_1\tau_2}\circ\nu_{\tau_2\tau_3}=\nu_{\tau_1\tau_3}\,,\quad\ \nu_{\tau\tau'}^{-1}=\nu_{\tau'\tau}\,,
\end{equation*}
and the relation $r\!\rtimes^{\tau'}\!\!T=(r\!\rtimes^{\tau}\!T)\circ\nu_{\tau\tau'}$ is easy to check. An important ingredient is the fact that $\nu_{\tau\tau'}$ leaves the space $L^1(\G;\A)$ invariant (actually it is an isometry).
}
\end{Remark}

\begin{Remark}\label{mamsatu}
{\rm For further use, let us also examine $^*$-morphisms in the setting of crossed products (cf \cite{Wi}). Assume that $(\A,\theta,\G)$ and 
$(\A',\theta',\G)$ are $C^*$-dynamical systems and $\gamma:\A\rightarrow\A'$ is {\it an equivariant $^*$-morphism}, i.e. a $^*$-morphism satisfying 
\begin{equation}\label{varza}
\gamma\circ\th_x=\th'_x\circ\gamma\,,\quad\forall\,x\in\G\,.
\end{equation}
One defines 
\begin{equation}\label{varzar}
\gamma^\rtimes:L^1(\G;\A)\rightarrow L^1(\G;\A')\,,\quad\big[\gamma^\rtimes(\Phi)\big](x):=\gamma[\Phi(x)]\,.
\end{equation}
It is easy to check that $\gamma^\rtimes$ is a $^*$-morphism of the two Banach $^*$-algebra structures and thus it extends to a $^*$-morphism 
$\gamma^\rtimes:\A\!\rtimes_\theta^{\tau}\!\G\rightarrow\A'\!\rtimes_{\theta'}^{\tau}\!\G$\,. If $\gamma$ is injective, $\gamma^\rtimes$ is also injective.}
\end{Remark}

%-------------------------------------------------------------------------------------------------------
\subsection{The Schr\"odinger representation and $\tau$-quantizations}\label{firea}
%-------------------------------------------------------------------------------------------------------

It will be convenient to assume that $\A$ is a $C^*$-subalgebra of $\mathcal L\mathcal U\mathcal C_{{\rm b}}(\G)$ (bounded, left uniformly continuous functions on $\G$) invariant under left translations and that $\[\theta_y(f)\]\!(x):=f(y^{-1}x)$\,. The maximal choice $\A=\mathcal L\mathcal U\mathcal C_{{\rm b}}(\G)$ is very convenient, but studying ``pseudo-differential operators with coefficients of a certain type, modelled by $\A$", can sometimes be useful. Applications and extensions will appear elsewhere.

\medskip
For $\A$-valued functions $\Phi$ defined on $\G$ and for elements $x,q$ of the group, we are going to use notations as $[\Phi(x)](q)=\Phi(q,x)$\,, interpreting $\Phi$ as a function of two variables. The strange order of these variables is convenient to make the connection with previous sections. We can also understand the action as given by $\theta_y(\Phi(x))(q)=\Phi(y^{-1}q,x)\,.$

Thus on the dense subset $L^1(\G;\A)\subset\A\!\rtimes^\tau_\theta\!\G$ the composition law (\ref{diamont}) becomes more explicit
\begin{equation}\label{diamonda}
(\Phi\di^\tau\!\Psi)(q,x):=\int_{\G}\!\Phi\Big(\tau(y)^{-1}\tau(x)q,y\Big)\,\Psi\Big(\tau(y^{-1}x)^{-1}y^{-1}\tau(x)q,y^{-1}x\Big)d\m(y)\,,
\end{equation}
while the involution (\ref{friedrichs}) becomes 
\begin{equation}\label{willil}
\Phi^{\diamond^\tau}\!(q,x):=\overline{\Phi\big(\tau(x^{-1})^{-1}x^{-1}\tau(x)q,x^{-1}\big)}\,.
\end{equation}

\begin{Remark}\label{asteptai}
{\rm If $\G$ admits a symmetric quantization to the left and $\tau\equiv\si_L$ is a symmetry function to the left, as in Remark \ref{schnell}, the involution boils down to $\Phi^{\diamond^{\si_L}}\!(q,x):=\overline{\Phi\big(q,x^{-1}\big)}$\,.
}
\end{Remark}

In the situation described above, we always have a natural covariant representation $(r,L,\H)$\,, called {\it the Schr\"odinger representation}, given in $\mathcal H:=L^2(\G)$ by
\begin{equation}\label{razvan}
\[L(y)v\]\!(x):=v\!\(y^{-1}x\)\,,\ \quad r(f)v:=f v\,;
\end{equation}
thus $L(y)$ is the unitary left-translation by $y^{-1}$ in $L^2(\G)$ and $r(f)$ is just the operator of multiplication by the bounded function $f$\,.
The corresponding (modified) integrated form 
\begin{equation}\label{SCHL}
{\sf Sch}_L^\tau:=r\rtimes^\tau\! L\,, 
\end{equation}
computed as in \eqref{max}, is given for $\Phi\in L^1(\G;\A)$ and $v\in L^2(\G)$ by the formula
\begin{equation}\label{rada}
\begin{aligned}
\[{\sf Sch}_L^\tau(\Phi)v\]\!(x)=&\int_\G\Phi\big(\tau(z)^{-1}x,z\big)v(z^{-1}x)\,d\m(z)\\
=&\int_\G\Phi\!\left(\tau(xy^{-1})^{-1}x,xy^{-1}\right)\!v(y)\,d\m(y)\,. 
\end{aligned}
\end{equation}
The good surprise is that if we compose ${\sf Sch}^\tau_L$ with the inverse of the partial Fourier transform
one finds again, at least formally, the left pseudo-differential representation \eqref{bilfret} and \eqref{zzaruzz}:
\begin{equation}\label{surprize}
{\sf Op}_L^\tau\!={\sf Sch}_L^\tau\circ({\sf id}\otimes{\fscr F})^{-1}={\sf Int}\circ{\sf CV}^\tau_L\circ\big({\sf id}\otimes{\fscr F}^{-1}\big)\,.
\end{equation}
It is worth comparing this expression of ${\sf Sch}_L^\tau$ in \eqref{rada}
with \eqref{radar}.

To extend the meaning of \eqref{surprize} beyond the $L^2$-theory and to take full advantage of the $C^*$-algebraic formalism, one needs to be more careful. Recall that the Fourier transform defines an injective linear contraction ${\fscr F}:L^1(\G)\rightarrow \mathscr B(\wG)$\,. We already mentioned that $L^1(\G;\A)$ can be identified with the completed projective tensor product $\A\,\overline\otimes\,L^1(\G)$\,. Then, by \cite[Ex.\! 43.2]{Tr}, one gets a linear continuous injection 
\begin{equation*}\label{absenta}
{\sf id}_\A\,\overline\otimes\,{\fscr F}:\A\,\overline\otimes\,L^1(\G)\rightarrow\A\,\overline\otimes\,\mathscr B(\wG)
\end{equation*}
and we endow the image $\big({\sf id}_\A\,\overline\otimes\,{\fscr F}\big)\big[\A\,\overline\otimes\,L^1(\G)\big]$ with the Banach $^*$-algebra structure transported from $L^1(\G;\A)\cong \A\,\overline\otimes\,L^1(\G)\,$ through $\,{\sf id}_\A\,\overline\otimes\,{\fscr F}$\,. 

Let us denote by $\mathfrak C^\tau_{\!\A}$ the envelopping $C^*$-algebra $\A\!\rtimes_\theta^{\tau}\!\G$ of the Banach $^*$-algebra $L^1(\G;\A)$ (with the $\tau$-structure indicated above). Similarly, we denote by $\mathfrak B_{\!\A}^\tau$ the envelopping $C^*$-algebra of the Banach $^*$-algebra $({\sf id}_\A\,\overline\otimes\,{\fscr F})\big[\A\,\overline\otimes\,L^1(\G)\big]$\,. By the universal property of the enveloping functor, ${\sf id}_\A\,\overline\otimes\,{\fscr F}$ extends to an isomorphism $\,\mathfrak F_\A:\mathfrak C^\tau_{\!\A}\rightarrow\mathfrak B^\tau_{\!\A}$\,. 

Now $\,{\sf Op}^\tau_L:={\sf Sch}^\tau_L\circ\mathfrak F_\A^{-1}\,$ defines a $^*$-representation of the $C^*$-algebra $\mathfrak B^\tau_{\!\A}$ in the Hilbert space $L^2(\G)$\,, which is compatible with \eqref{surprize} when both expressions make sense. It seems pointless to use different notations for the two basically identical quantizations, one defined in the $C^*$-algebraical setting, starting from the Schr\"odinger representation, and the other introduced in Section \ref{fourtamin}, built on the family \eqref{ffrieda} of unitary operators. The difference is mearly a question of domains, but each of them can be still extended or restricted (at least for particular classes of groups $\G$\,), being constructed on the versatile operations ${\sf Int}\,, {\sf CV}^\tau, {\fscr F}^{-1}$\,.

\begin{Remark}\label{asha}
{\rm Starting from \eqref{diamonda} and \eqref{willil}, one also considers the transported composition 
\begin{equation*}\label{ute}
a\#^\tau b:=\mathfrak F_\A\big[(\mathfrak F_\A^{-1}a )\di^\tau\!(\mathfrak F_\A^{-1}b)\big]
\end{equation*} 
defined to satisfy ${\sf Op}^\tau_L(a\#^\tau b)={\sf Op}^\tau_L(a)\,{\sf Op}^\tau_L(b)$ as well as the involution 
\begin{equation*}\label{nicoleta}
a^{\#^\tau}\!:=\mathfrak F_\A\[(\mathfrak F_\A^{-1}a)^{\di^\tau}\]
\end{equation*} 
verifying ${\sf Op}^\tau_L(a^{\#^\tau})={\sf Op}^\tau_L(a)^*$. 
}
\end{Remark}

\begin{Remark}\label{froddo}
{\rm As a consequence of Remark \ref{labush} (see also Remark \ref{david}) and of the properties of envelopping $C^*$-algebras, there are isomorphisms $\mu_{\tau,\tau'}:\mathfrak B^\tau_{\!\A} \rightarrow \mathfrak B^{\tau'}_{\!\A}$ leaving $\big({\sf id}_\A\,\overline\otimes\,{\fscr F}\big)\big[L^1(\G;\A)\big]$ invariant and satisfying 
\begin{equation*}\label{labish}
{\sf Op}^\tau_L={\sf Op}^{\tau'}_L\circ\mu_{\tau,\tau'}\,,\quad \tau,\tau':\G\rightarrow\G\,.
\end{equation*}
Therefore, the $C^*$-subalgebra 
\begin{equation}\label{cincin}
\mathfrak D_{\!\A}:={\sf Op}_L^\tau\big(\mathfrak B^\tau_{\!\A}\big)={\sf Sch}_L^\tau\big(\mathfrak C^\tau_{\!\A}\big)\subset\mathbb B\big[L^2(\G)\big]
\end{equation} 
is $\tau$-independent. It could be called {\it the $C^*$-algebra of left global pseudo-differential operators with coefficients in $\A$ on the admissible group $\G$\,.}
}
\end{Remark}

\begin{Proposition}\label{danda}
The $C^*$-algebra $\mathfrak D_{\!\A}$ is isomorphic to the reduced crossed product $\big(\A\!\rtimes_\theta^{\tau}\!\G\big)_{\rm red}$\,. If $\,\G$ is amenable, the representation ${\sf Op}_L^\tau:\mathfrak B^\tau_{\!\A}\rightarrow\mathfrak D_{\!\A}\subset\mathbb B\big[L^2(\G)\big]$ is faithful.
\end{Proposition}

\begin{proof}
Of course, it is enough to work with one of the mappings $\tau:\G\rightarrow\G$\,, for instance for $\tau(x)={\sf e}$\,. Since $\mathfrak F_\A$ is an  isomorphism, it suffices to study the $^*$-representation ${\sf Sch}_L:=r\rtimes L:\mathfrak C_\A:=\A\!\rtimes_\th\!\G\rightarrow\mathbb B\big[L^2(\G)\big]$ and its range. For this we are going to recall the left regular $^*$-representation ${\sf Left}$ of the crossed product in the Hilbert space $\mathscr H:=L^2(\G\times\G)\cong L^2\big(\G;L^2(\G)\big)$ and show that ${\sf Sch}_L$ and ${\sf Left}$ are "unitarily equivalent up to multiplicity". The range of ${\sf Left}$ in $\mathbb B\big[L^2(\G\times\G)\big]$ is, by definition, {\it the reduced crossed product} $\big(\A\!\rtimes_\theta^{\tau}\!\G\big)_{\rm red}$\,. Since ${\sf Left}$ is injective if (and only if) $\G$ is amenable \cite{Wi}, this would finish our proof.

The $^*$-representation ${\sf Left}=r'\rtimes L'$ is the integrated form of the covariant representation $(r',L',\mathscr H)$ given by
\begin{equation*}\label{nutret}
[r'(f)\nu](q,x):=f(xq)\nu(q,x)\,,\quad x,q\in\G\,,\,f\in\A\,,\,\nu\in L^2(\G\times\G)\,,
\end{equation*}
\begin{equation*}\label{nutresc}
[L'(y)\nu](q,x):=\nu(q,y^{-1}x)\,,\quad x,y,q\in\G\,,\,\nu\in L^2(\G\times\G)\,.
\end{equation*}
Then the unitary operator $W:L^2(\G\times\G)\rightarrow L^2(\G\times\G)$ defined by
\begin{equation*}\label{nutrie}
(W\nu)(q,x):=\nu(q,xq)
\end{equation*}
satisfies for all $f\in\A$ and $y\in\G$
\begin{equation*}\label{nutrient}
W^*r'(f)W={\rm id}_{L^2(\G)}\otimes r(f)\,,\quad W^*L'(y)W={\rm id}_{L^2(\G)}\otimes L(y)\,,
\end{equation*}
which readily implies the unitary equivalence at the level of $^*$-representations
\begin{equation*}\label{abel}
W^*{\sf Left}(\Phi)W={\rm id}_{L^2(\G)}\otimes{\sf Sch}^\tau_L(\Phi)\,,\quad\forall\,\Phi\in\mathfrak C_\A\,,
\end{equation*}
and we are done. 
\end{proof}

\begin{Remark}\label{anterdiza}
{\rm Explicit descriptions of the $C^*$-algebras $\mathfrak B^\tau_{\!\A}$  are difficult to achieve. Even for $\G=\mathbb R^n$ some of the elements of $\mathfrak B^\tau_{\!\A}$ are not ordinary functions on $(\R^n)^*\times\R^n$\,.
}
\end{Remark}

\begin{Remark}\label{gertrudis}
{\rm Let us denote by $\mathcal C_0(\G)$ the $C^*$-algebra of all the continuous complex-valued functions on $\G$ which converge to $0$ at infinity (they are arbitrarily small outside sufficiently large compact subsets). It is well-known \cite{Wi} that the Schr\"odinger $^*$-representation sends $\mathcal C_0(\G)\!\rtimes_\theta\G$ onto $\mathbb K\!\[L^2(\G)\]\subset\mathbb B\!\[L^2(\G)\]$\,; it is an isomorphism between $\mathcal C_0(\G)\!\rtimes_\theta\!\G$ and $\mathbb K\!\[L^2(\G)\]$ if and only if $\G$ is amenable.
This provides classes of compact global pseudo-differential operators: 
\begin{equation}\label{barbara}
\mathbb K\!\[L^2(\G)\]={\sf Op}^\tau_L\big(\mathfrak B^\tau_{\mathcal C_0(\G)}\big)={\sf Sch}^\tau_L\big[\mathcal C_0(\G)\!\rtimes_\th\! \G\big]\,,
\end{equation}
giving a characterisation of compact operators.
}
\end{Remark}

\begin{Remark}\label{chiliman}
{\rm One sees that, in the process of construction of the crossed product $C^*$-algebra, taking the completion in the envelopping norm supplies a lot of interesting new elements. One has 
\begin{equation*}\label{frang}
L^2(\G\times\G)\cong\mathscr B^2(\wG\times\G)\cong\mathbb B^2\!\left[L^2(\G)\right]\subset\mathbb K\!\left[L^2(\G)\right]\cong\mathcal C_0(\G)\!\rtimes_\theta\!\G=\overline{L^1\big(\G;\mathcal C_0(\G)\big)}
\end{equation*}
(the last expression involves the closure in the enveloping norm), while $L^2(\G\times\G)$ and $L^1\big(\G;\mathcal C_0(\G)\big)$ are incomparable as soon as $\G$ is an infinite group. There are many Hilbert-Schmidt operators whose symbols are not partial Fourier transforms of elements from the class $L^1\big(\G;\mathcal C_0(\G)\big)$\,.
}
\end{Remark}

\begin{Remark}\label{kriemhilda}
{\rm Recall that $\mathbb K\big[L^2(\G)\big]={\sf Op}^\tau_L\big(\mathfrak B^\tau_{\mathcal C_0(\G)}\big)$ is an irreducible family of operators in $L^2(\G)$\,. So if $\mathfrak B$ is a space of symbols containing $({\sf id}\otimes{\fscr F})\big[\mathcal C_{\rm c}\big(\G;\mathcal C_0(\G)\big)\big]$
or $\mathscr B^2(\widehat{\Gamma})$\,, and if ${\sf Op}^\tau_L(b)$ makes sense for every $b\in\mathfrak B$\,, then ${\sf Op}^\tau_L\!\(\mathfrak B\)$ is irreducible. This happens, for instance, if $\mathfrak B=\mathfrak B^\tau_{\!\A}$ and $\mathcal C_0(\G)\subset\A$\,. In many other situations ${\sf Op}^\tau_L\!\(\mathfrak B\)$ could be reducible. Let us set $[R(z)u](x):=u(xz)$\,; a simple computation shows that $R(z){\sf Op}^\tau_L(b)R(z^{-1})={\sf Op}^\tau_L(b_z)$\,, where $b_z(x,\xi):=b(xz,\xi)$\,. Thus, if $b$ does not depend on the first variable, ${\sf Op}^\tau_L(b)$ commutes with the right translations and irreducibility is lost for $\A:=\mathbb C$\,. The same happens if $\A$ is defined through a periodicity (invariance) condition with respect to some nontrivial closed subgroup of $\G$\,.
}
\end{Remark}

%-------------------------------------------------------------------------------------------------------
\subsection{Multiplication and convolution operators}\label{fourtinitin}
%-------------------------------------------------------------------------------------------------------

We reconsider the Schr\"odinger covariant representation $\big(r,L,L^2(\G)\big)$ of the $C^*$-dynamical system $(\A,\th,\G)$ where, as before, $\A\subset\mathcal{LUC}_{\rm b}(\G)$ is invariant under left translations. 

\medskip
Let us define ${\sf Conv}_L:L^1(\G)\rightarrow\mathbb B\!\[L^2(\G)\]$ by the formula (interpreted in weak sense)
\begin{equation}\label{gertrudie}
{\sf Conv}_L(f):=\int_\G\!f(y)L(y)d\m(y)=\int_\G\!f(y)W^\tau_L({\mathfrak 1},y)(y)d\m(y)\,.
\end{equation}
Clearly ${\sf Conv}_L(f)$ is the operator of left-convolution with $f$\,: one has ${\sf Conv}_L(f)v=f\ast v$ for every $v\in L^2(\G)$\,. 
It is easy to check the right-invariance:
\begin{equation}\label{crokodil}
{\sf Conv}_L(f)\,R(x)=R(x)\,{\sf Conv}_L(f)\,,\quad\forall\,f\in L^1(\G)\,,\,x\in\G\,.
\end{equation}
The map ${\sf Conv}_L$ extends to a $C^*$-epimorphism from the group $C^*$-algebra $C^*(\G)=\mathbb C\!\rtimes_\th\!\G$ to the reduced group $C^*$-algebra $C^*_{\rm red}(\G)=\big(\mathbb C\!\rtimes_\th\!\G\big)_{\rm red}\subset\mathbb B\!\[L^2(\G)\]$\,, which is an isomorphism if and only if $\G$ is amenable \cite{Wi}.

\medskip
Of course, the family of right-convolution operators $\big\{u\mapsto{\sf Conv}_R(f)u:=u\ast f\mid f\in L^1(\G)\big\}$ is also available and it has similar properties. The analog of \eqref{gertrudie} is in this case 
\begin{equation*}\label{draguetti}
{\sf Conv}_R(f):=\int_\G\!\check f(y)R(y)d\m(y)\,,\quad{\rm with}\quad \check f(y):=f(y^{-1})\,.
\end{equation*}

Note the commutativity property
\begin{equation*}\label{crococ}
{\sf Conv}_L(f)\,{\sf Conv}_R(f')={\sf Conv}_R(f')\,{\sf Conv}_L(f)\,
\end{equation*} 
as well as the identities
\begin{equation*}\label{croco}
{\sf Conv}_L(f)\,{\sf Conv}_L(f')={\sf Conv}_L(f\ast f')\,,\quad{\sf Conv}_R(f)\,{\sf Conv}_R(f')={\sf Conv}_R(f'\ast f)\,.
\end{equation*}

\begin{Remark}\label{blue}
{\rm More generally, one can also define ${\sf Conv}_L(\mu)$ and ${\sf Conv}_R(\mu)$ for any bounded complex Radon measure $\mu\in M^1(\G)$\,.
Let us denote by $\mathfrak L(\G):=[L(\G)]''$ the left von Neumann algebra of $\G$ and by $\mathfrak R(\G):=[R(\G)]''$ the right von Neumann algebra of $\G$\,. One has ${\sf Conv}_L\big[M^1(\G)\big]\subset\mathfrak L(\G)$ and ${\sf Conv}_R\big[M^1(\G)\big]\subset\mathfrak R(\G)$\,.}
\end{Remark}

It is easy to check that 
\begin{equation*}\label{pis}
{\fscr F}\circ{\sf Conv}_L(f)\circ{\fscr F}^{-1}={\sf Dec}_R({\fscr F f})\quad{\rm and}\quad{\fscr F}\circ{\sf Conv}_R(f)\circ{\fscr F}^{-1}={\sf Dec}_L({\fscr F f})\,,
\end{equation*}
where 
\begin{equation*}\label{rac}
{\sf Dec}_R({\fscr F f})\,,{\sf Dec}_L({\fscr F f})\in\mathscr B(\wG):=\int^\oplus_{\wG}\mathbb B(\H_\xi)\,d\wm(\xi)\subset\mathbb B\big[\mathscr B^2(\wG)\big]
\end{equation*}
are decomposable (multiplication) operators defined for every $\varphi\in\mathscr B^2(\wG)$ by
\begin{equation*}\label{rak}
[{\sf Dec}_R({\fscr F f})\varphi](\xi):=\varphi(\xi)({\fscr F f})(\xi)\,,\quad[{\sf Dec}_L({\fscr F f})\varphi](\xi):=({\fscr F f})(\xi)\varphi(\xi)\,.
\end{equation*}

\medskip
We want to compute ${\sf Op}^\tau_L(g\otimes\beta)={\sf Sch}_L^\tau\big[g\otimes({\fscr F}^{-1}\beta)\big]$\,, where $g$ is some bounded uniformly continuous function on $\G$ and the inverse Fourier transform of $\beta$ belongs to $L^1(\G)$\,. Of course we set $(g\otimes\beta)(x,\xi):=g(x)\beta(\xi)\in\mathbb B(\H_\xi)$\,. One gets the formula 
\begin{equation*}\label{crocor}
\(\[{\sf Op}^\tau_L(g\otimes\beta)\]u\)(x)=\int_\G g\big[\tau(xy^{-1})^{-1}x\big]({\fscr F}^{-1}\beta)(xy^{-1})u(y)\,d\m(y)\,,
\end{equation*}
which is not very inspiring for general $\tau$\,. But using the notation ${\sf Mult}(g):= r(g)$ (a multiplication operator in $L^2(\G)$ given in \eqref{razvan})\,, one gets the particular cases, for ${\sf Op}_L\equiv{\sf Op}_L^{\e}$:
\begin{equation*}\label{freya}
\begin{aligned}
\big(\[{\sf Op}_L(g\otimes\beta)\]u\big)(x)&=g(x)\!\int_\G ({\fscr F}^{-1}\beta)(xy^{-1})u(y)\,d\m(y)\\
&=g(x)\!\int_\G ({\fscr F}^{-1}\beta)(z)u(z^{-1}x)\,d\m(z)\,,
\end{aligned}
\end{equation*}
which can be rewritten
\begin{equation}\label{crocodil}
{\sf Op}_L(g\otimes\beta)={\sf Mult}(g)\,{\sf Conv}_L({\fscr F}^{-1}\beta)\,,
\end{equation}
and
\begin{equation*}\label{freir}
\begin{aligned}
\(\[{\sf Op}_L^{{\rm id}}(g\otimes\beta)\]u\)\!(x)&=\int_\G({\fscr F}^{-1}\beta)(xy^{-1})g(y)u(y)\,d\m(y)\\
&=\int_\G({\fscr F}^{-1}\beta)(z)(gu)(z^{-1}x)\,d\m(y)\,,
\end{aligned}
\end{equation*}
i.e.
\begin{equation}\label{krokodil}
{\sf Op}^{{\rm id}}_L(g\otimes\beta)={\sf Conv}_L({\fscr F}^{-1}\beta)\,{\sf Mult}(g)\,.
\end{equation}
Thus in the quantization ${\sf Op}_L\equiv{\sf Op}_L^{\e}$ the operators of multiplication stay at the left and those of left-convolution to the right and vice versa for the quantization ${\sf Op}_L^{\rm id}$\,.

\begin{Remark}\label{strange}
{\rm In both \eqref{crocodil} and \eqref{krokodil} {\it left} convolution operators appear. 
But using the right quantization ${\sf Op}_R^\tau$ one gets
\begin{equation*}\label{crocoror}
\(\[{\sf Op}^\tau_R(g\otimes\beta)\]u\)(x)=\int_\G g\big[x\tau(y^{-1}x)^{-1}\big]({\fscr F}^{-1}\beta)(y^{-1}x)u(y)\,d\m(y)\,,
\end{equation*}
with particular cases
\begin{equation*}\label{crocotril}
{\sf Op}_R(g\otimes\beta)={\sf Mult}(g)\,{\sf Conv}_R({\fscr F}^{-1}\beta)\,,
\end{equation*}
\begin{equation*}\label{krokotril}
{\sf Op}^{{\rm id}}_R(g\otimes\beta)={\sf Conv}_R({\fscr F}^{-1}\beta)\,{\sf Mult}(g)\,,
\end{equation*}
and this should be compared with \eqref{crocodil} and \eqref{krokodil}.
}
\end{Remark}

\begin{Remark}\label{fredegonda}
{\rm As mentioned in Remark \ref{froddo}, the represented $C^*$-algebra 
\begin{equation*}\label{heidi}
\mathfrak D_\A:={\sf Sch}^\tau_L\(\A\rtimes_\th^\tau\G\)={\sf Op}_L^\tau\big[\mathfrak F_\A\!\(\A\rtimes_\th^\tau\G\)\big]\subset\mathbb B\big[L^2(\G)\big]
\end{equation*}
is independent of $\tau$\,. Actually it coincides with the closed vector space spanned by products of the form ${\sf Mult}(g)\,{\sf Conv}_L(f)$ (respectively ${\sf Conv}_L(f)\,{\sf Mult}(g)$) with $g\in\A$ and, say,  $f\in\(L^1\cap L^2\)\!(\G)$ (or even $f\in\mathcal D(\G)$)\,. So this closed vector space is automatically a $C^*$-algebra, although this is not clear at a first sight. The remote reason is the last axiom of Definition \ref{peter}. 
}
\end{Remark}

\begin{Remark}\label{fredegondar}
{\rm In \cite{DR1, DR2}, in the case of a compact Lie group $\G$\,, precise characterisations of the convolution operators belonging to the Schatten-von Neumann classes $\,\mathbb B^p\big[L^2(\G)\big]$ are given. The main result \cite[Th. 3.7]{DR1} holds, with the same proof, for arbitrary compact groups. 

When $\G$ is not compact, the single compact convolution operator is $0={\sf Op}_L^\tau(0)={\sf Conv}_L(0)={\sf Conv}_R(0)$\,. A way to see this is to recall Remark \ref{gertrudis} and to note that the constant function $g=1$ belongs to $\mathcal C_0(\G)$ if and only if $\G$ is compact. Another, more direct, argument is as follows: If $\G$ is not compact then $R(x)$ converges weakly to $0$ when $x\to\infty$\,. Multiplication to the left by a compact operator would improve this to strong convergence. But for $u\in L^2(\G)$ and a compact ${\sf Conv}_L(f)$ one has
$$
\parallel\!{\sf Conv}_L(f)u\!\parallel_{L^2(\G)}\,=\,\parallel\!R(x){\sf Conv}_L(f)u\!\parallel_{L^2(\G)}\,=\,\parallel\!{\sf Conv}_L(f)R(x)u\!\parallel_{L^2(\G)}\underset{x\to\infty}{\longrightarrow} 0
$$
and this implies ${\sf Conv}_L(f)=0$\,. Replacing $R(\cdot)$ by $L(\cdot)$\,, a similar argument shows that the single compact right convolution operator is the null operator.
}
\end{Remark}

%-------------------------------------------------------------------------------------------------------
\subsection{Covariant families of pseudo-differential operators}\label{midol}
%-------------------------------------------------------------------------------------------------------

An important ingredient in constructing the Schr\"odinger representation has been the fact that the $C^*$-algebra $\A$ was an algebra of (bounded, uniformly continuous) functions on $\G$\,. If $\A$ is just an Abelian $C^*$-algebra endowed with the action $\rho$ of our group $\G$\,, by Gelfand theory, it is connected to a topological dynamical system $(\Omega,\varrho,\G)$\,. The locally compact space $\Omega$ is the Gelfand spectrum of $\A$ and we have the $\G$-equivariant isomorphism $\A\cong\mathcal C_0(\Omega)$ if the action $\rho_x$ of $x\in\G$ on $\mathcal C_0(\Omega)$ is given just by composition with $\varrho_{x^{-1}}$\,.
In this section we are going to prove that to such a data one associates a covariant family of pseudo-differential calculi with operator-valued symbols. For convenient bundle sections $h$ defined on $\Omega\times\wG$ one gets families $\big\{{\sf Op}^\tau_{(\o)}(h)\mid \o\in\Omega\big\}$ of ``usual" left pseudo-differential operators (the index $L$ will be skiped). By covariance, modulo unitary equivalence, they are actually indexed by the orbits of the topological dynamical system, while their spectra are indexed by the quasi-orbits in $\Omega$\,. 

As before, the locally compact group $\G$ is supposed second countable, unimodular and type I, while $\tau:\G\rightarrow\G$ is measurable. 

Since the Schr\"odinger covariant representation \eqref{razvan} no longer makes sense as it stands, we are going to construct for each point $\o\in\O$ a covariant representation $\big(r_{(\o)},L,L^2(\G)\big)$ and then let the formalism act. One sets explicitly
\begin{equation}\label{avar}
\big[r_{(\o)}(f)u\big](x):=f\big[\varrho_{x}(\o)\big]u(x)\,,\quad f\in\mathcal C_0(\Omega)\,,\,u\in L^2(\G)\,,\,x\in\G\,,
\end{equation}
\begin{equation}\label{avarr}
\big[L(y)u\big](x):=u(y^{-1}x)\,,\quad u\in L^2(\G)\,,\,x,y\in\G\,.
\end{equation}
Proceeding as in Subsection \ref{firea}, one constructs the integrated form $\,{\sf Sch}^\tau_{(\o)}\!:=r_{(\o)}\!\rtimes\!L$ associated to the covariant representation $\big(r_{(\o)},L,L^2(\G)\big)$ and then sets 
\begin{equation}\label{tristan}
{\sf Op}^\tau_{(\o)}:={\sf Sch}^\tau_{(\o)}\circ\mathfrak F_{\mathcal C_0(\O)}^{-1}\,.
\end{equation}
As in Subsection \ref{firea}, the isomorphism $\mathfrak F_\O\equiv\mathfrak F_{\mathcal C_0(\O)}$ is the extension of the Banach $^*$-algebra monomorphism
\begin{equation*}\label{georgel}
{\sf id}_{\mathcal C_0(\O)}\,\overline\otimes\,{\fscr F}:L^1\big(\G;\mathcal C_0(\O)\big)\cong \mathcal C_0(\O)\,\overline\otimes\,L^1(\G)\rightarrow\mathcal C_0(\O)\,\overline\otimes\,\mathscr B(\G)
\end{equation*}
to the enveloping $C^*$-algebra $\mathcal C_0(\O)\!\rtimes^\tau_\rho\!\G$\,; the fact that $\A=\mathcal C_0(\O)$ is more general as before is not important.
Setting $\mathfrak B^\tau_\O\equiv\mathfrak B^\tau_{\mathcal C_0(\O)}$ for the enveloping $C^*$-algebra of $\,\,\big({\sf id}_{\mathcal C_0(\O)}\overline\otimes\,{\fscr F}\big) \big[L^1\big(\G;\mathcal C_0(\O)\big)\big]$ (with the transported structure), we have the isomorphism 
\begin{equation*}\label{mamacu}
\mathfrak F_\O:\mathcal C_0(\O)\!\rtimes^\tau_\rho\!\G\rightarrow\mathfrak B^\tau_\O\,.
\end{equation*} 
One gets for every section 
\begin{equation*}\label{nepasa}
\big\{h(\o,\xi)\in\mathbb B(\H_\xi)\mid \xi\in\wG\,,\,\o\in\O\big\}
\end{equation*} 
from $\big({\sf id}_{\mathcal C_0(\O)}\overline\otimes\,{\fscr F}\big)\big[L^1\big(\G;\mathcal C_0(\O)\big)\big]$ a family of operators 
\begin{equation*}\label{harto}
\Big\{{\sf Op}^\tau_{(\o)}(h)=\big(r_{(\o)}\!\rtimes\!L\big)\big(\mathfrak F_{\mathcal C_0(\O)}h\big)\in\mathbb B\big[L^2(\G)\big]\;\big\vert\;\o\in\O\Big\}
\end{equation*} 
given explicitly (but somewhat formally) by
\begin{equation}\label{fried}
\[{\sf Op}^\tau_{(\o)}(h)u\]\!(x)=\int_\G\Big(\int_{\wG}\!{\rm Tr}_\xi\!\[\xi(xy^{-1})h\big(\varrho_{\tau(xy^{-1})^{-1}x}(\o),\xi\big)\]\!d\wm(\xi)\Big)u(y)d\m(y)\,.
\end{equation}
More generally, the family $\big\{{\sf Op}^\tau_{(\o)}(h)\mid\o\in\O\big\}$ makes sense for $h\in\mathfrak B^\tau_\Omega$\,, but it is no longer clear when the symbol $h$ can still be interpreted as a function on $\O\times\wG$\,.

\begin{Proposition}\label{pecenieg}
Let $h\in\mathfrak B^\tau_\Omega$\,. If $\,\o,\o'$ belong to the same $\varrho$-orbit, then ${\sf Op}^\tau_{(\o)}(h)$ and ${\sf Op}^\tau_{(\o')}(h)$ are unitarily equivalent.
\end{Proposition}

\begin{proof}
The points $\,\o,\o'$ are on the same orbit if and only if there exists $z\in\G$ such that $\,\o'=\varrho_z(\o)$\,. In terms of the unitary right translation $[R(z)u](\cdot):=u(\cdot z)$\,, the operatorial covariance relation
\begin{equation}\label{cuman}
R(z){\sf Op}^\tau_{(\o)}(h)R(z)^*={\sf Op}^\tau_{(\varrho_z(\o))}(h)
\end{equation}
follows by an easy but formal calculation relying on \eqref{fried}. This can be upgraded to a rigorous justification by a density argument, but it is better to argue as follows:
Formula \eqref{cuman} for arbitrary $h\in\mathfrak B^\tau_\O$ is equivalent to 
\begin{equation*}\label{cruman}
R(z)\,{\sf Sch}^\tau_{(\o)}(\Phi)R(z)^*={\sf Sch}^\tau_{(\varrho_z(\o))}(\Phi)\,,\quad\forall\,\Phi\in\mathcal C_0(\O)\!\rtimes^\tau_\rho\!\G\,.
\end{equation*}
Since ${\sf Sch}^\tau_{(\o')}$ is the integrated form of the covariant representation $\big(r_{(\o')},L,L^2(\G)\big)$ indicated in \eqref{avar} and \eqref{avarr}, it is enough to prove
\begin{equation*}\label{priva}
R(z)L(x)R(z)^*=L(x)\,,\quad\forall\,x,z\in\G
\end{equation*}
and 
\begin{equation*}\label{adova}
R(z)\,r_{(\o)}(f)R(z)^*=r_{(\varrho_z(\o))}(f)\,,\quad\forall\,z\in\G\,,\,f\in\mathcal C_0(\O)\,.
\end{equation*}
The first one is trivial. The second one follows from
$$
\begin{aligned}
\big[R(z)\,r_{(\o)}(f)R(z)^*u\big](x)&=\big[r_{(\o)}(f)R(z^{-1})u\big](xz)\\
&=f\big[\varrho_{xz}(\o)\big]\big[R(z^{-1})u\big](xz)\\
&=f\big[\varrho_{x}\big(\varrho_z(\o)\big)\big]u(x)\\
&=\big[r_{(\varrho_z(\o))}(f)u\big](x)\,,
\end{aligned}
$$
completing the proof.
\end{proof}

\begin{Remark}\label{anglu}
{\rm In fact one has 
\begin{equation}\label{gepid}
{\sf Op}^\tau_{(\o)}(h)={\sf Op}^\tau_L\big(h_{(\o)}\big)\,,\quad{\rm with}\quad h_{(\o)}(x,\xi):=h\big(\varrho_{x}(\o),\xi\big)\,.
\end{equation}
This relation supplies another interpretation of the family $\big\{{\sf Op}^\tau_{(\o)}(h)\mid \o\in\Omega\big\}$\,. We can see it as being obtained by applying the left quantization procedure ${\sf Op}_L^\tau$ of the preceding sections to a family $\big\{h_{(\o)}\mid \o\in\Omega\big\}$ of symbols (classical observables) defined in $\G\times\wG$\,, associated through the action $\varrho$ to a single function $h$ on $\O\times\wG$\,.  Note that this family satisfies {\it the covariance condition} 
\begin{equation}\label{iut}
h_{\varrho_z(\o)}(x,\xi)=h_{(\o)}(xz,\xi)\,,\quad x,z\in\G\,,\,\xi\in\wG\,,\,\o\in\Omega\,.
\end{equation}
Using the reinterpretation \eqref{gepid}, the unitary equivalence \eqref{cuman} can be reformulated only in terms of the quantization ${\rm Op}^\tau_L$ as
\begin{equation*}\label{cumanian}
R(z)\,{\sf Op}^\tau_L\big(h_{(\o)}\big)R(z)^*={\sf Op}^\tau_L\big(h_{(\varrho_z(\o))}\big)\,,
\end{equation*}
which is easily proved directly using relation \eqref{iut} if $h$ is not too general.
}
\end{Remark}

We recall that {\it a quasi-orbit} for the action $\varrho$ is the closure of an orbit. If ${\sf O}_\o:=\varrho_\G(\o)$ is the orbit of the point $\o\in\O$\,, we denote by $\,{\sf Q}_\o:=\overline{{\sf O}}_\o=\overline{\varrho_\G(\o)}$\, {\it the quasi-orbit generated by} $\o$\,. As a preparation for Theorem \ref{peceneg}, we decompose the correspondance $\Phi\mapsto{\sf Sch}^\tau_\o(\Phi)$ into several parts. The starting point is the chain
\begin{equation*}\label{pair}
\mathcal C_0(\O)\overset{\gamma_\o}{\longrightarrow}\mathcal C_0({\sf Q}_{\o})\overset{\beta_\o}{\longrightarrow}\mathcal{LUC}_{\rm u}(\G)\,,
\end{equation*}
involving the restriction $^*$-morphism
$$
\gamma_{\o}:\mathcal C_0(\O)\to\mathcal C_0({\sf Q}_{\o})\,,\quad\gamma_{\o}(f):=f|_{{\sf Q}_{\o}}
$$
and the composition $^*$-morphism
$$
\beta_{\o}:\mathcal C_0({\sf Q}_{\o})\to\mathcal{LUC}_{\rm u}(\G)\,,\quad\big[\beta_{\o}(g)\big](x):=g\big[\varrho_x(\o)\big]\,.
$$
Note that $\beta_\o$ is injective, since $\varrho_\G(\o)$ is dense in ${\sf Q}_{\o}$\,.
Both these $^*$-morphisms are equivariant in the sense of Remark \ref{mamsatu} if on $\mathcal C_0(\O)$ one has the action $\rho$\,, on $\mathcal C_0({\sf Q}_{\o})$ its obvious restriction and on $\mathcal{LUC}_{\rm u}(\G)$ the action $\th$ of $\G$ by left translations, as in Subsection \ref{firea}. Correspondingly, one gets the chain
\begin{equation*}\label{bair}
\mathcal C_0(\O)^\rtimes\overset{\gamma^\rtimes_\o}{\longrightarrow}\mathcal C_0({\sf Q}_{\o})^\rtimes\overset{\beta^\rtimes_\o}{\longrightarrow}\mathcal{LUC}_{\rm u}(\G)^\rtimes\overset{{\sf Sch}_L^\tau}{\longrightarrow}\mathbb B\big[L^2(\G)\big]\,.
\end{equation*}
We indicated crossed products of the form $\mathcal B\!\rtimes^\tau\!\G$ by $\mathcal B^\rtimes$ (leaving the actions unnoticed) and the $^*$-morphism 
$\delta^\rtimes$ acting between crossed products is deduced canonically from an equivariant $^*$-morphism $\delta$ by the procedure described in Remark \ref{mamsatu}.
The arrow ${\sf Sch}^\tau_L$ is just the left Schr\"odinger representation of Subsection \ref{firea}. It is easy to check that
\begin{equation}\label{starc}
{\sf Sch}_L^\tau\circ\beta_\o^\rtimes\circ\gamma_\o^\rtimes={\sf Sch}^\tau_{(\o)}\,,
\end{equation}
which also leads to recapturing \eqref{gepid} after a partial Fourier transformation.

\medskip
Note that some points $\,\o\in{\sf Q}_{\o'}$ could generate strictly smaller quasi-orbits ${\sf Q}_{\o}\subset{\sf Q}_{\o'}$\,. On the other hand a quasi-orbit can be generated by points belonging to different orbits, so Proposition \ref{pecenieg} is not enough to prove the following result.

\begin{Theorem}\label{peceneg}
Suppose that the group $\G$ is admissible and amenable and that $h\in\mathfrak B^\tau_\O$\,.
\begin{enumerate}
\item
If $\,\o,\o'$ generate the same $\varrho$-quasi-orbit, then ${\sf Op}^\tau_{(\o)}(h)$ and ${\sf Op}^\tau_{(\o')}(h)$ have the same spectrum.
\item
If $\,(\O,\varrho,\G)$ is a minimal dynamical system then all the operators ${\sf Op}^\tau_{(\o)}(h)$ have the same spectrum.
\item
Assume that $\O$ is compact and metrizable and endowed with a Borel probability measure $\mu$ which is $\varrho$-invariant and ergodic. Then the topological support $\,{\rm supp}(\mu)$ is a $\varrho$-quasi-orbit and one has $\mu\big[\big\{\o\in\O\mid \overline{\sf O}_\o={\rm supp}(\mu)\big\}\big]=1$\,. The operators ${\sf Op}^\tau_{(\o)}(h)$ corresponding to points generating this quasi-orbit have all the same spectrum; in particular $\,{\sf sp}\big[{\sf Op}^\tau_{(\o)}(h)\big]$ is constant $\mu$-a.e. 
\end{enumerate}
\end{Theorem}

\begin{proof}
1. Let us denote by ${\sf Q}_\o:=\overline{\varrho_\G(\o)}$ the quasi-orbit generated by $\o$ and similarly for $\o'$. We show that if ${\sf Q}_\o\subset{\sf Q}_{\o'}$ then $\,{\sf sp}\big[{\sf Op}^\tau_{(\o)}(h)\big]\subset{\sf sp}\big[{\sf Op}^\tau_{(\o')}(h)\big]$ and this clearly implies the statement by changing the role of $\o$ and $\o'$.
Actually, by \eqref{tristan}, under the stated inclusion of quasi-orbits, one needs to show that 
$\,{\sf sp}\big[{\sf Sch}^\tau_{(\o)}(\Phi)\big]\subset{\sf sp}\big[{\sf Sch}^\tau_{(\o')}(\Phi)\big]$ for every element $\Phi$ of the crossed product 
$\mathcal C_0(\O)\!\rtimes^\tau_\rho\!\G\,.$

The basic idea, trivial consequence of the definitions, is the following: If $\,\Upsilon:\CC'\rightarrow\CC$ is a $^*$-morphism between two $C^*$-algebras and $g'$ is an element of $\CC'$, then $\,{\sf sp}[\Upsilon(g')\!\mid\!\CC]\subset{\sf sp}\big[g'\!\mid\!\CC'\big]$\,, and we have equality of spectra if $\Upsilon$ is injective. The notation indicates the $C^*$-algebra in which each spectrum is computed.

In our case, by \eqref{starc}, one can write 
\begin{equation*}\label{sfirc}
{\sf Sch}^\tau_{(\o)}(\Phi)=\big[{\sf Sch}_L^\tau\!\circ\beta^\rtimes_\o\big]\big[\gamma^\rtimes_\o(\Phi)\big]\quad{\rm and}\quad{\sf Sch}^\tau_{(\o')}(\Phi)=\big[{\sf Sch}^\tau_L\circ\beta^\rtimes_{\o'}\big]\big[\gamma^\rtimes_{\o'}(\Phi)\big]\,.
\end{equation*}
Since $\G$ is amenable ${\sf Sch}^\tau_L$ is injective and, as remarked before, $\beta^\rtimes_{\o}$ and $\beta^\rtimes_{\o}$ are always injective. Thus we are left with proving that 
\begin{equation}\label{sfurc}
{\sf sp}\big[\gamma^\rtimes_{\o}(\Phi)\!\mid\!\mathcal C_0({\sf Q}_{\o})^\rtimes\big]\subset{\sf sp}\big[\gamma^\rtimes_{\o'}(\Phi)\!\mid\!\mathcal C_0({\sf Q}_{\o})^\rtimes\big]\,,
\end{equation} 
assuming the inclusion ${\sf Q}_\o\subset{\sf Q}_{\o'}$ of quasi-orbits. We use now
$$
\Upsilon\equiv\gamma_{\o'\!,\o}^\rtimes:\mathcal C_0({\sf Q}_{\o'})\!\rtimes^\tau_\rho\!\G\to\mathcal C_0({\sf Q}_{\o})\!\rtimes^\tau_\rho\!\G\,,
$$
which is obtained by applying the functorial construction of Remark \ref{mamsatu} to the covariant restriction $^*$-morphism
$$
\gamma_{\o'\!,\o}:\mathcal C_0({\sf Q}_{\o'})\to\mathcal C_0({\sf Q}_{\o})\,,\quad\gamma_{\o'\!,\o}(f):=f|_{{\sf Q}_{\o}}\,.
$$
Note that $\,\gamma_{\o}=\gamma_{\o'\!,\o}\circ\gamma_{\o'}$ (succesive restrictions), which functorially implies $\,\gamma_{\o}^\rtimes=\gamma_{\o'\!,\o}^\rtimes\circ\gamma_{\o'}^\rtimes$\,. Then $\gamma_{\o}^\rtimes(\Phi)=\gamma_{\o'\!,\o}^\rtimes\big[\gamma_{\o'}^\rtimes(\Phi)\big]$ and \eqref{sfurc} and thus  the result follows.

\medskip
2. In a minimal dynamical system, by definition, all the orbits are dense. Thus any point generates the same single quasi-orbit ${\sf Q}=\Omega$ and one applies 1.

\medskip
3. The statement concerning the properties of $\,{\rm supp}(\mu)$ is contained in \cite[Lemma 3.1]{BHZ}. Then the spectral information follows applying 1. once again.
\end{proof}

The final point of Theorem \ref{peceneg} treats ``a random Hamiltonian of pseudo-differential type". Almost everywhere constancy of the spectrum in an ergodic random setting is a familiar property proved in many other situations \cite{CL,PF}. But note that a precise statement about the family of points giving the almost sure spectrum is available above.

%-------------------------------------------------------------------------------------------------------
\section{The case of nilpotent Lie groups}\label{firtanun}
%-------------------------------------------------------------------------------------------------------

We now give the application of the introduced construction to the case of nilpotent Lie groups.
Two previous main approaches seem to exist here. The first one uses the fact that, since the exponential mapping is a 
global diffeomorphism, one can introduce classes of symbols and the symbolic calculus on the group from
the one on its Lie algebra. This allows for operators on a nilpotent Lie group $\G$ to have 
scalar-valued symbols which can be interpreted as functions on the dual $\mathfrak g'$ of its Lie algebra.
Such approach becomes effective mostly for invariant operators on general nilpotent Lie groups
\cite{Melin, Glo1, Glo2}, see also \cite{Ta} for the case of the Heisenberg group. The second approach
applies also well to noninvariant operators on $\G$ and leads to operator-valued symbols, as developed
in \cite{FR, FR1}. This is also a special case (with $\tau(\cdot)={\sf e}$) of $\tau$-quantizations developed in this paper. 

We now extend both approaches to $\tau$-quantizations with the link between them provided in Remark \ref{link}.

%-------------------------------------------------------------------------------------------------------
\subsection{Some more Fourier transformations}\label{anun}
%-------------------------------------------------------------------------------------------------------

Let us suppose that $\G$ is a nilpotent Lie group with unit $\e$ and Haar measure ${\sf m}$\,; it will also be assumed connected and simply connected. Such a group is unimodular, second countable and type I, so it fits in our setting and all the previous constructions and statements hold.

Let $\mathfrak g$ be the Lie algebra of $\G$ and $\mathfrak g'$ its dual. If $Y\in\mathfrak g$ and $X'\in\mathfrak g'$ we set $\<Y\!\!\mid\!\!X'\>:=X'(Y)$\,. We shall develop further the theory in this nilpotent setting, but only to the extent the next two basic properties are used:
\begin{enumerate}
\item
the exponential map $\exp:\mathfrak g\rightarrow\G$ is a diffeomorphism, with inverse $\log:\G\rightarrow\mathfrak g$\,, \cite[Th. 1.2.1]{CG};
\item
under $\exp$ the Haar measure on $\G$ corresponds to the Haar measure $dX$ on $\g$ (normalised accordingly), cf \cite[Th. 1.2.10]{CG};
\end{enumerate}

It follows from the properties above that $L^p(\G)$ is isomorphic to $L^p(\mathfrak g)$\,. Actually, for each $p\in[1,\infty]$\,, one has a surjective linear isometry
\begin{equation*}\label{gupta}
 L^p(\G)\overset{{\rm Exp}}{\longrightarrow} L^p(\mathfrak g)\,,\quad {\rm Exp}(u):=u\circ\exp
 \end{equation*}
 with inverse
\begin{equation*}\label{dasgupta}
 L^p(\mathfrak g)\overset{{\rm Log}}{\longrightarrow} L^p(\G)\,,\quad {\rm Log}({\bf u}):={\bf u}\circ\log\,.
 \end{equation*} 

There is a unitary Fourier transformation $\,\mathcal F:L^2(\mathfrak g)\rightarrow L^2(\g')$ associated to the duality $\<\cdot\!\mid\!\cdot\>:\g\times\g'\rightarrow\mathbb R$\,. It is defined by
\begin{equation*}\label{clata}
(\mathcal F \mathbf u)(X'):=\int_{\g}e^{-i\<X\mid X'\>} \mathbf u(X)dX\,,
\end{equation*}
with inverse given (for a suitable normalization of $dX'$) by
\begin{equation*}\label{clatta}
(\mathcal F^{-1}\mathbf u')(X):=\int_{\g'}\!e^{i\<X\mid X'\>} \mathbf u'(X')dX'.
\end{equation*}

Now composing with the mappings ${\rm Exp}$ and ${\rm Log}$ one gets unitary Fourier transformations
\begin{equation*}\label{fericire}
\mathbf F:=\mathcal F\circ{\rm Exp}:L^2(\G)\rightarrow L^2(\g')\,,\quad \mathbf F^{-1}:={\rm Log}\circ\mathcal F^{-1}:L^2(\g')\rightarrow L^2(\G)\,,
\end{equation*}
the second one being the inverse of the first. They are defined essentially by
\begin{equation*}\label{qlata}
(\mathbf F u)(X')=\int_{\g}e^{-i\<X\mid X'\>}u(\exp X)dX=\int_\G e^{-i\<\log x\mid X'\>} u(x)d\m(x)\,,
\end{equation*}
\begin{equation*}\label{qlatta}
(\mathbf F^{-1}\mathbf u')(x)=\int_{\g'}\!e^{i\<\log x\mid X'\>}\mathbf u'(X')dX'.
\end{equation*}

Recalling Plancherel's Theorem for unimodular second countable type I groups, one gets finally a commuting diagram of unitary transformations
$$
\begin{diagram}
\node{L^2(\G)} \arrow{e,t}{{\rm Exp}} \arrow{s,l}{\fscr F}\arrow{se,r}{\mathbf F}\node{L^2(\g)}\arrow{s,r}{{\mathcal F}}\\ 
\node{\mathscr B^2(\wG)} \arrow{e,b}{{\mathfrak I}} \node{L^2(\g')}
\end{diagram}
$$
The lower horizontal arrow is defined as $\,\mathfrak I:=\mathbf F\circ{\fscr F}^{-1}=\mathcal F\circ{\rm Exp}\circ{\fscr F}^{-1}$
and is given explicitly on $\mathscr B^1(\wG)\cap\mathscr B^2(\wG)$ by
\begin{equation*}\label{feri}
(\mathfrak I \phi)(X')=\int_\G\int_{\wG}  e^{-i\<\log x\mid X'\>}{\rm Tr}_\xi\big[ \phi(\xi)\xi(x)\big]d\m(x)d\wm(\xi)\,.
\end{equation*}

\begin{Remark}\label{nula}
{\rm If $\G=\mathbb R^n$ it is possible, by suitable interpretations, to identify $\G\sim\wG$ with $\g$ and with $\g'$ (as vector spaces) and then the three Fourier transformations ${\fscr F}\,,\mathcal F$ and $\mathbf F$ will concide and $\mathfrak I$ will become the identity.}
\end{Remark}

%-------------------------------------------------------------------------------------------------------
\subsection{A quantization by scalar symbols on nilpotent Lie groups}\label{tanun}
%-------------------------------------------------------------------------------------------------------

To get pseudo-differential operators one could start, as in Subsection \ref{firea}, with a $C^*$-dynamical system $(\A,\th,\G)$ where $\A$ is a $C^*$-algebra of bounded left-uniformly continuous functions on $\G$ which is invariant under the action $\theta$ by left translations. We compose the left Schr\"odinger representation \eqref{rada} with the inverse of the
partial Fourier transform
\begin{equation*}\label{fanel}
{\sf id}\otimes\mathbf F:(L^1\cap L^2)(\G;\A)\rightarrow \A\otimes L^2(\g')\,,
\end{equation*}
finding the pseudo-differential representation 
\begin{equation}\label{dia}
\mathbf{Op}^\tau_L:={\sf Sch}_L^\tau\circ({\sf id}\otimes\mathbf F^{-1})={\sf Int}\circ{\sf CV}_L^\tau\!\circ\big({\sf id}\otimes{\mathbf F}^{-1}\big)
\end{equation}
which can afterwards be extended to the relevant enveloping $C^*$-algebra.
 One gets 
\begin{equation}\label{dulcisor}
\big[\mathbf{Op}^\tau_L({\bf s})u\big]\!(x)=\int_\G\!\int_{\g'}\!e^{i\<\log(xy^{-1})\mid X'\>}{\bf s}\big(\tau(xy^{-1})^{-1}x,X'\big)u(y)\,d\m(y)dX'\,,
\end{equation}
so $\mathbf{Op}^\tau_L({\bf s})$ is an integral operator with kernel $\,\mathbf{Ker}^\tau_{({\bf s})}:\G\times\G\rightarrow\mathbb C$ given by
\begin{equation*}\label{acrisor}
\mathbf{Ker}_{({\bf s})}^\tau(x,y)=\int_{\g'}\!e^{i\<\log(xy^{-1})\mid X'\>}{\bf s}\!\left(\tau(xy^{-1})^{-1}x,X'\right)dX'\,.
\end{equation*}
Examining this kernel, or using directly \eqref{dia}, one sees that \eqref{dulcisor} also defines a unitary mapping 
\begin{equation*}\label{abel}
\mathbf{Op}_L^\tau\!:L^2(\g'\times\G)\rightarrow\mathbb B^2\!\[L^2(\G)\].
\end{equation*}

Actually there is a Weyl system on which the construction of pseudo-differential operators with symbols ${\bf s}:\g'\times\G\rightarrow\mathbb C$ can be based:

\begin{Definition}\label{sigmund}
For $(x,X')\in\G\times\g'$ one defines a unitary operator $\mathbf W^\tau_L(x,X')$ in $L^2(\G)$ by
\begin{equation*}\label{friedar}
\begin{aligned}
\[\mathbf W^\tau_L(x,X')u\]\!(z):&=e^{i\<\log\[\tau(x)^{-1}z\]\mid X'\>}u(x^{-1}z)\\
&=e^{i\<\log\[\tau (x)^{-1}z\]\mid X'\>}[L(x)u](z)\,.
\end{aligned}
\end{equation*}
\end{Definition}

By direct computations, one shows the following

\begin{Lemma}\label{calcul}
Let us denote by $Q$ the operator of multiplication by the variable in $L^2(\G)$\,. For any pairs $(x,X'),(y,Y')\in\G\times\g'$ one has
\begin{equation*}\label{fulppe}
\mathbf W^\tau_L(x,X')\,\mathbf W^\tau_L(y,Y')=\Upsilon^\tau\big[(x,X'),(y,Y');Q\big]\,\mathbf W^\tau_L(xy,X'+Y')\,,
\end{equation*}
where $\Upsilon^\tau\big[(x,X'),(y,Y');Q\big]$ is the operator of multiplication by the function
\begin{equation*}\label{sireata}
\begin{aligned}
z\mapsto\Upsilon^\tau\big[(x,X'),(y,Y');z\big]&=\exp\Big\{i\Big[\<\,\log\!\big[\tau (x)^{-1}z\big]-\log\!\big[\tau(xy)^{-1}z\big]\mid X'\,\>-\\
&-\<\,\log\!\big[\tau(xy)^{-1}z\big]-\log\!\big[\tau(y^{-1})x^{-1}z\big]\mid Y'\,\>\Big]\Big\}\,.
\end{aligned}
\end{equation*}
\end{Lemma}

\begin{Remark}\label{strasnic}
{\rm The family $\mathcal C(\G;\T)$ of all continuous functions on $\G$ with values in the torus is a Polish group and the  mapping $\,\Upsilon:(\G\times\g')\times(\G\times\g')\rightarrow \mathcal C(\G;\T)$ can be seen as a $2$-cocycle. We are not going to pursue here the cohomological meaning and usefulness of these facts.}
\end{Remark}

In terms of the Weyl system $\big\{\mathbf W^\tau_L(x,X')\mid (x,X')\in\G\times\mathfrak g'\big\}$ one can write 
\begin{equation}\label{stephann}
\mathbf{Op}^\tau_L({\bf s}):=\int_{\G}\!\int_{\g'}\,\widetilde{\bf s}(X',x)\mathbf W^\tau_L(x,X')\,d\m(x)dX'\,;
\end{equation}
we used the notation $\,\widetilde{\bf s}:=({\bf F}\otimes{\bf F}^{-1}){\bf s}$\,. 
The technical details are similar but simpler than those in Subsection \ref{teletin} and are left to the reader.

\begin{Remark}\label{link}
{\rm One also considers the composition $\sharp_\tau$ defined to satisfy the equality
$\mathbf{Op}^\tau_L({\bf r}\,\sharp_\tau\,{\bf s})=\mathbf{Op}^\tau_L({\bf r})\,\mathbf{Op}^\tau_L({\bf s})\,,$ 
as well as the involution $^{\sharp_\tau}$ verifying $\mathbf{Op}^\tau_L({\bf s}^{\sharp_\tau})=\mathbf{Op}^\tau_L({\bf s})^*.$ 
Then $\big(L^2(\g'\times\G),\sharp_\tau,^{\sharp_\tau}\big)$ will be a $^*$-algebra. It is isomorphic to the $^*$-algebra $\big(\mathscr B^2(\wG\times\G),\#_\tau,^{\#_\tau}\big)$ defined in Subsection \ref{fraterin}.
Actually one has the following commutative diagram of isomorphisms:
$$
\begin{diagram}
\node{L^2(\G)\otimes L^2(\G)} \arrow{se,r}{\sf Sch_L^\tau}\arrow{e,t}{{\sf id}\otimes{\fscr F}} \arrow{s,l}{{\sf id}\otimes\mathbf F}\node{L^2(\G)\otimes\mathscr B^2(\wG)}\arrow{s,r}{{\sf Op}^\tau_L}\\ 
\node{L^2(\G)\otimes L^2(\mathfrak g')} \arrow{e,b}{{\mathbf{Op}^\tau_L}} \node{\mathbb B^2[L^2(\G)]}
\end{diagram}
$$
One justifies this diagram by comparing \eqref{dia} with \eqref{surprize}.
The conclusion of this diagram is that for simply connected nilpotent Lie groups the ``operator-valued pseudo-differential calculus" ${\sf Op}^\tau_L$ with symbols defined on $\G\times\wG$ can be obtained from the ``scalar-valued pseudo-differential calculus" $\mathbf{Op}^\tau_L$ (which provides a quantization on the cotangent bundle $\G\times\mathfrak g'\cong T'(\G)$) just by composing at the level of symbols with the isomorphism 
$({\sf id}\otimes{\mathbf F})\circ({\sf id}\otimes{\fscr F})^{-1}={\sf id}\otimes\big(\mathbf F\circ{\fscr F}^{-1}\big)$\,.}
\end{Remark}

\begin{Remark}\label{streasina}
{\rm In Prop. \ref{zponentialaa} we have shown that a connected simply connected nilpotent Lie group $\G$ admits a symmetric quantization, corresponding to the map $\tau=\si$ given by \eqref{clantaa} globally defined. With this choice one also has $\,{\bf s}^{\sharp_{\si}}\!(x,X')=\overline{{\bf s}(x,X')}\,$ for every $(x,X')\in\G\times\mathfrak g'$\,.
}
\end{Remark}

\begin{Remark}\label{nilp-right}
{\rm A right quantization ${\mathbf{Op}^\tau_R}$ with scalar symbols is also possible; for completeness, we list the 
main quantization formula
\begin{equation}\label{dia-r}
\mathbf{Op}^\tau_R:={\sf Sch}_R^\tau\circ({\sf id}\otimes\mathbf F^{-1})\equiv{\sf Int}\circ{\sf CV}_R^\tau\!\circ\big({\sf id}\otimes{\mathbf F}^{-1}\big)\,,
\end{equation}
where ${\sf CV}_R^\tau$ is the change of variables given by the composition with the mapping
\begin{equation}\label{religie-r}
{\rm cv}^\tau\equiv{\rm cv}^\tau_R:\G\times\G\rightarrow\G\times\G\,,\quad {\rm cv}_R^\tau(x,y):=\big(x\tau(y^{-1}x)^{-1}\!,y^{-1}x\big)\,,
\end{equation}
see also \eqref{religie}. 
Here, ${\sf Sch}_R^\tau:={\sf Int}\circ{\sf CV}_R^\tau\,$ also allows for an integrated interpretation similar to 
\eqref{SCHL}.
More explicitly, ${\mathbf{Op}^\tau_R}$ can be written as
\begin{equation}\label{dulcisor-r}
\big[\mathbf{Op}^\tau_R({\bf s})u\big]\!(x)=\int_\G\!\int_{\g'}\!e^{i\<\log(y^{-1}x)\mid X'\>}{\bf s}\big(x\tau(y^{-1}x)^{-1},X'\big)u(y)\,d\m(y)dX'\,,
\end{equation}
so $\mathbf{Op}^\tau_R({\bf s})$ is an integral operator with kernel $\,\mathbf{Ker}^\tau_{({\bf s}),R}:\G\times\G\rightarrow\mathbb C$ given by
\begin{equation*}\label{acrisor-r}
\mathbf{Ker}_{({\bf s}),R}^\tau(x,y)=\int_{\g'}\!e^{i\<\log(y^{-1}x)\mid X'\>}{\bf s}\!\left(x\tau(y^{-1}x)^{-1},X'\right)dX'\,.
\end{equation*}
Consequently, we have the commutative diagram of isomorphisms of $H^*$-algebras
$$
\begin{diagram}
\node{L^2(\G)\otimes L^2(\G)} \arrow{se,r}{\sf Sch_R^\tau}\arrow{e,t}{{\sf id}\otimes{\fscr F}} \arrow{s,l}{{\sf id}\otimes\mathbf F}\node{L^2(\G)\otimes\mathscr B^2(\wG)}\arrow{s,r}{{\sf Op}^\tau_R}\\ 
\node{L^2(\G)\otimes L^2(\mathfrak g')} \arrow{e,b}{{\mathbf{Op}^\tau_R}} \node{\mathbb B^2[L^2(\G)]}
\end{diagram}
$$
where all the listed mappings in this diagram are unitary, and where ${{\sf Op}^\tau}={{\sf Op}^\tau_R}$ is the $\tau$-quantization
formula \eqref{bilfred} that we have started with, and so
$$
{{\sf Op}^\tau}={{\sf Op}^\tau_R}=\mathbf{Op}^\tau_R\circ \[{\sf id}\otimes\big(\mathbf F\circ{\fscr F}^{-1}\big)\].
$$
}
\end{Remark}

%-------------------------------------------------------------------------------------------------------

Marius M\u antoiu:
  \endgraf
  Departamento de Matem\'aticas, Universidad de Chile
  \endgraf
  Casilla 653, Las Palmeras 3425, Nunoa, Santiago,
  \endgraf
  Chile
    \endgraf
  {\it E-mail address} {\rm mantoiu@uchile.cl}

\medskip

  Michael Ruzhansky:
  \endgraf
  Department of Mathematics, 
  Imperial College London
  \endgraf
  180 Queen's Gate, London SW7 2AZ,
  \endgraf
  United Kingdom
  \endgraf
  {\it E-mail address} {\rm m.ruzhansky@imperial.ac.uk}

\end{document}